\documentclass[11pt,reqno]{amsart}
\usepackage{color, amsmath,amssymb, amsfonts, amstext,amsthm, latexsym}
\usepackage{pdfsync}
\allowdisplaybreaks
\newcommand{\cF}{{\mathcal F}}
\newcommand{\RR}{{\mathbb R}}

\newcommand{\NN}{{\mathcal N}_{\alpha,p}}

\newcommand{\rde}{\mathbb{R}^d}

\newcommand{\beq}{\begin{equation}}
\newcommand{\beqn}{\begin{equation*}}
\newcommand{\eeq}{\end{equation}}
\newcommand{\eeqn}{\end{equation*}}

\numberwithin{equation}{section}
\newtheorem{theorem}{Theorem}[section]

\newtheorem{lemma}[theorem]{Lemma}
\newtheorem{remark}[theorem]{Remark}
\newtheorem{prop}[theorem]{Proposition}

\begin{document}
\title[Stochastic wave equations with superlinear coefficients]
{Global solutions to stochastic wave equations with superlinear coefficients }

\author[A. Millet]{ Annie Millet}
\address{SAMM, EA 4543,
Universit\'e Paris 1 Panth\'eon Sorbonne, 90 Rue de
Tolbiac, 75634 Paris Cedex, France {\it and} LPSM, UMR 8001, 
  Universit\'es Paris~6-Paris~7} 
\email{annie.millet@univ-paris1.fr} 

\author[M. Sanz-Sol\'e]{Marta Sanz-Sol\'e}
\address{Department of Mathematics and Informatics,
Barcelona Graduate School of Mathematics\\
University of Barcelona, Gran Via de les Corts Catalanes 585,
E-08007 Barcelona, Spain\\
}
\email{marta.sanz@ub.edu}



\begin{abstract}
 We prove existence and uniqueness of a random field solution $(u(t,x); (t,x)\in [0,T]\times \RR^d)$
 to a stochastic wave equation in dimensions $d=1,2,3$
 with diffusion and drift coefficients of the form $|z| \big( \ln_+(|z|) \big)^a$ for some $a>0$. 
 The proof relies on a sharp analysis of moment estimates
of time and space increments of the corresponding stochastic wave equation with globally Lipschitz coefficients. 
We give examples of  spatially correlated Gaussian driving noises where the results apply.
\end{abstract}
\maketitle

\section{Introduction}\label{si} 

In this paper, we study the  stochastic wave equation in spatial dimension $d\in\{1,2,3\}$, with a multiplicative noise $W$, 
\begin{align}
\label{wave-1}
&\frac{\partial^2}{{\partial t}^2} u(t,x)- \Delta_x u(t,x) = b(u(t,x)) + \sigma(u(t,x)) \dot{W}(t,x),\quad  (t,x)\in(0,T]\times \RR^d,\notag\\
&u(0,x)=u_0(x), \quad \frac{\partial}{\partial t} u(0,x)=v_0(x), \quad x\in \RR^d.							
\end{align}
The choice of $\dot W$ depends on the dimension $d$. First, we consider the case $d=1$ with space-time white noise. 
Then, we consider  the dimensions $d = 2,3$ with a noise white in time and coloured in space. 
The initial conditions $u_0$ and $v_0$ are real-valued functions. 
The coefficients $b, \sigma : \RR\to \RR$ are locally Lipschitz functions such that, for $|z|\to\infty$,
\beq
\label{super-coeff}
|b(z)| \le \theta_1 + \theta_2 |z|\left(\ln |z|\right)^\delta,\  |\sigma(z)| \le \sigma_1 + \sigma_2 |z|\left(\ln |z|\right)^a, 
\eeq
where $ \theta_i, \sigma_i\in\RR_+ $, $i=1,2$, $\theta_2, \sigma_2 > 0$, $\delta, a >0$. 

We are interested in studying conditions ensuring global existence of a random field solution to \eqref{wave-1}, 
that is, the existence of a stochastic process $\left(u(t,x), (t,x)\in[0,T]\times \RR^d\right)$ satisfying 
\begin{align}
\label{n3}
u(t,x) & = [G(t) \ast v_0](x)  + \frac{\partial}{\partial t} \big[ G(t)\ast u_0\big] (x) + \int_0^t ds\, [G(s)\ast b(u(t-s,\cdot))](x)\notag\\
& + \int_0^t \int_{\RR^d} G(t-s,x-y) \sigma(u(s,y)) W(ds,dy),\quad {\text{a.s.}}
\end{align}
either for all $(t,x)\in[0,T]\times \RR^d$ or $(t,x)\in [0,T]\times D$, with $D\subset \RR^d$ bounded.
In \eqref{n3}, $G(t)$, $t>0$, is the fundamental solution to the wave operator, the notation ``$\ast$" denotes the convolution in the space variable, and the stochastic integral is defined for example in \cite{Dal}.


It is a well-known phenomenon in PDEs that if the coefficients are superlinear, blow-up may occur 
(see for instance \cite{G-V-2002}, and \cite[Section X.13, p. 293]{R-S}). For parabolic SPDEs, 
there is an extensive literature devoted to the study of blow-up phenomena.
We refer the reader to \cite{DKZ} for a  sample of references. There are however less results on stochastic wave equations. To the best of our knowledge, existence or absence of blow-up has been studied so far in the setting of functional-valued solutions, rather than for random field solutions, and mostly but not only, with strong conditions on the space covariance (see e.g. \cite{Ch}, \cite{Ond1}, \cite{MiMo}).  A quite general setting is considered in \cite{M-QS-2012}, where, 
  existence (but not uniqueness) of functional-valued global solution is proved. 


Our research is motivated by \cite{DKZ}, on the parabolic SPDE 
\beq
\label{heat-1}
\frac{\partial}{{\partial t}} u(t,x)- \frac{\partial^2}{{\partial x}^2}u(t,x) = b(u(t,x)) + \sigma(u(t,x)) \dot{W}(t,x),\quad (t,x)\in (0,T]\times (0,1),
\eeq
$u(0,x)=u_0(x)$, $x\in[0,1]$, with vanishing Dirichlet boundary conditions 
and locally Lipschitz coefficients such that, as $|z|\to\infty$,
$|b(z)| = O(|z|(\ln|z|)), \ |\sigma(z)| = o\left(|z|(\ln|z|)^{1/4}\right)$.
One of the main results in \cite{DKZ} is the existence 
of a unique global random field solution to \eqref{heat-1} on $\mathcal{C}(\RR_+\times [0,1])$. This solution satisfies 
$\sup_{(t,x)\in[0,T]\times[0,1]}|u(t,x)| < \infty$, a.s., for any $T>0$.
If in equation \eqref{heat-1}, $\sigma$ is constant and $|b(z)|\ge |z|(\ln |z|)^{1+\varepsilon}$ when $|z|\to\infty$, 
with $\varepsilon$ arbitrarily close to zero, Bonder and Groisman \cite{B-G-2009} prove that blow-up occurs in finite time $t>0$. 
The results in \cite{DKZ} imply that this condition on $b$ is sharp. 

The main results of this work are Theorem \ref{wp-d=1} and Theorem \ref{wp-d}, relative to the two type of noises considered in the paper. 
Two scenarios are considered: (i) we restrict the spatial domain to a bounded set $D$; (ii) the initial values have compact support and $b(0)=\sigma(0)=0$. (see Section \ref{s-app} for details). Loosely formulated, we prove:
\smallskip

{\em If the initial conditions satisfy some H\"older continuity properties, the coefficients are such that \eqref{super-coeff} holds (see  condition {\bf(Cs)} in Section \ref{s2}), and $b$ dominates $\sigma$ (see  
conditions {\bf(C1)}, {\bf(Cd)} in Sections \ref{s2} and \ref{s3}, respectively), then a global random field solution to \eqref{n3} exists.}
\smallskip

Our approach follows the $L^\infty$-method of \cite{DKZ}, nevertheless, we do not use comparison theorems concerning monotony of coefficients  
since they do not hold for the wave equation.
The main task consists in establishing qualitative sharp upper bounds on $E\left(\sup_{(t,x)\in K}|u(t,x)|^p\right)$,
for some range of values of $p$, when the coefficients are {\em globally Lipschitz} and
$K$ is compact subset of $\RR_+\times \RR^d$.
 Such upper bounds depend on the value at the origin and the Lipschitz constants of the coefficients $b$ and $\sigma$ 
 (see Propositions \ref{Kol-1d} and \ref{Kol-d}, and the notation \eqref{c-L}). These bounds  
 are obtained from $L^p$-estimates of increments in time and in space of the process $(u(t,x))_{(t,x)}$ 
  (see Propositions \ref{s2-2-p1} and \ref{s3-3-p6}) via a version of Kolmogorov's theorem (\cite[Theorem A.3.1]{Da-SS-Book}).
Why is this important?  Existence of solutions to equations with locally Lipschitz coefficients is often proved by
transforming the coefficients into globally Lipschitz functions, using truncation. 
With a classical argument, involving an increasing sequence of stopping times $(\tau_N)_N$, 
if $\tau_N\uparrow \infty$ a.s., then existence of global solution follows. 
Let $u_N$ denote the random field solution to \eqref{n3} with truncated (by $N$) coefficients $b^N$, $\sigma^N$ (see \eqref{b_N}). In our case, a sufficient condition for $\tau_N\uparrow \infty$ to hold (a.s.) is
\beq
\label{s0.1}
E\Big(\sup_{(t,x)\in K}|u_N(t,x)|^p\Big) = o(N^p).
\eeq
We prove \eqref{s0.1} in the two scenarios described above, thereby deducing absence of blow-up. 

In Section \ref{s2}, we consider the case $d=1$ and space-time white noise. 
The simplicity of this case allows to better highlight the approach. In Section \ref{s3}, we deal with the case $d = 2,3$.  
Since we are interested in random field solutions, in contrast with the case $d = 1$, we cannot take a space-time white noise.
Instead, we consider a class of Gaussian noises white in time and coloured in space, for which a well developed stochastic integral theory exists 
 (see e.g. \cite{Dal}, \cite{Da-SS-Book}). In comparison with Section \ref{s2},  the arguments and computations are more difficult; 
 they are inspired by the approach to sample path regularity of the random field solution of \eqref{n3} for $d=3$ given in \cite{DSS-Memoirs} 
 and \cite{HHN}. 
Section \ref{s4} provides several examples of covariance densities where the results of the paper apply. Finally, in Section \ref{s-app}, we give the background on the two settings for the wave equation considered in the paper. 
\smallskip

We end this introduction with some remarks. 
Consider the case where $b$ and $\sigma$ are globally Lipschitz functions. From the first statement of Proposition \ref{s3-3-p6} (see 
\eqref{s3-3.72bis}), we deduce the existence of a version of the process $(u(t,x))_{(t,x)}$ with locally H\"older-continuous sample paths,
 jointly in $(t,x)$. 
 Thus, for the class of spatial covariances considered in Section \ref{s3}, this gives a unified approach to sample path regularity of the stochastic 
 wave equation when $d= 2,3$. Related results are in 
\cite{MiSS} for $d=2$, and  \cite{DSS-Memoirs}, \cite{HHN} for $d=3$.

Without much additional effort, the results of this paper can be extended to equation \eqref{n3} with coefficients $b(t, x; u(t,x))$ and $\sigma(t, x; u(t,x))$. 
 
\section{Preliminaries and notations}
\label{s1}

We recall that for $d=1,2$ and for any fixed $t>0$, the fundamental solution $G(t)$ to the partial differential operator
 $\frac{\partial^2}{{\partial t}^2} - \Delta_x$,  is a function. More precisely,
\beq
G(t,x) = \begin{cases}
\frac{1}{2} \, 1_{\{|x|<t\}}, & x\in \RR, \label{Green-1}\\
\frac{1}{2\pi}\frac{1}{\sqrt{t^2-|x|^2}}1_{\{|x|<t\}}, &  x\in \RR^2, 
\end{cases}
\eeq
while for $d=3$, 
\beq
\label{Green-3}
G(t,dx) =  \frac{1}{4\pi t}\ \sigma_t(dx), \ x\in \RR^3,
\eeq
where $\sigma_t(dx)$ denotes the uniform surface measure on the sphere centred at zero and with radius $t$, 
(see e.g. \cite[Ch. 5]{Folland}). 

 
 Recall  that, for any $d\ge 1$, the Fourier transform of $G(t, \cdot)$ is (see \cite[p. 49]{treves})
\beq
\label{i2}
\cF G(t,\cdot)(\zeta) = \int_{\RR^d} e^{-ix\cdot \zeta} G(x)\ dx = \frac{\sin(t|\zeta|)}{|\zeta|}.
\eeq

We will write $G(t, x-dy)$ to denote the translation by $-x$ of the measure $G(t, dy)$
 in the distribution sense (see e.g. \cite[p. 55]{Schwartz}).
 
We will often write \eqref{n3} in the compact form
\begin{equation}
\label{sol-1}	
u(t,x)=\sum_{i=0}^2 I_i(t,x), \qquad t\in [0,T], \; x\in \RR^d, 
\end{equation} 
where 
\begin{align}
\label{decom}
I_0(t,x) = & [G(t) \ast v_0](x)  + \frac{\partial}{\partial t} \big[ G(t)\ast u_0\big] (x),\notag \\
I_1(t,x) =  &\int_0^t \! ds\, [G(s)\ast b(u(t-s,\cdot))](x),\notag\\
 I_2(t,x) = & \int_0^t \int_{\RR^d} G(t-s,x-dy)  \sigma(u(s,y)) {W}(ds,dy). 
\end{align}
\smallskip



\noindent {\em Notations}
\smallskip

 As mentioned in the introduction, we assume first
 that the coefficients of \eqref{n3}, $b$ and $\sigma$, are globally Lipschitz continuous functions. 
Therefore, we have
\begin{equation} 			
\label{c-L}
|b(z)| \leq c(b)+L(b) |z|,\quad  |\sigma(z)| \leq c(\sigma) + L(\sigma) |z|, \ z\in \RR,
\end{equation}
with $c(b) = |b(0)|$, $c(\sigma)=|\sigma(0)|$ and $L(b)$, $L(\sigma)$, the Lipschitz constants of $b$ and $\sigma$, respectively.

Let $\Phi : \Omega\times [0,T] \times \RR^d \to \RR$ be a jointly measurable random field.    
For fixed $\alpha >0$, $p\in [2,\infty)$, we define the family of seminorms
\begin{equation}
\label{norm}
{\mathcal N}_{\alpha,p}(\Phi) : = \sup_{t\ge 0} \, \sup_{x\in\RR^d} \; e^{-\alpha t} \|\Phi(t,x)\|_{p},
\end{equation}
 where $\|\; \cdot \; \|_p$ denotes the norm in $L^p(\Omega)$.
 
For $\phi:\RR \to \RR$, set $\|\phi\|_\infty = \sup_{x\in {\RR}} |\phi(x)|$ and, for $R\ge 0$, $\Vert \phi\Vert_{\infty,R} = \sup_{|x|\le R}|\phi(x)|$. 
For $\gamma\in(0,1)$, we define
\beq
\label{n1}
\Vert\phi\Vert_{\gamma} = \sup_{x\ne y}\frac{|\phi(x)-\phi(y)|}{|x-y|^{\gamma}}.
\eeq
 
Except if specified otherwise, $C, \bar C, \tilde C, c, \ldots$ are positive and finite constants that may change  throughout the paper,
 and $C(a), \bar C(a)$, etc.,  denote positive finite constants depending on the parameter $a$. 
 
\section{The stochastic wave equation in dimension one} \label{s2}  
In this section, we consider the stochastic wave equation \eqref{n3} for $d=1$, with a space-time white noise $W$ and coefficients satisfying the superlinear growth condition \eqref{super-coeff}.
 The study goes through several steps developed in the next subsections.
 \subsection{Qualitative moment estimates} \label{s2-1}
We assume that the coefficients of \eqref{n3}, $b$ and $\sigma$, are globally Lipschitz continuous functions therefore satisfying \eqref{c-L}. We also suppose that $L(b)$ and $L(\sigma)$ are strictly positive. The goal is to obtain upper bounds on $\sup_{x\in\RR}\Vert u(t,x)\Vert_p$ in terms of the constants $c(b)$, $c(\sigma)$, $L(b)$, $L(\sigma)$ for some range of values of $p$. This will be done using the approach of \cite[Chapter 5]{Kho} for the stochastic heat equation (see also \cite{DKZ}). 
\smallskip

By the definition of $I_0(t,x)$ and $G$ given in \eqref{decom} and \eqref{Green-1}, respectively, we have
\begin{align}
 \label{IC-1d}
I_0(t,x) = & 
\frac{1}{2} \int_{x-t}^{x+t}  v_0(y) dy + \frac{1}{2} \big( u_0(x-t)+u_0(x+t)\big).
\end{align} 
From this equality, we deduce
\begin{equation}
 \label{IC-1}
\sup_{x\in \RR} |I_0(t,x)| \leq t \|v_0\|_{\infty} + \|u_0\|_\infty.
\end{equation}
Clearly, if $u_0$, $v_0$ are bounded  functions then $\sup_{x\in \RR} |I_0(t,x)|<\infty$.

\begin{prop}
\label{Proposition II.3}  
(\cite[Proposition II.3]{CaNu})
Assume that the function $(t,x)\mapsto I_0(t,x)$ is continuous and $\sup_{x\in \RR} |I_0(t,x)|<\infty$. Suppose that $b$ and $\sigma$ are globally Lipschitz continuous functions.
Then \eqref{n3} has a unique random field solution 
 $\big(u(t,x); (t,x)\in[0,T]\times \RR\big)$. This solution satisfies
  \beqn
  \sup_{(t,x)\in[0,T]\times\RR}\Vert u(t,x)\Vert_p <\infty,\quad {\text{for any}} \quad p\in[1,\infty).
 \eeqn
\end{prop}

 In the proof of the next proposition, the following facts will be used:
 \beq
 \label{*}
 \sup_{t\ge 0}(t^k e^{-\alpha t})=k^k (e\alpha)^{-k}, \  k\in\mathbb{N,\quad }\sup_{t\ge 0}\int_0^t s e^{-\alpha s}\ ds = \alpha^{-2}, 
 \quad \alpha>0.
 \eeq
\begin{prop}
\label{lem1.1}
Let $u_0$ and $v_0$ be Borel functions satisfying $\|u_0\|_{\infty} + \|v_0\|_\infty<\infty$. 
 Suppose that $L(b) \geq 8 L(\sigma)^2$. 
Then, there exists a universal constant $C>0$ such that, for any $p\in\left[2,\frac{L(b)}{4 L(\sigma)^2}\right]$,
\beq
\label{NN-1}
{\mathcal N}_{2\sqrt{L(b)},p}(u)\leq  {\mathcal T}_0 +  C\Big[ \,  \frac{c(b)}{L(b)} +  \frac{c(\sigma)}{L(\sigma)}\, \Big],
\eeq
where 
\beq
\label{t0}
{\mathcal T}_0 = \frac{e^{-1} \|v_0\|_\infty}{\sqrt{L(b)} }+ 2 \|u_0\|_\infty.
\eeq 
Thus,
\beq
\label{N-1}
\sup_{x\in\RR} E(|u(t,x)|^p) \leq  e^{2pt\sqrt{L(b)}} \left\{ {\mathcal T}_0 
+ C \left[\frac{c(b)}{ L(b)} + \frac{ c(\sigma) }{L(\sigma)}\right] \right\}^p ,  \; t\in [0,T]. 		
\eeq
\end{prop} 
\begin{proof}
Fix $\alpha >0$ and $p\in [2,+\infty)$. Using  \eqref{IC-1} and  \eqref{*},
we obtain
\beq
\label{norm-0}
 \NN(I_0) \le \frac{e^{-1}}{\alpha}\|v_0\|_{\infty} +  \|u_0\|_{\infty}.
\eeq
 
 Applying Minkowski's inequality, and then \eqref{c-L}, we have
 \begin{align*}
 \Vert I_1(t,x)\Vert_p & \le \int_0^t ds \int_{\RR} dy \ G(t-s,x-y) \Vert b(u(s,y))\Vert_p\\
 & \le \int_0^t ds \int_{\RR} dy \ G(t-s,x-y)\left[c(b) + L(b) \Vert u(s,y)\Vert_p\right].
 \end{align*}
Since $\int_{\RR} G(t,x) dx =t$, using \eqref{*}  we deduce 
\begin{align}
\label{T1-1}
 \NN(I_1) 
 & \le c(b) \sup_{t\ge 0} \left(\frac{t^2}{2} e^{-\alpha t}\right) + L(b) \, \NN(u) \sup_{t\ge 0} \int_0^t (s) e^{-\alpha (s)} ds\nonumber\\
& \le \frac{2 e^{-2} }{\alpha^2} c(b) + \frac{1}{\alpha^2} L(b) \, \NN(u)
\le \frac{1}{\alpha^2}\left[c(b) + L(b)\NN(u)\right]. 
\end{align}


Applying first the version of Burkholder-Davies-Gundy's inequality given in \cite[Theorem B1, p. 97]{Kho}, then Minkowski's inequality and \eqref{c-L}, we obtain
 \begin{align*}
 \Vert I_2(t,x)\Vert_p^2 
& \le  4p\ \Big\|  \int_0^t \int_{\RR} G^2(t-s,x-y) \sigma^2(u(s,y)) ds dy\Big\|_{\frac{p}{2}}\\
& \le 4p\   \int_0^t \int_{\RR} G^2(t-s,x-y) \Vert \sigma^2(u(s,y))\Vert_{\frac{p}{2}} ds dy\\
& \le 8p\left\{\int_0^t ds \int_{\RR} dy\ G^2(t-s,x-y)\left[c(\sigma)^2 + L(\sigma)^2\Vert u(s,y))\Vert^2_{p}\right]\right\}.
\end{align*}
Since $G^2(t,x) = \frac{1}{2} G(t,x)$, using \eqref{*} we have
\begin{align}
\label{T2-1}
 \NN(I_2) 
&\le \sqrt{2p} \; c(\sigma)\sup_{t\ge 0}\left(t e^{-\alpha t}\right) + \sqrt{8p}\;  L(\sigma) \NN(u)\notag\\
&\qquad \times \left(\int_0^t ds \int_{\RR} dy\ G^2(t-s,x-y) e^{-2\alpha(t-s)}\right)^{1/2}\notag\\
&\le   \sqrt{2p} \;  \frac{e^{-1}}{\alpha} c(\sigma) + \sqrt{8p} \; L(\sigma) \NN(u)
\Big( \int_0^t \frac{1}{2} s e^{-2\alpha s} ds \Big)^{\frac{1}{2}}   \notag \\
& \leq \frac{\sqrt p}{\alpha}\, \left[c(\sigma) + L(\sigma) \NN(u)\right]. 
\end{align}
The inequalities \eqref{norm-0}, \eqref{T1-1} and \eqref{T2-1} imply
\begin{align}
\label{N-1alpha}
\NN(u) \leq &  \frac{e^{-1}}{\alpha} \|v_0\|_{\infty}+ \|u_0\|_{\infty} + \frac{c(b)}{\alpha^2}  + \frac{\sqrt{p}}{\alpha} c(\sigma)
+ 2 \max \Big( \frac{L(b)}{\alpha^2} , \frac{\sqrt{p} L(\sigma)}{\alpha}\Big) \NN(u).
\end{align}
Fix $\alpha^2 = 4 L(b)$; since $L(b) \geq  8 L(\sigma)^2$, the interval  $\left[2,\frac{L(b)}{4L(\sigma)^2}\right]$ is nonempty. 
Since for any $p$ in this interval we have 
$\sqrt{p} L(\sigma) \leq \frac{\sqrt{L(b)}}{2}=\frac{\alpha}{4}$, the choice of $\alpha$ implies
 $ \max \Big( \frac{L(b)}{\alpha^2} , \frac{\sqrt{p} L(\sigma)}{\alpha}\Big) = \frac{1}{4}$,
and $\frac{\sqrt p}{\alpha}\le \frac{1}{4 L(\sigma)}$.
Hence, from \eqref{N-1alpha} we deduce \eqref{NN-1}. The estimate \eqref{N-1} is an immediate consequence of the definition of $\NN(u)$ for $\alpha= 2\sqrt{L(b)}$.
\end{proof}


\subsection{Uniform bounds on moments}
 \label{s2-2}
 
In this section, we still assume that the coefficients of \eqref{n3} are globally Lipschitz continuous functions.  We prove an upper bound for
\beq
\label{s2-2.1}
E\Big( \sup_{t\in [0,T]} \sup_{|x|\le R} |u(t,x)|^p\Big),
\eeq
for any $R>0$, and for 
specific values of $p$ that depend on the initial values $u_0$, $v_0$, and the constants $c(b)$, $c(\sigma)$, $L(b)$, $L(\sigma)$. 
 This will be a consequence of the following proposition.

\begin{prop}
\label{s2-2-p1}
 Let   
$u_0$ be locally H\"older continuous with exponent $\gamma_1\in (0,1]$, and  $v_0$ be continuous.
Set $\gamma = \gamma_1 \wedge \frac{1}{2}$, and fix $T, R\ge 0$. Then, for any $p\in[2,\infty)$,
 there exists a positive constant $C(p,T,R)$ such that, for any $t, \bar t\in[0,T]$, $x,\bar x \in [-R,R]$ and $\alpha>0$,
\beq
\label{increments-1d}
\frac{\Vert u(t,x)-u(\bar t,\bar x)\Vert_p}{(|t-\bar t|+|x-\bar x|)^\gamma}
\le C(p,T,R) \left[{\mathcal M}_1 + {\mathcal M}_2 
e^{\alpha T}\NN(u) \right], 
\eeq
where
\begin{align}
\label{emes}
{\mathcal M}_1 &= \|u_0\|_{\gamma_1}+ \|v_0\|_{\infty,R+T}+ c(b) + \sqrt p\  c(\sigma),  
\quad {\mathcal M}_2 =  L(b)+\sqrt{p}\ L(\sigma). 
\end{align} 
Moreover, if $L(b)\ge 8L(\sigma)^2$ then for any $p\in\left[2, \frac{L(b)}{4L(\sigma)^2}\right]$, 
\beq
\label{increments-1d-n}
\frac{\Vert u(t,x)-u(\bar t,\bar x)\Vert_p}{(|t-\bar t|+|x-\bar x|)^\gamma}
\le C(p,T,R) \left[{\mathcal M}_1 + {\mathcal M}_2 
e^{2\sqrt{L(b)} T}\left({\mathcal T}_0 + \frac{c(b)}{L(b)} +  \frac{c(\sigma)}{L(\sigma)}\right)\right],
\eeq
with ${\mathcal T}_0$ given in \eqref{t0}.
\end{prop}
\begin{proof} 
The function $V_0(z)=\int_0^z v_0(y)\ dy$ is continuously differentiable; hence,
\beqn
\Big\vert\int_{\bar x - \bar t}^{\bar x + \bar t} v_0(y)\ dy - \int_{x-t}^{x+t} v_0(y)\ dy\Big\vert \le 2\ \Vert v_0\Vert_{\infty,R+T}
\left(|x-\bar x| + |t - \bar t|\right).
\eeqn
Consequently, using the expression \eqref{IC-1d} and the $\gamma_1$- H\"older continuity of $u_0$,  we obtain
\beq
\label{diff-I0-1d}
\vert I_0(t,x) - I_0(\bar t,\bar{x})\vert \le C(T,R) \left(\Vert u_0\Vert_{\gamma_1} 
+\Vert v_0\Vert_{\infty,R+T}\right)\left(|x-\bar x|^{\gamma_1} 
+ |t-\bar t|^{\gamma_1}\right).
\eeq 
for some $C(T,R)>0$.

In the next arguments, we will use the following 
inequalities, 
whose proofs are easy. 
For all $0\le \bar t, t \le T$, $x,\bar x \in\RR$, there exists a positive constant $C(T)$ such that
\begin{align}
\label{s2-2.2}
 & \int_0^T \! ds\!  \int_{\RR}\!  dy \left\vert G(t-s,x-y)-G(\bar t-s,\bar x-y)\right\vert\notag\\
&\quad  =
2 \int_0^T \! ds \! \int_{\RR} \! dy \left\vert G(t-s,x-y)-G(\bar t-s,\bar x-y)\right\vert^2 
\le C(T)\left(|t-\bar t| + |x-\bar x|\right).
\end{align} 
 For any $\alpha >0$, as in the proof of Proposition \ref{lem1.1}, Minkovski's inequality and \eqref{c-L} imply 
\begin{align}			\label{diff-I1-1d}
\|I_1(t,x)-I_1&(\bar{t},\bar{x})\|_p \leq \int_0^T\!\! \!  ds \int_{\RR}\!  dy\ |G(t-s, x-y)-G(\bar{t}-s, \bar{x}-y)| \|b(u(s,y))\|_p \nonumber \\
\leq & \; C(T)\left[ c(b) + L(b) e^{\alpha T} \NN(u) \right]\left(|t-\bar{t}| + |x-\bar{x}| \right).
\end{align}

Upper bounds of increments of $I_2$ are also obtained following the arguments in the proof of Proposition \ref{lem1.1}, based on 
the Burkholder-Davies-Gundy and Minkowski inequalities. More precisely,
\begin{align*}
\|I_2(t,x)-I_2(\bar{t},\bar{x})\|_p^2 & \leq  4p  
\Big\|  \int_0^T\!\!\!  ds \int_\RR \! dy |G(t-s, x-y)-G(\bar{t}-s, \bar{x}-y)|^2 \sigma^2(u(s,y)) \Big\|_{\frac{p}{2}} \nonumber \\
& \leq 8 p\ C(T)\big[ c(\sigma)^2 + L(\sigma)^2 e^{2\alpha T} \NN(u)^2 \big] \left(|t-\bar{t}| + |x-\bar{x}| \right), 
\end{align*}
for any $\alpha>0$. Consequently,
\begin{align}
\label{diff-I2-1d}
\Vert I_2(t,x) - I_2(\bar{t},\bar{x})\Vert_p \le&\;  2\sqrt{2C(T)} \sqrt p \left[c(\sigma) + L(\sigma) e^{\alpha T} \NN(u)\right]\ 
 \left(|t-\bar{t}| + |x-\bar{x}| \right)^{\frac{1}{2}}.
\end{align}
Let $\gamma = \gamma_1 \wedge \frac{1}{2}$;  the inequalities \eqref{diff-I0-1d}, \eqref{diff-I1-1d} and 
 \eqref{diff-I2-1d} imply 
\eqref{increments-1d}. \\
Let $\alpha= 2\sqrt{L(b)}$ and $p\in\left[2, \frac{L(b)}{4L(\sigma)^2}\right]$; 
 then \eqref{NN-1} implies \eqref{increments-1d-n}.
The proof of the proposition is complete.
\end{proof}
 

From Proposition \ref{s2-2-p1}, using Kolmogorov's continuity lemma (see \cite[Theorem A.3.1]{Da-SS-Book} or \cite[Theorem C-6]{Kho}),
 we deduce the following.
\begin{prop}
\label{Kol-1d}
Let the initial values 
$u_0$, $v_0$ be as in Proposition \ref{s2-2-p1}. 
Let $\gamma = \gamma_1\wedge\frac{1}{2}$ and suppose that  $L(b)>\frac{8}{\gamma} L(\sigma)^2$.  
 Then $u$ has a version, still denoted by $u$, which is locally H\"older continuous  jointly in $(t,x)$  with exponent $\eta\in(0,\gamma)$.
 Furthermore,  given
any $p\in \big(\frac{2}{\gamma} , \frac{L(b)}{4 L(\sigma)^2}\big]$, 
there exists a constant   $C(p,T,R)$ such that
\beq	
\label{norm-unif-1d}
E\Big( \sup_{t\in [0,T]} \sup_{|x|\leq R} |u(t,x)|^p\Big) \leq 2^{p-1} \Vert u_0\Vert_{\infty,R}^p
+ C(p,T,R) \left[ {\mathcal M}_1^p +
{\mathcal M}_2^p\, {\mathcal M}_3^p \, e^{2pT\sqrt{L(b)}} \right],
\eeq
where ${\mathcal M}_1$, ${\mathcal M}_2$ are defined in \eqref{emes}, and
\beqn
{\mathcal M}_3= \frac{e^{-1} \|v_0\|_{\infty,R}}{\sqrt{L(b)}} + 2 \|u_0\|_{\infty,R} + C\Big[ \frac{c(b)}{L(b)} + \frac{c(\sigma)}{L(\sigma)}\Big],
\eeqn
with  the  universal constant $C$  in the right-hand side of \eqref{NN-1}.
\end{prop}

\begin{proof}  For any $s,t\in[0,T]$, $x,y\in[-R,R]$, set
$\Delta(t,x;s,y)=|t-s|^\gamma+|x-y|^\gamma$. \\
Proposition \ref{s2-2-p1} implies  
\beqn
E(|u(t,x)-u(s,y)|^p) \le K(\Delta(t,x;s,y))^p,
\eeqn
with
\begin{equation}
\label{s2-2.4}
 K:= C(p,T,R) \left[ {\mathcal M}_1^p + {\mathcal M}_2^p e^{\alpha p T} {\mathcal N}_{\alpha,p}(u)^p \right], \ \alpha>0.
 \end{equation}
 Apply \cite[Theorem A.3.1]{Da-SS-Book} with $k=1$, $\alpha_1=\alpha_2=\gamma$, $I=[0,T]$, $J=[-R,R]$,
$p\in \big( \frac{2}{\gamma}, \infty\big)$, to infer the existence of a version of $u$ (that we still denote by $u$) with jointly H\"older continuous sample paths of exponent $\eta\in(0,\gamma)$. Moreover, since by \eqref{wave-1},
$ C_1:= E\Big(\sup_{|x|\le R} |u(0,x)|^p\Big) =\Vert u_0\Vert_{\infty,R}^p$,
 we deduce from \cite[Equation(2.8.50)]{Da-SS-Book},
 \begin{align}
 \label{s2-2.5}
 E\Big( \sup_{t\in [0,T]} \sup_{|x|\leq R} |u(t,x)|^p\Big)  \leq 2^{p-1}\Vert u_0\Vert_{\infty,R}^p
 + C(p,T,R) K,
 \end{align}
where $K$ is defined in \eqref{s2-2.4}. Observe that $K$ depends on $\alpha$.

Choose $\alpha = 2\sqrt{L(b)}$. Then \eqref{s2-2.5} and  \eqref{s2-2.4} yield 
\begin{align}
\label{s2-2.50}
E\Big( \sup_{t\in [0,T]}& \sup_{|x|\leq R} |u(t,x)|^p\Big)  \leq 2^{p-1}\Vert u_0\Vert_{\infty,R}^p \nonumber \\
 &\quad+ C(p,T,R)   \left[ {\mathcal M}_1^p + {\mathcal M}_2^p e^{2 p T\sqrt{L(b)}} {\mathcal N}_{2\sqrt{L(b)},p}(u)^p \right].
 \end{align}
Notice that, since $\gamma\le 1/2$, the condition  $L(b)>\frac{8}{\gamma}L(\sigma)^2$  implies that the hypotheses 
of Proposition \ref{lem1.1} are satisfied. Hence,  using \eqref{NN-1} to upper estimate 
$ \mathcal{N}_{2\sqrt{L(b)},p}(u)$ on the right-hand side of \eqref{s2-2.50}, 
and since we are considering $|x|\le R$, we obtain \eqref{norm-unif-1d}. \end{proof} 
 
\subsection{Existence and uniqueness of global solution}
 \label{s2-3}
 
 In this section, we consider the equation \eqref{n3} with coefficients having superlinear growth
 and prove existence and uniqueness of a random field solution. 
 \smallskip
 
 We introduce the following set of hypotheses.
\smallskip

\noindent
{\bf (Cs)}\
The functions $b, \sigma: \RR\to \RR$ are locally Lipschitz and such that as $|z_1|, |z_2| \to \infty$, 
\begin{align*}
|b(z_1)-b(z_2)| &\leq \theta_2 |z_1-z_2| \left[\ln_+(|z_1-z_2|)\right]^\delta, \\
|\sigma(z_1)-\sigma(z_2)| &\le \sigma_2 |z_1-z_2| \left[\ln_+(|z_1-z_2|)\right]^a, 
\end{align*}
where $\theta_2, \sigma_2 \in (0,\infty)$, $\delta, a >0$, and $\ln_+(z)= \ln(z\vee e)$ for $z\ge 0$. 

\smallskip

\noindent{\bf (C1)} The parameters $\delta$, $a$ in {\bf (Cs)} satisfy one of the properties: (1) $\delta > 2a$;
(2) $\delta=2a$ and the constants $\theta_2$ and $\sigma_2$ are such that $\theta_2 > \bar \gamma \sigma_2^2$,
for some $\bar\gamma>0$. 
\smallskip

Notice that condition {\bf (Cs)} implies \eqref{super-coeff}, while {\bf (C1)}  says that 
 $b$ {\em dominates} $\sigma$.
 We define $\theta_1:=|b(0)|$ and $\sigma_1:=|\sigma(0)|$.

\begin{theorem}
\label{wp-d=1}
Assume that the initial condition $u_0$ is H\"older continuous with exponent $\gamma_1$, and $v_0$ is continuous. 
Set $\gamma = \gamma_1\wedge \frac{1}{2}$ and 
let the coefficients $b$ and $\sigma$ satisfy the conditions {\bf (Cs)} and {\bf(C1)} with $\delta <2$ and 
$\bar\gamma = 8\gamma^{-1}$. 
\begin{enumerate}
\item For any $M>0$, there exists a random field solution to \eqref{n3} in $[-M,M]$, $\big(u(t,x),\break (t,x)\in[0,T]\times [-M,M])$. This solution is unique and satisfies
\beq
\label{bd-as}
\sup_{(t,x)\in[0,T]\times[-M, M]}\vert u(t,x)\vert < \infty, \ a.s.
\eeq
\item Suppose that the initial conditions $u_0$, $v_0$ are functions with compact support included in $[-\rho,\rho]$,  
for some $\rho>0$, and $b(0)=\sigma(0)=0$. Then there exists a random field solution $\big(u(t,x),\ (t,x)\in[0,T]\times \RR)$ to \eqref{n3}.
This solution is unique and satisfies
\beq
\label{bd-as-bis}
\sup_{(t,x)\in[0,T]\times[-(\rho+T),\rho+T]}\vert u(t,x)\vert < \infty, \ a.s.
\eeq
\end{enumerate}
\end{theorem}

\begin{proof}
We start with some remarks. 
In statement 1. above, the notion of ``{\em random field solution to \eqref{n3} in $[-M,M]$}''  is made rigorous in Section \ref{s-app-s2}. According to Proposition \ref{app-p2}, the support of the sample paths of this solution is included in $[0,T]\times [-(M+T), M+T]$.

By proposition \ref{app-p1}, the assumptions in statement 2. imply that the support of the sample paths of the solution $\big(u(t,x),\ (t,x)\in[0,T]\times \RR)$
is included in $[0,T]\times[-(\rho+T),\rho+T]$. Hence, \eqref{bd-as-bis} is equivalent to 
\beqn
\sup_{(t,x)\in[0,T]\times \RR}\vert u(t,x)\vert < \infty, \ a.s.
\eeqn

\noindent{\em Solution for truncated Lipschitz continuous coefficients.}
 For a locally Lipschitz function $g: \RR\rightarrow \RR$  and $N\geq 1$, we define a globally Lipschitz function $g_N$ by
\begin{equation}
 \label{b_N} 
g_N(x) = g(x) 1_{\{|x|\le N\}} + g(N)1_{\{x>N\}} + g(-N) 1_{\{x<-N\}}.
\end{equation}
Using this definition for $\sigma$ and $b$, we consider \eqref{n3}
 with coefficients $\sigma_N$, $b_N$, and denote by
$u_N:=(u_N(t,x); (t,x)\in[0,T]\times \RR)$ its unique random field solution (see Proposition \ref{Proposition II.3}). 
 From {\bf (Cs)} we see that  if $N\geq 2$, $\sigma_N$, $b_N$ satisfy the conditions \eqref{c-L} 
with
\beq
\label{s2-3.1}
c(b_N)=\theta_1, \; c(\sigma_N)=\sigma_1, \; L(b_N)=\theta_2 (\ln (2N))^\delta, \; L(\sigma_N)=\sigma_2 (\ln (2N))^a. 
\eeq
Observe that, in the setting 2. of the Theorem, $\theta_1 = \sigma_1=0$.

Therefore, Proposition \ref{s2-2-p1} applies; by Kolmogorov's continuity criterion, there is a version of $u_N$ with jointly 
H\"older continuous sample paths of exponent $\eta\in(0,\gamma)$ in both variables. In the sequel we will consider this version that we will still denote by $u_N$. 
\smallskip

\noindent{\em Bounds for $L^p$ moments of $u_N$.}
 Assume that condition {\bf (C1)} (1) holds. Then, for $N$ large enough, we have
  $L(b_N) > \frac{8}{\gamma} L(\sigma_N)^2$. 
On the other hand, if condition {\bf (C1)} (2) is satisfied, then  $L(b_N) > \frac{8}{\gamma} L(\sigma_N)^2$ 
 holds for any  $N\ge 2$.  We can therefore apply Proposition \ref{Kol-1d} to see that for  any 
 $ p\in\left(\frac{2}{\gamma}, \frac{\theta_2(\ln (2N))^\delta}{4\sigma_2^2(\ln (2N))^{2a}}\right]$,  $R>0$, $N$ large enough (if necessary).

\begin{align}
\label{s2-3.2}
E\Big( \sup_{t\in [0,T]} \sup_{|x|\leq R} |u_N(t,x)|^p\Big) \leq &2^{p-1}\Vert u_0\Vert_{\infty,R}^p
\notag\\
& + C(p,T,R) \left[ {\mathcal M}_1^p +
{\mathcal M}_2^p(N)\, {\mathcal M}_3^p(N) \, e^{2pT\sqrt{L(b_N)}} \right],
\end{align}
where 
\begin{align}
\label{defms}
{\mathcal M}_1=&  \|u_0\|_{\gamma_1}+  \|v_0\|_{\infty,R+T}+ \theta_1 + \sqrt p\  \sigma_1, 
\quad 
{\mathcal M}_2(N)= L(b_N)+\sqrt{p}\ L(\sigma_N), \notag\\
{\mathcal M}_3(N)=& \frac{e^{-1} \|v_0\|_{\infty,R}}{\sqrt{L(b_N)}} 
+ 2 \|u_0\|_{\infty,R} + C\left[ \frac{\theta_1}{L(b_N)} + \frac{\sigma_1}{L(\sigma_N)}\right].
\end{align}

\noindent{\em Existence and uniqueness of a global solution.}
 Fix $R>0$. For any $N\ge 2$, set
\begin{equation} 				
\label{tauN}
\tau_N:=\inf\Big\{ t>0\; : \; \sup_{|x|\leq R} |u_N(t,x)|\geq N\Big\} \wedge T.
\end{equation}
The uniqueness of the solution and the local property of stochastic integrals imply that $u_N(t,x)=u_{N+1}(t,x)$ a.s. for $t\leq \tau_N$. 
Hence, almost surely, $\left(\tau_N\right)_{N\ge 2}$ is an increasing sequence, bounded by $T$. 

Assume that $\sup_N \tau_N=T$, a.s., and thus $\left\{ t\le \tau_N\right\}\uparrow \Omega$, a.s.  
On $\left\{ t\le \tau_N\right\}$, define $(u(t,x), (t,x)\in[0,T)\times \RR)$ by
$u(t,x)=u_N(t,x)$; then   $u(t,x)=u_M(t,x)$, for every $M\ge N$. 
The random variable $u(t,x)$ is well-defined and moreover, \eqref{n3} holds for any $(t,x)$, a.s.
Indeed, the definition of $\tau_N$ implies that on $\left\{ t\le \tau_N\right\}$,
\begin{align*}
u(t,x) =&  I_0(t,x) + \int_0^t ds \int_{\RR} dy\ G(t-s,x-y) b_N(u_N(s,y))\\
& +  \int_0^t \int_{\RR} G(t-s,x-y) \sigma_N(u_N(s,y)) W(ds,dy).
\end{align*}
But on $\left\{t\le \tau_N\right\}$, $b_N(u_N(s,y))= b(u_N(s,y))= b(u(s,y))$ and $\sigma_N(u_N(s,y))= \sigma(u_N(s,y))=\sigma(u(s,y))$. Since 
$\left\{ t\le \tau_N\right\} \uparrow \Omega$ a.s., we conclude that $(u(t,x), (t,x)\in[0,T)\times \RR)$ satisfies \eqref{n3}. 
Notice that, in this case, the stochastic integral in \eqref{n3} is not defined in $L^2(\Omega)$, 
but using instead an  extension defined a.s. (see e.g. \cite{Da-SS-Book}).

The last part of the proof is devoted to check that indeed, $\sup_N \tau_N=T$ a.s.
This will follow from the property
\begin{equation}			
\label{cv_tauN}
\lim_{N\to\infty} P(\tau_N < T) =0,
\end{equation}
that we now establish.
Let $C(p,T,R,N)$ denote the right-hand side of \eqref{s2-3.2}. To emphasise  the terms that depend on $N$, we write
\beq
\label{s3-2.3}
C(p,T,R,N) = C_1(p,T,R) + C_2(p,T,R,N),
\eeq
with
\begin{align*}
C_1(p,T,R) &= 2^{p-1}\Vert u_0\Vert_{\infty,R}^p + C(p,T,R){\mathcal M}_1^p,\\
C_2(p,T,R,N)  &= C(p,T,R) {\mathcal M}_2^p(N)\, {\mathcal M}_3^p(N) \, e^{2pT\sqrt{L(b_N)}}.
 \end{align*}
 Fix $p\in \left(\frac{2}{\gamma}, \frac{\theta_2(\ln(2N))^\delta}{4 \sigma_2^2(\ln(2N))^{2a}}\right]$. Applying Chebychev's inequality and then \eqref{s2-3.2}, we have
 \begin{align}
 \label{cheby}
P\left(\tau_N < T\right) &\le P\Big(\sup_{t\in[0,T]} \sup_{|x|\le R} |u_N(t,x)|\ge N\Big) 
\le N^{-p} E\Big(\sup_{t\in[0,T]}\sup_{|x|\le R} |u_N(t,x)|^p\Big)\notag\\
&\le N^{-p} C(p,T,R,N) = N^{-p}\left[C_1(p,T,R) + C_2(p,T,R,N)\right].
\end{align}
 Assume that
 \beq
\label{s3-2.4}
C_2(p,T,R,N) = {\rm o}(N^p).
\eeq 
Then, from \eqref{cheby}, we clearly obtain \eqref{cv_tauN}.

For the proof of \eqref{s3-2.4}, we first write the expressions of 
 ${\mathcal M}_2(N)$ and ${\mathcal M}_3(N)$  in \eqref{defms}, substituting $L(b_N)$ and $L(\sigma_N)$ by their respective values 
 given in \eqref{s2-3.1}. 
Because of the property $\sup_{N\geq 2} {\mathcal M}_3(N) \leq C$, we obtain 
\beqn
C_2(p,T,R,N) =  \tilde C_2(p,T,R) \exp\left(p \delta\ln[\ln(2N)] + 2pT \theta_2^{1/2}[\ln (2N)]^{\delta/2}\right).
\eeqn
Since $\delta <2$, this implies \eqref{s3-2.4}.

Let $M>0$ be as in Claim 1. From the above discussion, we deduce \eqref{bd-as} by taking $R=M$. Similarly,  Claim 2. is obtained by considering $R=\rho+T$.

The proof of the theorem is complete.
\end{proof}
\section{The stochastic wave equation in dimensions 2 and 3} 
\label{s3}
The aim of this section is to discuss the same questions as in Section \ref{s2} when $d = 2,3$, and the noise $W$ is white in time and  
coloured  in space. It is well-known that for dimensions $d\ge 2$, 
if $W$ is a space-time white noise, the stochastic convolution in \eqref{n3} fails to be a well-defined random variable in $L^2(\Omega)$, 
for almost any $(t,x)\in[0,T]\times \RR^d$. This is the case even if $\sigma$ is constant.
However, we can still obtain a random field solution of \eqref{n3} by taking a smoother noise in the spatial variable 
(see e.g. \cite{Wal}). This leads to the introduction in the next subsection \ref{s3-0} of a new class of Gaussian noises. 
\medskip

\subsection{Spatially homogeneous Gaussian noise and stochastic 
integrals}
\label{s3-0}

Let $\Lambda$ be a non-negative definite distribution in $\mathcal{S}^\prime(\RR^d)$. 
By the Bochner-Schwartz theorem (see e.g. \cite[Chap. VII, Thoerem XVIII]{Schwartz}), $\Lambda$ is the Fourier transform 
of a non-negative, tempered, symmetric measure
$\mu$ on $\RR^d$ called the spectral measure of $\Lambda$.
 In particular, $\Lambda$ is also a tempered distribution. 

On a complete probability space $(\Omega,\mathcal{A},P)$, 
we consider a Gaussian process $\{W(\varphi), \varphi\in \mathcal{C}_0(\RR^{d+1}\}$, indexed by the set of Schwartz test functions,
 with mean zero and covariance 
\beq
\label{s3.1}
E\left(W(\varphi)W(\psi)\right) = \int_0^\infty dt \int_{\RR^d} \Lambda(dx)\ \left(\varphi(t)\ast \tilde\psi(t)\right)(x),
\eeq
where ``$\ast$" denotes the convolution operator in the spatial variable and $\tilde\psi$ means reflection in the spatial variable too.

\noindent We will consider spatial covariances $\Lambda$ satisfying the following hypothesis (\cite{Dal}):
\medskip

\noindent {\bf (h0)} The spectral measure $\mu=\cF^{-1}\Lambda$ is such that
\beq
\int_{\RR^d} \frac{\mu(d\zeta)}{1+|\zeta|^2} < \infty.
\eeq
From \eqref{i2}, we see that this is equivalent to  
$\int_0^T dt \int_{\RR^d} \mu(d\zeta) \left\vert \cF G(t)(\zeta)\right\vert^2<\infty$.

Consider a jointly measurable adapted process $Z=\left(Z(t,x), (t,x)\in[0,T]\times \RR^d\right)$ such that $\sup_{(t,x)\in[0,T]\times \RR^d}E(|Z(t,x|^p) < \infty$, for some $p\in [ 2, \infty)$, and assume {\bf (h0)}. Then, the stochastic integral 
\beqn
((GZ)\cdot W)(t,x) := \int_0^t \int_{\RR^d} \ G(t-s,x-y) Z(s,y)\ W(ds,dy)
\eeqn
is a well-defined random variable. Moreover, for any $x\in\RR^d$, the process $((GZ)\cdot W)(t,x), t\in[0,T])$ is a martingale with respect to
 the natural filtration generated by $W$.
\smallskip

We  will consider the particular class of covariances $\Lambda$ described in {\bf(h1)} below.
 \smallskip
 
\noindent{\bf(h1)}\ \ 
 $\Lambda$ is an absolutely continuous measure, $\Lambda(dx)= f(x) dx$, $f\ge 0$.
   Its spectral measure $\mu=\cF^{-1}\Lambda$ is such that, for all signed measures $\Phi$ and $\Psi$ with finite total variation, 
  \beq
  \label{PG}
  \int_{\RR^d} \int_{\RR^d}  \Phi(dx)\ \Psi(dy) f(x-y) = C \int_{\RR^d} \mu(d\zeta) \cF\Phi(\zeta) \overline{\cF\Psi(\zeta)}.
  \eeq
  \begin{remark} \label{rk4.1}
  Assume 
 $\int_{\RR^d} \mu(d\zeta) \big[ |\cF(\Vert\Phi\Vert)(\zeta)|^2 +  |\cF(\Vert\Psi\Vert)(\zeta)|^2\big] <\infty$,
  where  the notation $\Vert\cdot\Vert$ stands for the total variation. Suppose also that $f:\RR^d\rightarrow [0,+\infty]$
  is lower semicontinuous. Then if $\Phi=\Psi$, \cite[Corollary 3.4]{F-K-2013} implies the validity of  \eqref{PG} with $C= (2\pi)^{-d}$. 
  By a polarity argument, \eqref{PG} can be extended to $\Phi\ne\Psi$ {\rm(}see \cite[p. 487]{KhosXiao}{\rm)}.
  \end{remark}
 Assume {\bf(h0)} and {\bf(h1)}. Since for any  $t>0$, $G(t,dx)$ is a non-negative finite measure with compact support, using \eqref{PG}
 with $\Phi = G(t,dx)$ and $\Psi = G(s,dy)$, $s,t >0$, we have
  \beq
  \label{PGwaves}
  \int_{\RR^d} \int_{\RR^d} \!G(t,dx)\ G(s,dy) f(x-y) = \frac{1}{(2\pi)^d} \int_{\RR^d} \!\mu(d\zeta) \cF G(t,\cdot)(\zeta) \overline{\cF G(s,\cdot)(\zeta)}.
  \eeq
  In particular,
  \begin{equation} 
 \label{s3.11Bis}
J(t) := \int_{\RR^d} \int_{\RR^d} G(t,dy) G(t,dy) f(x-y) = \frac{1}{(2\pi)^d} \int_{\RR^d} \mu(d\zeta) |\cF G(t)(\zeta)|^2. 
\end{equation}
Using \eqref{i2}, we have
$|\cF G(t)(\zeta)|^2\le \frac{2}{1+|\zeta|^2}\ 1_{\{|\zeta|\ge 1\}} + t^2 \ 1_{\{|\zeta| < 1\}}
\le \frac{2(1+t^2)}{1+|\zeta|^2}$ for $t>0$. 
Hence, 
\beq \label{cmu}
J(t) \le 2 (1+t^2) C_\mu,\quad {\rm where } \quad 
C_\mu:=  \frac{1}{(2\pi)^d} \int_{\RR^d} \frac{\mu(d\zeta)}{1+|\zeta|^2} < \infty,
\eeq
which implies, $\sup_{t\in[0,T]} J(t) < \infty$.

Assuming {\bf(h0)} and {\bf(h1)}, the stochastic integral $((GZ)\cdot W)(t,x)$ 
satisfies the following sharp version of the Burkholder Davies Gundy inequality
\begin{align}
\label{s3.4}
&\left\Vert ((G Z)\cdot W)(t,x)\right\Vert_p^p  \le \left(2\sqrt p\right)^p \notag \\
&\; \times E\Big(\!\int_0^t \!ds \! \int_{\RR^d} \! \int_{\RR^d}\! \! G(t-s,x-dy) G(t-s,x-dz) f(y-z) Z(s,y)Z(s,z)\Big)^{\frac{p}{2}}
\end{align}
(see e.g. \cite{Da-SS-Book}, \cite{nq-2007}). 
\smallskip


We end this section with a technical lemma related with the identity \eqref{PG}.  
For $d=3$, with a different proof, the result can be found in \cite[Lemma 6.5]{HHN}.
\begin{lemma}
\label{l-a-1}
Let $d\ge 1$, $t>0$ and $G(t)$ be the fundamental solution of the wave operator on $\rde$. Let $\varphi$, $\psi$ be bounded Borel measurable functions defined on $\rde$. Let $\Lambda$ be a symmetric measure satisfying {\bf (h1)}, with spectral measure $\mu = \mathcal{F}^{-1}\Lambda$ satisfying {\bf (h0)}. Then, for any $s, t>0$ and $z\in\rde$, we have
\begin{align}
\label{a.1}
&\int_{\rde} \int_{\rde} \varphi(x) G(t,dx) \psi(y) G(s,dy) f(x-y+z)\notag \\
&\qquad \qquad = 
\frac{1}{(2\pi)^d}\int_{\rde}\mathcal{F}\left(\varphi G(t)\right)(\xi)
\overline{\mathcal{F}\left(\psi G(s)\right)(\xi)} e^{-iz \cdot \xi}\ \mu(d\xi).
\end{align}
\end{lemma}
\begin{proof}
By applying the translation $\tau_{z}  x = x+z$, the left-hand side of \eqref{a.1} equals 
\beqn
\int_{{\RR}^d} \int_{{\RR}^d} \varphi(\tau_{-z} x) \tau_{-z } G(t,dx) \psi(y) G(s,dy) f(\tau_{z} x-y) 
= \int_{{\RR}^d} \int_{{\RR}^d}f(w-y)\  \Phi(dw) \Psi(dy),
\eeqn
where $\Phi(dw)= \varphi(\tau_{-z} w)\tau_{-z} G(t,dw)$ and $\Psi=\psi(y) G(s,dy)$. 
We recall that $\tau_{-z} G(t,dw)$ stands for the
translation of the measure $G(t,dw)$ by $-z$ in the distribution sense (see e.g. \cite[p. 55]{Schwartz}).

Because of the assumptions on $\varphi$ and $\psi$, the measures $\Phi(dw)$ and   $\Psi(dy)$
 are signed measures with finite total variation. We can therefore apply \eqref{PG} to deduce
\beqn
\int_{{\RR}^d} \int_{{\RR}^d}f(z-y)\ \Phi(d w) \Psi(dy) = \frac{1}{(2\pi)^d}\int_{{\RR}^d}\mathcal{F}f(\xi) \mathcal{F}(\Phi)(\xi)
\overline{\mathcal{F}(\Psi)(\xi)}\ d\xi.
\eeqn
Using the identities 
$\mathcal{F}(\Phi)(\xi) = \mathcal{F}\left(\tau_{-z}\varphi(\cdot) \tau_{-z}G(t,\cdot)\right)(\xi) = e^{-i\xi\cdot z}\mathcal{F}\left(\varphi G(t,\cdot)\right)(\xi)$,
 we obtain \eqref{a.1}.  
 \end{proof}
\subsection{Qualitative moment estimates}			
\label{s3-1}
We introduce a set of assumptions that ensure the existence and uniqueness of a random field solution to \eqref{n3}.
\smallskip

\noindent{\bf (he)}

\noindent 
{\bf (i)}  The functions $b$ and $\sigma$ are Lipschitz continuous; 
thus  \eqref{c-L}  holds. 

\noindent {\bf (ii)} $W$ is a spatially homogeneous noise as described in Section \ref{s3-0}. 
Its covariance and spectral measures ($\Lambda$ and $\mu$, respectively) satisfy {\bf (h0)} and {\bf (h1)}.

\noindent {\bf  (iii)} The initial values $u_0$, $v_0$ are such that the function $(t,x)\mapsto I_0(t,x)$ defined in \eqref{decom} is continuous and
\beq
\label{s3.5}
\sup_{(t,x)\in[0,T]\times \RR^d} \left\vert I_0(t,x)\right\vert < \infty.
\eeq
\begin{theorem}
\label{randomfield-multi}
Assume that {\bf (he)} is satisfied. Then there exists a random field solution
$\left(u(t,x), (t,x)\in[0,T]\times \RR^d\right)$ to \eqref{n3}, and for any $p\in [ 1, \infty)$, 
\beq
\label{s3.6}
\sup_{(t,x)\in[0,T]\times \RR^d} \Vert u(t,x)\Vert_p < \infty.
\eeq
This solution is unique in the class of jointly measurable, adapted processes $u$ satisfying \eqref{s3.6} with $p=2$.  
\end{theorem}
In the case $u_0 = v_0 = 0$, this follows from  \cite[Theorem 13]{Dal} applied to the wave operator. For non-null initial conditions, this follows from \cite[Theorem 4.3]{dalang-quer-2011}.
\begin{prop}
\label{p-s3.1}
In addition to {\bf (he)}, we assume that the initial values $u_0$, $v_0$, satisfy the following conditions:
\begin{enumerate}
\item for $d=2$,  $u_0$ is a bounded  function of class $C^1$ with bounded partial derivatives; $v_0$ is continuous and bounded;
\item for $d=3$, $u_0$ is a  
 bounded function of class $C^2$ with bounded second order partial derivatives; $v_0$ is continuous and bounded.
\end{enumerate}
We also suppose that the covariance measure $\Lambda$ satisfies  {\bf (h1)}, and the Lipschitz constants $L(b)$, $L(\sigma)$ are such that
$L(b)\geq \left( 2^{12}\, 3^2\,  C_\mu^2\,  L(\sigma)^4\right)\vee \frac{1}{4}$,  where $C_\mu$ is given in \eqref{cmu}. 
Then, for any
$p\in \Big[ 2, \frac{  \sqrt{L(b)}}{ 2^5\, 3 \, C_\mu \, L(\sigma)^2}\Big]$ we have
\beq
\label{p-s3.1-e1}
\mathcal{N}_{2\sqrt{L(b)},p}(u)\le C\Big[\mathcal{T}_0
 + \frac{c(b)}{L(b)} + \frac{c(\sigma)}{L(\sigma)}\Big],
 \eeq
 where $C$ is a universal constant and
 \beq
 \label{p-s3.1-e2}
  \mathcal{T}_0 = 
  \begin{cases}
  \|u_0\|_\infty + \frac{1}{\sqrt{L(b)}} \, \big(  \|\nabla u_0\|_\infty + \|v_0\|_\infty\big) , & {\text {if}}\ \ d=2,\\
  \|u_0\|_\infty + \|\Delta u_0\|_\infty + \frac{1}{\sqrt{L(b)}} \|v_0\|_\infty, & {\text {if}}\ \ d=3.
  \end{cases}
\eeq
As a consequence, we deduce that for $t\in [0,T]$ and $p\in \Big[ 2, \frac{  \sqrt{L(b)}}{ 2^5\, 3 \, C_\mu \, L(\sigma)^2}\Big]$,
\beq
 \label{p-s3.1-sup}
 \sup_{x\in\RR^d}E(|u(t,x)|^p) \le C^p\ e^{2pt\sqrt{L(b)}} \Big[\mathcal{T}_0 + \frac{c(b)}{L(b)} + \frac{c(\sigma)}{L(\sigma)}\Big]^p.
 \eeq
\end{prop}
\begin{proof}
We will consider the contributions to $\NN$ of each of the terms $I_i(t,x)$ in \eqref{decom}. 
\smallskip

\noindent{\em Estimates of $\NN(I_0)$.}
Consider first the case $d=2$. 
Using (1.11) and (1.12) from \cite{MiSS}, we have for $t >0$ and $x\in \RR^2$ 
\beqn
\big| [G(t) \ast v_0](x) \big|\leq  t \, \|v_0\|_\infty, \quad 
\left|\frac{\partial}{\partial t}  \big[ G(t)\ast u_0\big] (x) \right| \leq C\left\{\|u_0\|_\infty + t\, \|\nabla u_0\|_\infty\right\}.
\eeqn
Since  $\sup_{t\geq 0} (t e^{-\alpha t}) = (e\alpha)^{-1}$,  
 we deduce that for any $\alpha >0$ and $p\in [2,\infty)$, 
\begin{equation} 			
\label{IC-2d}
\NN(I_0) \leq C \Big[ \|u_0\|_\infty + \frac{e^{-1}}{\alpha} \, \big( \|v_0\|_\infty + \|\nabla u_0\|_\infty\big) \Big]. 
\end{equation}

Let $d=3$. Using \eqref{Green-3} and 
$\int_{\RR^3} G(t, dx)=t$,  we obtain, for $t>0$ and $x\in \RR^3$,
\beqn
\vert[G(t) \ast v_0](x)\vert  = \Big\vert\int_{|y|=t} v_0(x-y)\ G(t,dy)\Big\vert
\le \Vert v_0\Vert_\infty \int_{|y|=t} G(t,dy)= t \Vert v_0\Vert_\infty.
\eeqn
By applying the formula
$
\frac{d}{d t} \big( G(t)\ast u_0\big) = \frac{1}{t} \big( G(t)\ast u_0\big) + \frac{1}{4\pi} \int_{\{ |y|\leq 1\} } (\Delta u_0)(.+ty) dy$
(see \cite{Sog}), we have
$\left\vert \frac{d}{d t} \big( G(t)\ast u_0\big)(x)\right\vert \le \Vert u_0\Vert_\infty + \frac{1}{3} \Vert\Delta u_0\Vert_\infty$.
Therefore, 
\begin{equation}			
\label{I0-3d}
{\mathcal N}_{\alpha,p}(I_0) \leq  \frac{e^{-1}}{\alpha} \|v_0\|_\infty +  \|u_0\|_\infty + \frac{1}{3}\, \|\Delta u_0\|_\infty.
\end{equation}
\noindent{\em Estimates of $\NN(I_1)$.}
Use the expression of $I_1(t,x)$ given in \eqref{decom} and then Minkovski's inequality along with $\eqref{c-L}$ to obtain
\begin{align*}
\Vert I_1(t,x)\Vert _p 
&\le \int_0^t ds \int_{\RR^d} G(t-s,dy)\left[c(b) + L(b) \left\Vert u(s,x-y)\right\Vert_p\right] \\
&= \frac{t^2}{2} c(b)+ L(b) \int_0^t ds\ (t-s) \left(\sup_{x\in\RR^d}\left\Vert u(s,x)\right\Vert_p \right).
\end{align*}
From the above estimates, 
an argument similar to that used to prove \eqref{T1-1} implies 
\begin{align}
\label{s3.7}
\NN(I_1) 
&\le c(b) \sup_{t\ge 0}\Big(\frac{t^2}{2} e^{-\alpha t}\Big)\notag\\
&\quad \quad + L(b)\sup_{t\in[0,T]}\int_0^t ds (t-s) e^{-\alpha(t-s)}\Big(\sup_{(s,x)\in[0,T]\times\RR^d} e^{-\alpha s} \Vert u(s,x)\Vert_p\Big)\notag\\
&\le \frac{2e^{-2}}{\alpha^2} c(b) + \frac{1}{\alpha^2}L(b) \NN(u).
\end{align} 

\noindent{\em Estimates of $\NN(I_2)$}
Applying  \eqref{s3.4} with $Z(s,y):= \sigma(u(s,y))$ and then Minkowski's inequality,
we obtain
\begin{align*}
\Vert& I_2(t,x)\Vert_p^2 \le 4p \left\{E\Big[\int_0^t ds\ \int_{\RR^d} \int_{\RR^d} G(t-s,x-dy) G(t-s,x-dz) f(y-z)\right.\\
&\left.\qquad \qquad \times \sigma(u(s,y)) \sigma(u(s,z))\Big]^{\frac{p}{2}}\right\}^{\frac{2}{p}}\\
& \le 4p \int_0^t ds\ \int_{\RR^d} \int_{\RR^d} G(t-s,x-dy) G(t-s,x-dz) f(y-z) 
\Vert\sigma(u(s,y)) \sigma(u(s,z))\Vert_{\frac{p}{2}}\\
& \le  4p \int_0^t ds\ \int_{\RR^d} \int_{\RR^d} G(t-s,x-dy) G(t-s,x-dz) f(y-z) 
\Vert\sigma(u(s,y))\Vert_p \Vert\sigma(u(s,z))\Vert_p.
\end{align*}
Then, from \eqref{c-L} and the inequality $2ab\le a^2+b^2$ (valid for  $a,b\in\RR$), we deduce
\begin{align}
\label{prendos}
\Vert I_2(t,x)\Vert_p^2 &\le 4p \int_0^t ds\ \int_{\RR^d} \int_{\RR^d} G(t-s,x-dy) G(t-s,x-dz) f(y-z)\notag\\
&\qquad\qquad \times\big[c(\sigma)+L(\sigma)\Vert u(s,y)\Vert_p\big]^2\notag\\
&\le 8p\int_0^t ds\ \int_{\RR^d} \int_{\RR^d} G(t-s,x-dy) G(t-s,x-dz) f(y-z)\notag\\
&\qquad\qquad \times \big[c(\sigma)^2 + L(\sigma)^2\Vert u(s,y)\Vert_p^2\big].
\end{align}
Using the notation introduced in \eqref{s3.11Bis}, we can rewrite  \eqref{prendos} as follows
\beq
\label{prendoss}
\Vert I_2(t,x)\Vert_p^2 \le 8p \Big[c(\sigma)^2 \int_0^t ds\  J(t-s) + L(\sigma)^2 \int_0^t ds\ J(t-s) \, \sup_{y\in \RR^d} \Vert u(s,y)\Vert_p^2
\Big].
\eeq
From here, using the change of varables $s\mapsto t-s$, we have
\beq
\label{s3.13}
\NN(I_2)
\le \sqrt{8p}\,  \nu_1(\alpha)\, c(\sigma) + \sqrt{8p} \, \nu_2(\alpha) \, L(\sigma) \NN(u),
\eeq
where the finite constants $\nu_1(\alpha)$ are $\nu_2(\alpha)$ are defined by 
\beq
\label{s3.12}
\nu_1(\alpha):=\sup_{t\in[0,T]} \Big(e^{-2\alpha t}\int_0^t ds J(s)\Big)^{\frac{1}{2}}, \;
\nu_2(\alpha):= \sup_{t\in[0,T]} \Big(\int_0^t ds\ e^{-2\alpha s} J(s)\Big)^{\frac{1}{2}}.
\eeq

Thus, owing to \eqref{s3.7}, \eqref{s3.13}, we deduce  
\begin{align}
\label{s3.14}
\NN(u)& \le  \NN(I_0) + \frac{2e^{-2}}{\alpha^2} c(b) + \sqrt{8p}\,  c(\sigma)\, \nu_1(\alpha)\notag\\
&\quad\quad+ 2\max\left[\frac{L(b)}{\alpha^2}, \sqrt{8p}\, L(\sigma)\, \nu_2(\alpha)\right] \NN(u).
\end{align}
Using 
\eqref{cmu}  and the value of $\sup_{t\geq 0} (t^k e^{-\alpha t})$ for  $k=1,3,$ shown in \eqref{*}, we see that 
\begin{align}
\label{s3.15-1}
\nu_1(\alpha)\le C_\mu^{\frac{1}{2}} \; \sup_{t\in [0,T]} \Big( e^{-2\alpha t} \int_0^t 2\, (1+s^2)\, ds \Big)^{\frac{1}{2}} 
 \leq C_\mu^{\frac{1}{2}} \; \Big( \frac{e^{-1}}{\alpha} + \frac{9}{4}\, \frac{e^{-3}}{\alpha^3} \Big)^{\frac{1}{2}}.
\end{align}
Furthermore, using the 
inequality \eqref{cmu} and computing $\int_0^t s^2 e^{-\alpha s} ds$, we obtain
\begin{align}
\label{s3.15-2}
 \nu_2(\alpha)& \le C_\mu^{\frac{1}{2}} \; \sup_{t\in [0,T]} \Big( \int_0^t 2 (1+s^2)\, e^{-2\alpha s} ds \Big)^{\frac{1}{2}} 
 \leq C_\mu^{\frac{1}{2}} \;  \Big( \frac{1}{\alpha} + \frac{1}{2\alpha^3} \Big)^{\frac{1}{2}}.
\end{align}

Thus, 
\eqref{s3.14}-\eqref{s3.15-2} yield
\begin{align}
\label{s3.16}
\NN(u) 
& \le \NN(I_0) + \frac{2e^{-2}}{\alpha^2} c(b) + \sqrt{8p} \, c(\sigma)\, C_\mu^{\frac{1}{2}}
\left( \frac{e^{-1}}{\alpha} + \frac{9}{4}\, \frac{e^{-3}}{\alpha^3} \right)^{\frac{1}{2}} \notag \\
&  \quad + 2 \max \left[ \frac{L(b)}{\alpha^2}\ , \,  
\sqrt{8p}\,  L(\sigma)\, C_\mu^{\frac{1}{2}}   \Big( \frac{1}{\alpha} + \frac{1}{2\alpha^3} \Big)^{\frac{1}{2}}\right] \NN(u). 
\end{align}
Choose $\alpha^2 = 4 L(b)$. Since by assumption $L(b)\ge \frac{1}{4}$, we have $\alpha\ge 1$, which yields
\beqn
\frac{e^{-1}}{\alpha} + \frac{9}{4}\frac{e^{-3}}{\alpha^3} \le \frac{13}{8 e}L(b)^{-\frac{1}{2}}\quad \mbox{\rm and} \quad 
 \frac{1}{\alpha} + \frac{1}{2\alpha^3} \le \frac{3}{4}L(b)^{-\frac{1}{2}}.
 \eeqn
 Moreover, using once more the assumption $L(b)\geq \big[ 2^{12} \, 3^2 \, C_\mu^2\,  L(\sigma)^4\big] \vee \frac{1}{4}$, 
 we see that for $\alpha^2 = 4L(b)$ and for any $p\in \left[ 2,  \sqrt{L(b)}/\big( 2^5 \, 3\,  C_\mu\,  L(\sigma)^2\big) \right]$, 
 \beqn
\max \Big[ \frac{L(b)}{\alpha^2}\ , \,  
\sqrt{8p}\,  L(\sigma)\, C_\mu^{\frac{1}{2}}  \, \Big( \frac{1}{\alpha} + \frac{1}{2\alpha^3} \Big)^{\frac{1}{2}}   \Big]
 = \frac{1}{4}.
 \eeqn 
Hence, from  \eqref{s3.16}, 
 using the upper bound $p\le \frac{ \sqrt{L(b)}}{  2^5 \, 3\,  C_\mu\,  L(\sigma)^2}$,   \eqref{IC-2d} and \eqref{I0-3d}, we deduce
\begin{align*}
\mathcal{N}_{2\sqrt{L(b)},p}(u)&\le 2 \mathcal{N}_{2\sqrt{L(b)},p}(I_0) + e^{-2}\frac{c(b)}{L(b)} +
\left(\frac{13}{ 3 e 2^{3} }\right)^{\frac{1}{2}}
\frac{c(\sigma)}{L(\sigma)}\\
&\le C_1\mathcal{T}_0
 + C_2\left[ \frac{c(b)}{L(b)} + \frac{c(\sigma)}{L(\sigma)}\right],
 \end{align*} 
with  $\mathcal{T}_0$ defined in \eqref{p-s3.1-e2}.  
  This completes the proof of \eqref{p-s3.1-e1}.
  
   The inequality \eqref{p-s3.1-sup}  
 follows from \eqref{p-s3.1-e1} using the definition of $\NN(u)$.
\end{proof}

\subsection{Uniform bounds on moments}
\label{s3-3}

In this section, we address the problems of Section \ref{s2-2} when $d= 2, 3$, and $W$ is a noise white in time and coloured in space. 
The main task is to prove a result similar to
Proposition \ref{s2-2-p1} on moment estimates of increments 
in time and in space for the solution to equation \eqref{n3} with globally Lipschitz coefficients. 
\smallskip

\noindent{\em Increments of $I_0(t,x)$ in time and space}
\begin{prop}
\label{s3-3-p1}
Let $I_0(t,x)$, $(t,x)\in[0,T]\times \rde$ be as in \eqref{decom} and $R\ge 0$ be fixed. 
\begin{enumerate}

\item Let $d=2$. Assume that  $u_0$ is $\mathcal{C}^1$, $\nabla u_0$ is H\"older continuous with exponent $\gamma_1\in (0,1]$,
 and  $v_0$ is H\"older continuous with exponent $\gamma_2\in(0,1]$. 
Then, there exists a positive constant $C(T,R)$  
such that, for any $t,\bar t\in[0,T]$, and any $x, \bar x \in B(0;R)$, 
\begin{align}
\label{diff-I0-2d}
|I_0(t,x)-I_0(\bar{t},\bar{x})|
& \leq   C(T,R) \left(\|v_0\|_{\infty, R+T} +\Vert v_0\Vert_{\gamma_2} + \Vert \nabla u_0\Vert_{\infty,R+T}
 + \Vert \nabla u_0\Vert_{\gamma_1}\right)\notag\\
&\qquad \qquad\times\left( |t-\bar{t}|^{\gamma_1\wedge\gamma_2}+ |x-\bar x|^{\gamma_1\wedge \gamma_2}\right).
\end{align}

\item Let $d=3$. Assume that $u_0$ is $\mathcal{C}^2$, $\Delta u_0$ is H\"older continuous with exponent $\gamma_1\in(0,1]$, 
 and $v_0$ is H\"older continuous with exponent $\gamma_2\in(0,1]$. 
Then, there exists a positive constant $C(T,R)$ 
such that, for any $t,\bar t\in[0,T]$, and any $x, \bar x \in B(0;R)$, 
\begin{align} 
\label{diff-I0-3d}
|I_0(t,x)-I_0(\bar{t},\bar{x})| \leq & \, 
 C(T,R) \big[ \Vert v_0\Vert_{\gamma_2} + \Vert \nabla u_0\Vert_{\infty,R+T} + \Vert\Delta u_0\Vert_{\gamma_1}\big] \nonumber \\
&\quad  \times\left( |t-\bar t|^{\gamma_1\wedge \gamma_2} + |x-\bar x|^{\gamma_1\wedge \gamma_2}\right).
\end{align} 
\end{enumerate}
\end{prop}
\begin{proof}

(1). Let $0\le t\le \bar t\le T$ and $x\in B(0;R)$ be fixed. The scaling property $G(t,dx)=t\, G(1,dx)$ for $t>0$ implies 
\begin{align*}
& \left\vert [G(t) - G(\bar{t})]\ast v_0(x)\right\vert 
 = \Big\vert\int_{\RR^2} G(t, dy)\Big[ v_0(x-y) -\frac{\bar t}{t}  v_0\Big(x-\frac{\bar t}{t} y\Big)\Big]\Big\vert\\
&\quad \le \frac{\bar t}{t}\int_{\RR^2} G(t, dy)\ \Big\vert v_0(x-y) - v_0\Big(x-\frac{\bar t}{t} y\Big)\Big\vert 
+\Big\vert 1-\frac{\bar t}{t}\Big\vert\int_{\RR^2} G(t, dy)\ \vert v_0(x-y)\vert \\
&\quad \le T\Vert v_0\Vert_{\gamma_2} |t-\bar t|^{\gamma_2} + \Vert v_0\Vert_{\infty,R+T}  |t-\bar t|.
 \end{align*}
Consequently,
 $ {\displaystyle{\sup_{|x|\le R}\big\vert [G(t) - G(\bar{t})]\ast v_0(x)\big\vert \le C(T) \left(\Vert v_0\Vert_{\infty,R+T} + 
 \Vert v_0\Vert_{\gamma_2}\right)|| t-\bar t\vert^{\gamma_2}}}$. 

 According to the computations in \cite[p. 812-813]{MiSS}, we have 
 \beqn
 \sup_{|x|\le R} \Big| \frac{\partial}{\partial t}\big[  G(t)*u_0(x) - G(\bar{t})*u_0(x)\big]\Big| \leq C (\|\nabla u_0\|_{\infty,R+T} |t-\bar{t}| + \| \nabla u_0\|_{\gamma_1} \big)
 |t-\bar{t}|^{\gamma_1}.
 \eeqn 
 Thus,
 \beq
 \label{s3-3.01}
 \sup_{|x|\le R}|I_0(t,x) - I_0(\bar t,x)| \le C(T) \big( \|v_0\|_{\infty,R+T} + \Vert v_0\Vert_{\gamma_2} + \|\nabla u_0\|_{\infty,R+T} 
 + \| \nabla u_0\|_{\gamma_1}\big) 
 |t-\bar t|^{\gamma_1\wedge \gamma_2}.
 \eeq 
 Let now $0\le t\le T$ and $x, \bar x\in B(0, R)$ be fixed; then 
 \begin{align}
 		\label{s3-3.03}
 \big\vert (G&(t\ast v_0)(x) -  (G(t\ast v_0)(\bar x)\big\vert \le \int_{\RR^2} G(t,y)|v_0(x-y)-v_0(\bar x-y)|\ dy\notag\\
 & \le \Vert v_0\Vert_{\gamma_2}|x-\bar x|^{\gamma_2}\ \Big(\int_{\RR^2} G(t,y)\  dy\Big)
 \le T\Vert v_0\Vert_{\gamma_2}|x-\bar x|^{\gamma_2}.
 \end{align} 
 
 According to the computations in \cite[p. 815-816]{MiSS}, we have
\beqn
  \Big| \frac{\partial}{\partial t}\big[  G(t)*u_0(x) - G(t)*u_0(\bar{x})\big] \Big| \leq C \big(\|\nabla u_0\|_{\infty,T+R} |x-\bar{x}| + \| \nabla u_0\|_{\gamma_1}
 |x-\bar{x}|^{\gamma_1}\big) 
 \eeqn
 for $t\in [0,T]$. Therefore, 
 \beq
 \label{s3-3.02}
 \sup_{0\le t\le T}|I_0(t,x) - I_0(t,\bar x)| \le C(T,R) \big(\Vert v_0\Vert_{\gamma_2} + \|\nabla u_0\|_{\infty,T+R}+ \| \nabla u_0\|_{\gamma_1}\big) 
 |x-\bar{x}|^{\gamma_1\wedge \gamma_2}.
 \eeq
From the estimates \eqref{s3-3.01}--\eqref{s3-3.02}, we deduce \eqref{diff-I0-2d}.
\medskip


\noindent (2). Fix $0\le t\le \bar t\le T$ and $x\in B(0, R)$. 
According to \cite[Lemma 4.9, p. 43]{DSS-Memoirs}, we have
\begin{align*}
\sup_{|x|\le R} \Big\Vert \frac{\partial}{\partial t} \left(G(\cdot) \ast u_0\right)(x)\Big\Vert_{\gamma_1} &\le C \left(\Vert \nabla u_0\Vert_{\infty,R+T} + \Vert\Delta u_0\Vert_{\gamma_1}\right),\notag\\
\sup_{|x|\le R} \Vert \left(G(\cdot) \ast v_0\right)(x)\Vert_{\gamma_2} & \le C \Vert v_0\Vert_{\gamma_2},
\end{align*}
where $C>0$ is a universal constant. Consequently,
\beq
\label{s3-3.05}
\sup_{|x|\le R}|I_0(t,x) - I_0(\bar t,x)|  \le C(T,R) \big(\Vert \nabla u_0\Vert_{\infty,R+T} + \Vert\Delta u_0\Vert_{\gamma_1}+ 
 \Vert v_0\Vert_{\gamma_2}\big) |t-\bar t|^{\gamma_1\wedge \gamma_2}.
\eeq
 Fix  $0\le t\le T$ and $x, \bar x\in B(0, R)$. 
 Using the arguments in \cite[p. 362]{HHN} (see also \cite[Chapter 4]{DSS-Memoirs}), and
  the validity of the computations in \eqref{s3-3.03}  in dimension 3, we deduce 
 \beqn
  \sup_{0\le t\le T} \Big| \frac{\partial}{\partial t}\big[  G(t)*u_0(x) - G(t)*u_0(\bar{x})\big] \Big| 
  \leq C\left(\Vert\nabla u_0\Vert_{\infty,R+T} + \Vert\Delta u_0\Vert_{\gamma_1}\right) |x-\bar x|^{\gamma_1}.
   \eeqn
   Hence,
   \beq
  \label{s3-3.07}
   \sup_{0\le t\le T}|I_0(t,x) - I_0(t,\bar x)|\le C(T,R) \big(\Vert \nabla u_0\Vert_{\infty,R+T} + \Vert\Delta u_0\Vert_{\gamma_1}
   +  \Vert v_0\Vert_{\gamma_2}\big) |x-\bar x|^{\gamma_1\wedge \gamma_2}.
\eeq
The proof of \eqref{diff-I0-3d} is a consequence of \eqref{s3-3.05} and \eqref{s3-3.07}. 
\end{proof}

\begin{remark}
\label{s3-r1}
In comparison with the assumptions (1) and (2) of Proposition \ref{p-s3.1},
 in Proposition \ref{s3-3-p1} we restrict the space variable to a bounded set and 
therefore, the boundedness hypotheses are satisfied.
\end{remark}
\noindent{\em Increments of $I_1(t,x)$ in time and space} 
\begin{prop}
\label{s3-3-p2}

Let $I_1(t,x), (t,x)\in [0,T]\times \rde$ be as in \eqref{decom}. 
\smallskip

\noindent 1. Assume that the hypotheses {\bf (he)} are satisfied. Then there exists a positive constant $C(T)$ depending on $T$
 such that for any $(t,x), (\bar t,\bar x)\in [0,T]\times \rde$ and for any $p\in[2,\infty)$,
\begin{align}
\label{diff-I1-d}
&\Vert I_1(t,x) - I_1(\bar t, \bar x)\Vert_p \leq   C(T) \Big\{ |t-\bar t|  \Big[ c(b) + L(b) \sup_{(t,x)\times \rde}\Vert u(t,x)\Vert_p\Big]\\
& \; +   L(b) \! \int_0^t \! ds  \Big( \sup_{|z_1-z_2|=|x-\bar x|} \Vert u(s,z_1)-u(s,z_2)\Vert_p
+ \!\! \!  \sup_{|z_1-z_2| \le |t-\bar t|}\!  \Vert u(s,z_1)-u(s,z_2)\Vert_p\Big) \Big\}. 
\notag 
\end{align}

\noindent 2.  Assume the hypotheses of Proposition \ref{p-s3.1}. Then there exists a positive constant $C(T)$ depending on $T$ such that for
  any $p\in\Big[2,\frac{\sqrt{L(b)}}{2^5 3 C_\mu L(\sigma)^2}\Big]$ and any $(t,x), (\bar t,\bar x)\in [0,T]\times \rde$,
\begin{align}
\label{diff-I1-dbis}
&\left\Vert I_1(t,x) - I_1(\bar t, \bar x)\right\Vert_p 
\leq C(T) \Big\{ |t-\bar t|\left[c(b) + L(b) e^{2T\sqrt{L(b)}} \mathcal{N}_{2\sqrt{L(b)},p}(u)\right]  \\
& \; + L(b)\int_0^t \!\! ds\,  \Big( \sup_{|z_1-z_2|=|x-\bar x|} \Vert u(s,z_1)-u(s,z_2)\Vert_p
+ \!\!  \sup_{|z_1-z_2| \le |t-\bar t|} \Vert u(s,z_1)-u(s,z_2)\Vert_p\Big) \Big\},  \notag 
\end{align}
with $\mathcal{N}_{2\sqrt{L(b)},p}(u)$ satisfying \eqref{p-s3.1-e1}.
\end{prop}
\begin{proof} {\it 1.} \; 
Fix $t\in[0,T]$ and $x, \bar x\in \rde$. 
The Minkoswki inequality, the Lipschitz continuity of $b$ and the property $\int_{\RR^d} G(t,dx)=t$  yield 
\begin{align}
\label{s3-3.15}
\big\Vert I_1(t,x) - I_1(t,\bar x)\big\Vert_p 
& \le L(b )\int_0^t ds \ (t-s) \Big[\sup_{|z_1-z_2|=|x-\bar x|} \left\Vert u(s,z_1)-u(s,z_2)\right\Vert_p\Big]\notag\\
&\le L(b ) \, T \int_0^t ds \ \sup_{|z_1-z_2|=|x-\bar x|} \left\Vert u(s,z_1)-u(s,z_2)\right\Vert_p.
\end{align}

 By the triangle inequality 
$\left\Vert I_1(\bar t,x) - I_1(t,x)\right\Vert_p\le T_1(p;t,\bar t,x) + T_2(p;t,\bar t,x)$, for any $0\le t\le \bar t\le T$, where
\begin{align*}
T_1(p;t,\bar t,x) & = \Big\Vert \int_0^t ds \int_{\rde} \big[G(\bar t-s, dy) - G(t-s, dy)\big] b(u(s,x-y))\Big\Vert_p,\\
T_2(p;t,\bar t,x) & = \Big\Vert \int_t^{\bar t} ds \int_{\rde} G(\bar t-s, dy ) b(u(s,x-y))\Big\Vert_p.
\end{align*}
By the scaling property 
of the fundamental solution $G(t)$, $T_1(p;t,\bar t,x)$ equals
\beqn
\Big\Vert \int_0^t  ds \int_{\rde}  G(1,dz) \left[(t-s)b(u(s,x-(t-s)z)) - (\bar t-s)b(u(s,x-(\bar t-s)z))\right]\Big\Vert_p.
\eeqn
Apply Minkowski's inequality and use the Lipschitz property of $b$ and \eqref{c-L}; this yields 
\begin{align}
\label{s3-3.16}
&T_1(p;t,\bar t,x) 
  \le |t-\bar t| \int_0^t  ds \int_{\rde}  G(1,dz) \big[c(b) + L(b) \Vert u(s,x-(\bar t-s)z))\Vert_p \big] \notag \\
&\quad + L(b) \int_0^t ds \int_{\rde}  G(1,dz) (t-s) \left\Vert u(s,x-(t-s)z)) - u(s,x-(\bar t-s)z))\right\Vert_p. 
\end{align}

Since the support of $G(1,dz)$ is included in the closed ball $\overline{B(0;1)}$, we have
\begin{align*}
 &\int_0^t ds\ (t-s)\int_{\rde}  G(1,dz)  \left\Vert u(s,x-(t-s)z)) - u(s,x-(\bar t-s)z))\right\Vert_p\\
&\qquad  \le T  \int_0^t ds \int_{\rde}  G(1,dz) \sup_{|z_1-z_2| \le |t-\bar t|}  \left\Vert u(s,z_1) - u(s,z_2)\right\Vert_p\\
&\qquad = T  \int_0^t  ds \, \sup_{|z_1-z_2| \le |t-\bar t|}  \left\Vert u(s,z_1) - u(s,z_2)\right\Vert_p.
 \end{align*}
 The first term on the right-hand side of  \eqref{s3-3.16} is bounded from above by
 \begin{equation} 	\label{4.39.1}
 |t-\bar t| \Big\{T c(b) +L(b) \int_0^t ds \sup_{x\in\rde}\left\Vert u(s,x)\right\Vert_p\Big\}.
 \end{equation}
 
Thus,
\begin{align}
\label{s3-3.17}
T_1(p;t,\bar t,x) \le
& \;  T \Big(L(b) \int_0^t \sup_{|z_1-z_2|\le |t-\bar t|}  \left\Vert u(s,z_1) - u(s,z_2)\right\Vert_p ds  \notag\\
& + |t-\bar t| \Big\{c(b)  +L(b) \sup_{(t,x)\in[0,T]\times \rde}\left\Vert u(t,x)\right\Vert_p\Big\}\Big).
 \end{align}

 With similar arguments, we deduce the following upper bounds for $T_2(p;t,\bar t,x)$:
 
 \begin{align}
 \label{s3-3.18}
 T_2(p;t,\bar t,x)&\le 
 c(b)  \int_t^{\bar t} ds \int_{\rde} G(\bar t-s,dy) + L(b)\int_t^{\bar t} ds (\bar t-s) \sup_{x\in\rde}\left\Vert u(s,x)\right\Vert_p\notag\\
&\le c(b)\frac{(\bar t - t)^2}{2} + L(b )\frac{(\bar t - t)^2}{2} \sup_{(t,x)\in[0,T]\times \rde}\left\Vert u(t,x)\right\Vert_p.
\end{align}
From \eqref{s3-3.17} and \eqref{s3-3.18}, we obtain \eqref{diff-I1-d} for $x=\bar x$. Along with 
\eqref{s3-3.15}, we obtain \eqref{diff-I1-d}.
\bigskip

\noindent {\it 2.}\ 
This claim follows from the definition of \eqref{norm} and Proposition \ref{p-s3.1}.
\end{proof}


\noindent{\em Space increments  of $I_2(t,x)$}
\smallskip

While keeping assumption {\bf (h1)}, we consider a strengthening of {\bf (h0)}, denoted by {\bf (h2)}. This is condition ($c^\prime$) in \cite[p. 367]{HHN} on the spectral measure $\mu$. 
\medskip

\noindent {\bf (h2)} 
There exists $\gamma \in (0,1)$ such that the Fourier transform of the
tempered measure $|\zeta|^{2\gamma} \mu (d\zeta)$ is a non-negative locally integrable function $g_\gamma$, and moreover,
 \begin{equation*} 
 \int_{\RR^d} \frac{\mu(d\zeta)}{1+|\zeta|^{2-2\gamma}} <\infty.
 \end{equation*}
 Set
 \beq
 \label{cmugamma}
  C_\mu^{(\gamma)}:= \frac{1}{(2\pi)^d} \int_{\RR^d} \frac{\mu(d\zeta)}{1+|\zeta|^{2-2\gamma}}.
  \eeq
%
\begin{prop} 
\label{s3-3-p3}

 Let $I_2(t,x)$, $(t,x)\in[0,T]\times \rde$ be as in \eqref{decom}.
 \smallskip
 
 \noindent 1. Assume that the hypotheses {\bf(he)}, {\bf(h1)} and {\bf(h2)} are satisfied.
Then, 
 for  any $p\in [2,\infty)$ and  $t\in [0,T]$, there exists a positive constant $C$  such that, for every $x,\bar{x}\in \RR^d$, $t\in [0,T]$, 
 \begin{align}
  \label{diff-I2-x-d}
\Vert I_2(t,x)  - &I_2(t,\bar x)\Vert_p^2 \le C p\,  (1+T^2) C_\mu L(\sigma)^2 
\Big(\int_0^t ds \sup_{|z_1-z_2|=|x-\bar x|}
 \Vert u(s,z_1)-u(s,z_2)\Vert_p^2\Big)\notag\\
&\; + Cp\, (T+T^3) C^{(\gamma)}_\mu |x-\bar x|^{2\gamma} \Big[c(\sigma)+L(\sigma) \sup_{(s,y)\in[0,T]\times \rde}\Vert u(s,y)\Vert_p\Big]^2,
  \end{align}
  where $C_\mu$, $C_\mu^{(\gamma)}$  are defined in \eqref{cmu}, \eqref{cmugamma}, respectively.
  \smallskip
  
  \noindent{2.} Assume that the hypotheses of Proposition \ref{p-s3.1} hold. Then, for any 
  $p\in\Big[2, \frac{\sqrt{L(b)}}{2^5 3 C_\mu L(\sigma)^2}\Big]$, 
  there exists a positive constant $C$  such that, for every $x,\bar{x}\in \RR^d$, $t\in [0,T]$, 
 \begin{align}
  \label{diff-I2-x-dbis}
\Vert I_2(t,x) - &I_2(t,\bar x)\Vert_p^2 \le C p\, (1+T^2) C_\mu L(\sigma)^2 
\Big(\int_0^t ds \sup_{|z_1-z_2|=|x-\bar x|} 
\Vert u(s,z_1)-u(s,z_2)\Vert_p^2\Big)\notag\\
&\; + Cp\, (T+T^3) C_\mu^{(\gamma)}  |x-\bar x|^{2\gamma} \left[c(\sigma)+L(\sigma) e^{2T\sqrt{L(b)}} \mathcal{N}_{2\sqrt{L(b)},p}(u)\right]^2,
  \end{align}
   where 
   $\mathcal{N}_{2\sqrt{L(b)},p}(u)$ satisfies \eqref{p-s3.1-e1}.
   \end{prop}
\begin{proof} 
To simplify the presentation, we will use the notation of  \cite[Theorem 3.1]{HHN} that we recall below.  
For $s\in[0,T]$ and $x, \bar x, y, z\in \rde$, set $\xi=x-\bar x$ and
\begin{align*}
& \Sigma_x(s,y) = \sigma(u(s,x-y)),\quad \Sigma_{x,\bar{x}}(s,y) = \sigma(u(s,x-y))-\sigma(u(s,\bar{x}-y)),\\
& h_1(s,y,z) = f(y-z) \Sigma_{x,\bar{x}}(s,y) \Sigma_{x,\bar{x}}(s,z), \\
& h_2(s,y,z) = \big[f(y-z+\xi)-f(y-z)\big] \Sigma_x(s,z) \Sigma_{x,\bar{x}}(s,y), \quad  h_3(s,y,z) = h_2(s,z,y), \\
& h_4(s,y,z) = \big[2f(y-z)-f(y-z+\xi)-f(y-z-\xi)\big] \Sigma_x(s,y) \Sigma_x(s,z).
\end{align*}
Fix $p\in [2,\infty)$ and apply the Burkholder-Davies-Gundy inequality to obtain
\begin{equation}			\label{I2x-I2x}
\left\Vert I_2(t,x) - I_2(t,\bar x)\right\Vert_p^2 \le 4p \sum_{i=1}^4 \Vert Q_i(t;x,\bar x)\Vert_{\frac{p}{2}},
\end{equation}
 where, using  the transfer of  increments strategy introduced in  
\cite[p. 19]{DSS-Memoirs} (see also in \cite[p. 374]{HHN}),  we set 
\beqn
Q_i(t;x,\bar x)=\int_0^t ds \int_{\RR^d} \int_{\RR^d} G(t-s, dy) G(t-s, dz) h_i(s,y,z), \quad  i=1, \ldots, 4.
\eeqn

In \cite[Theorem 3.2]{HHN}, upper bounds of the terms $\Vert Q_i(t;x,\bar x)\Vert_{\frac{p}{2}}$ 
are established. We sketch here their proofs, paying attention to the value of the relevant constants and checking 
that the arguments also hold for $d=2$. 
\smallskip

\noindent {\em Upper bound of $\Vert Q_1(t;x,\bar x)\Vert_{\frac{p}{2}}$}.
Using Minkowski's inequality,  then the Cauchy-Schwarz inequality, the Lipschitz property of $\sigma$ and \eqref{cmu} we obtain 
\begin{align}
\label{Q1}
\Vert Q_1(t;x,\bar x)\Vert_{\frac{p}{2}} 
 \le &L(\sigma)^2 \int_0^t ds \int_{\RR^d} \int_{\RR^d} G(t-s, dy) G(t-s, dz)
f(y-z)\notag\\
&\qquad \qquad\qquad \qquad \times\Big[ \sup_{|z_1-z_2|=|x-\bar x|} \Vert u(s,z_1)-u(s,z_2)\Vert_p^2\Big]\notag\\
& \le 2 L(\sigma)^2 (1+T^2) C_\mu \int_0^t ds \sup_{|z_1-z_2|=|x-\bar x|} \Vert u(s,z_1)-u(s,z_2)\Vert_p^2.
\end{align}

For the study of the remaining terms $\Vert Q_i(t;x,\bar x)\Vert_{\frac{p}{2}}$, $i=2,3,4$, in order to be in the setting of   Lemma \ref{l-a-1},
we use a truncation argument on the processes $\Sigma_x(s,y)$, $\Sigma_{x,\bar{x}}(s,y)$. 
For 
$k\geq 1$, set $\Sigma^k_x(s,y)=\Sigma_x(s,y) 1_{\{|\Sigma_x(s,y)|\leq k\}}$, 
$\Sigma^k_{x,\bar{x}}(s,y)=\Sigma_{x,\bar{x}}(s,y) 1_{\{|\Sigma_{x,\bar{x}} (s,y)|\leq k\}}$, and 
\beqn
Q_i^k(t;x,\bar x)=\int_0^t ds \int_{\RR^d} \int_{\RR^d} G(t-s, dy) G(t-s, dz) h_i^k(s,y,z), \ i=2,3,4,
\eeqn
where each $h_i^k(s,y,z)$ is defined as $h_i(s,y,z)$ by replacing $\Sigma_x(s,y)$ and $\Sigma_{x,\bar{x}}(s,y)$  by
$\Sigma^k_x(s,y)$ and $\Sigma^k_{x,\bar{x}}(s,y)$, respectively.

\noindent {\em Upper bound for $\Vert Q_2^k(t;x,\bar x)\Vert_{\frac{p}{2}}$}.
Apply Lemma \ref{l-a-1} to the bounded functions $\varphi(z)= \Sigma_x^k(s,z)$ and $\psi(y)=\Sigma_{x,\bar{x}}^k(s,y)$. Then, up to the constant 
$(2\pi)^{-d}$, $Q_2^k(t;x,\bar x)$ is equal to 
\beqn
\int_0^t ds \int_{\RR^d} \overline{\cF \big( \Sigma^k_x(s,.) G(t-s,.)\big)(\zeta)}\ 
 \cF\big(\Sigma^k_{x,\bar{x}}(s,.) G(t-s,.)\big)(\zeta) \left[e^{-i\xi.\zeta}-1\right] \mu(d\zeta),
\eeqn
where $\xi = x-\bar{x}$. Since for any 
$\gamma \in (0,1]$, $|e^{-i \xi.\zeta}-1| \leq C |\xi|^\gamma |\zeta|^\gamma$, and $2\sqrt{ab} \le (a+b)$ for 
 $a, b\ge 0$, computations similar to those in \cite[p. 368]{HHN} imply
\beq
\label{twoqus}
\Vert Q_2^k(t;x,\bar x)\Vert_{\frac{p}{2}}\le C\left(\Vert Q_2^{k,1}(t;x,\bar x)\Vert_{\frac{p}{2}}+ \Vert Q_2^{k,2}(t;x,\bar x)\Vert_{\frac{p}{2}}\right),
\eeq
where
\begin{align*}
 Q_2^{k,1}(t;x,\bar x) &: = |\xi|^{2\gamma} \int_0^t ds \int_{\rde} 
\left\vert\cF \big( \Sigma^k_x(s,.)  G(t-s,.)\big)(\zeta)\right\vert^2 |\zeta|^{2\gamma} \mu(d\zeta),\\
 Q_2^{k,2}(t;x,\bar x) &: = \int_0^t ds \int_{\rde} 
\left\vert \cF\big(\Sigma^k_{x,\bar{x}}(s,.) G(t-s,.)\big)(\zeta)\right\vert^2 \mu(d\zeta).
\end{align*}
Set 
$J^{(\gamma)}(t) = \frac{1}{(2\pi)^d} \int_{\rde} \mu(d\zeta) |\zeta|^{2\gamma}|\cF G(t)(\zeta)|^2$.
 A minor change in the proof of \eqref{cmu} yields 
 \beq
\label{s3-3.30}
J^{(\gamma)}(t)\le 2(1+t^2) C_\mu^{(\gamma)}<\infty,
\eeq
where $C_\mu^{(\gamma)}$ is defined in \eqref{cmugamma}, and we have used 
the assumption {\bf (h2)}.

 
Using the Plancherel identity, the Minkowski inequality with respect to the non-negative measure 
$[G(t-s,.)*G(t-s,.)](y) g_\gamma(y) \,dy\, ds$ and once more the Plancherel identity, since $\tilde{G}(s,.)$ is symmetric, we deduce
(as in \cite[p. 369]{HHN}), 
 \begin{align*}
 \Vert Q_2^{k,1} (t;x,\bar x) \Vert_{\frac{p}{2}} 
& \le C |x-\bar{x}|^{2\gamma} \! \sup_{(s,y)\in [0,T]\times \RR^d} \| \Sigma_x^k(s,y)\|_p^2 
 \int_0^t ds \int_{\RR^d} \mu(d\zeta)   |\zeta|^{2\gamma} |\cF G(t-s)(\zeta)|^2 \\
&   \le C |x-\bar{x}|^{2\gamma} T\,  (1+T^2) \, C_\mu^{(\gamma)} \sup_{(s,y)\in [0,T]\times \RR^d} \| \Sigma_x^k(s,y)\|_p^2 .
\end{align*}
From \eqref{c-L} we have
\[\sup_{(s,y)\in[0,T]\times\rde} \Vert\Sigma_x(s,y)\Vert_p  \le c(\sigma) + L(\sigma) \sup_{(s,y)\in[0,T]\times\rde} \Vert u(s,y)\Vert_p.\]
 Therefore, 
\beq
\label{s3-3.31}
\Vert Q_2^{k,1}(t;x,\bar x)\Vert_{\frac{p}{2}} \le C |x-\bar x|^{2\gamma}(T+T^3) C_\mu^{(\gamma)} 
\Big[c(\sigma) + L(\sigma) \sup_{(s,y)\in[0,T]\times\rde} \Vert u(s,y)\Vert_p\Big]^2.
\eeq
With similar arguments, we obtain 
\begin{align} \label{s3-3.32}
&\| Q_2^{k,2}(t;x,\bar{x})\|_{\frac{p}{2}} 
 \leq C \int_0^t ds \sup_{y\in \RR^d} \|\Sigma^k_{x,\bar{x}}(s,y)\|_p^2  \int_{\rde} 
\left\vert\cF( G(t-s,.))(\zeta)\right\vert^2 \mu(d\zeta) \nonumber \\
&\quad \leq C \int_0^t  \sup_{y\in \RR^d} \|\Sigma^k_{x,\bar{x}}(s,y)\|_p^2  J(t-s) ds
\le  C (1+T^2)\, C_\mu \int_0^t \sup_{y\in \RR^d} \|\Sigma^k_{x,\bar{x}}(s,y)\|_p^2 ds \nonumber \\
&\quad \le C (1+T^2) C_\mu L(\sigma)^2\int_0^t ds \sup_{|z_1-z_2|=|x-\bar x|}\Vert u(s,z_1)-u(s,z_2)\Vert_p^2,
\end{align}
where in the last inequality, we have used \eqref{cmu}, the Lipschitz property of $\sigma$ and  the upper estimate 
$\Vert \Sigma^k_{x,\bar x}(s,y)\Vert_p \le \Vert \Sigma_{x,\bar x}(s,y)\Vert_p$.

Summarising, \eqref{twoqus}, along with \eqref{s3-3.31} and \eqref{s3-3.32} imply 
\begin{align}
\label{s3-3.33}
\Vert Q_2^k&(t; x,\bar x)\Vert_{\frac{p}{2}}  \le  C (T+T^3) C_\mu^{(\gamma)}  |x-\bar x|^{2\gamma}
\Big[c(\sigma) + L(\sigma) \sup_{(s,y)\in[0,T]\times\rde} \Vert u(s,y)\Vert_p\Big]^2\notag\\
&\quad  + C (1+T^2) C_\mu  L(\sigma)^2\int_0^t ds \sup_{|z_1-z_2|=|x-\bar x|}\Vert u(s,z_1)-u(s,z_2)\Vert_p^2,
\end{align}
for  some universal positive  constant $C$. 

Notice that, since $|e^{-i \xi.\zeta}-1|=|e^{i \xi.\zeta}-1|$, swapping $y$ and $z$ we deduce that \eqref{s3-3.33} also holds for
$\Vert Q_3^k(t; x,\bar x)\Vert_{\frac{p}{2}}$.
\medskip

\noindent{\em Upper bound of $\Vert Q_4^k(t; x,\bar x)\Vert_{\frac{p}{2}}$}.
Applying Lemma \ref{l-a-1} with $\varphi=\psi=\Sigma_x^k(s,.)$ and then Plancherel's identity, we obtain 
\begin{align*}
| Q_4^k(t; x,\bar x) | &\le  \frac{1}{(2\pi)^d} \int_0^t \!ds\! \int_{\RR^d}\! dy\,  \big|2-e^{-i\xi.\zeta} -e^{i\xi.\zeta}\big| \; 
\big| \cF\big(\varphi G(t-s,.)\big) \big|^2\, \mu(d\zeta)\\
& \leq C |\xi|^{2\gamma} \int_0^t ds \int_{\RR^d}  dy\;  g_\gamma(y) \big[ \big( \Sigma^k_x(s,.) G(t-s,.)\big) * \widetilde{\big(\Sigma^k_x(s,.) G(t-s,.)}\big)
\big](y),  
\end{align*}
where in the last inequality, we have used that for $\gamma\in (0,1]$,
$\vert 1-\cos(\xi.\zeta)\vert\le C (|\xi||\zeta|)^{2\gamma}$.

Consider the non-negative measure $g_\gamma(y)\big[ G(t-s,.) * G(t-s,.)\big](y) ds\ dy$. The Minkowski inequality with respect to this measure, the Plancherel identity and \eqref{s3-3.30}  yield 
\begin{align*}
 \| Q_4^k(t; x,\bar x) \|_{\frac{p}{2}} & 
 \le C |\xi|^{2\gamma} \int_0^t ds \int_{\RR^d} dy\ g_\gamma(y) \big[ G(t-s,.) * G(t-s,.)\big](y)\\
 &\qquad \qquad\qquad\qquad \quad\quad\times \sup_{y,z\in \RR^d}
\|\Sigma_x^k(s,y) \Sigma_x^k(s,y+z)\|_{\frac{p}{2}} \\
& \le C  |\xi|^{2\gamma} \int_0^t ds \sup_{y\in \RR^d} \|\Sigma_x(s,y)\|_p^2 \, 2 (1+T^2) C_\mu^{(\gamma)}.
\end{align*}
Thus, an argument similar to that proving  \eqref{s3-3.31} implies 
\beq
\label{s3-3.34}
\Vert Q_4^k(t; x,\bar x)\Vert_{\frac{p}{2}}
 \le C (T+T^3) C_\mu^{(\gamma)}\,  |x-\bar x|^{2\gamma}  \Big[c(\sigma) 
 + L(\sigma) \sup_{(s,y)\in[0,T]\times\rde} \Vert u(s,y)\Vert_p\Big]^2.
\eeq
The upper estimates \eqref{I2x-I2x}, \eqref{Q1}, \eqref{s3-3.33} and \eqref{s3-3.34} conclude the proof of \eqref{diff-I2-x-d}. 
\smallskip

The statement in part 2  is an immediate consequence of the definition of $\mathcal{N}_{2\sqrt{L(b)},p}(u)$ and Proposition \ref{p-s3.1}.
 The proof of the proposition is complete.
\end{proof}
\medskip


From Propositions \ref{p-s3.1}-\ref{s3-3-p3}, we derive estimates on space increments of the random field \eqref{n3} for $d = 2, 3$. 
Later on, they will be used to deduce estimates on time increments of $I_2(t,x)$.
For its further use, set
\beq
\label{izero}
K_0(u_0, v_0) = \begin{cases}
\Vert v_0\Vert_{\gamma_2} + \Vert \nabla u_0\Vert_{\infty,R+T} + \Vert \nabla u_0\Vert_{\gamma_1}, & d=2,\\
\Vert v_0\Vert_{\gamma_2} + \Vert \nabla u_0\Vert_{\infty,R+T} + \Vert\Delta u_0\Vert_{\gamma_1}, & d=3.
\end{cases}
\eeq

\begin{prop}
\label{s3-3-p4}
We are assuming the following.
\begin{enumerate}
\item The initial value functions $u_0$ and $v_0$ satisfy the conditions of Proposition \ref{s3-3-p1} 
with some H\"older exponents $\gamma_1, \gamma_2 \in (0,1]$.
\item The coefficients $\sigma$ and $b$ are globally Lipschitz continuous functions.
\item The covariance measure $\Lambda$ of the noise $W$ satisfies {\bf(h1)}, and the corresponding spectral measure $\mu$ satisfies {\bf(h2)}.
\end{enumerate}
{\rm{(i)}} \ Fix $T, R>0$. Then, for any $p\in[2,\infty)$ and $\alpha>0$,
there exist positive constants  $c_1(T,R)$, $c_2(T)$ and $c_3(T)$ such that if 
\begin{align}
\label{ccs}
C_1 &:=  c_1(T,R)\ K_0(u_0, v_0),\notag\\ 
C_2 &: = c_2(T) \, \big( p C_\mu^{(\gamma)} \big)^{\frac{1}{2}}\Big[c(\sigma) +   L(\sigma) \sup_{(t,x)\in[0,T]\times \rde}\Vert u(t,x)\Vert_p \Big],\notag\\
C_3 &:=  c_3(T) \, \big[  L(b)^2  + p\, C_\mu \,  L(\sigma)^2 \big],
\end{align}
with $C_\mu$ and $C_\mu^{(\gamma)}$ defined in \eqref{cmu} and \eqref{cmugamma}, 
then for any $t\in[0,T]$, and $x, \bar x\in B(0;R)$,
\begin{align} \label{s3-3.37}
\sup_{|z_1-z_2|\le|x-\bar x|}\Vert u(t,z_1) - u(t,z_2)\Vert_p^2 
\le \exp(T C_3)\,  \big(C_1^2  |x-\bar x|^{2(\gamma_1\wedge \gamma_2)}
+ C_2^2 |x-\bar x|^{2\gamma} \big) .
\end{align} 
Consequently,
\beq
\label{s3-3.38}
\sup_{t\in[0,T]}\sup_{|z_1-z_2|\le|x-\bar x|}\Vert u(t,z_1) - u(t,z_2)\Vert_p \le \tilde C |x-\bar x|^{\nu_1},
\eeq
with  $\nu_1= \min(\gamma, \gamma_1, \gamma_2)$,  and 
$\tilde C = C(R)\left(C_1 + C_2\right)\exp(TC_3/2)$. 
\smallskip

{{\rm(ii)}}\ Suppose furthermore that the Lipschitz constants $L(b)$, $L(\sigma)$ do not vanish and are such that
$L(b)\geq \left( 2^{12}\, 3^2\,  C_\mu^2\,  L(\sigma)^4\right)\vee \frac{1}{4}$. 
Then, for 
$p\in \big[ 2, 
 \sqrt{L(b)}/ \big(  2^5\, 3 \, C_\mu \, L(\sigma)^2\big) \big]$ we have
\beqn
C_2  \le c_2(T) \,\big( p\, C_\mu^{(\gamma)} \big)^{\frac{1}{2}} \, \Big[ c(\sigma) +  L(\sigma)\,  e^{2T\sqrt{L(b)}}\mathcal{N}_{2\sqrt{L(b)},p}(u)\Big],
\eeqn
with $\mathcal{N}_{2\sqrt{L(b)},p}(u)$ satisfying \eqref{p-s3.1-e1}. \end{prop}
\begin{proof} 
(i).  We first prove
\begin{align}
\label{s3-3.370}
\sup_{|z_1-z_2|\le|x-\bar x|}\Vert u(t,z_1) - &u(t,z_2)\Vert_p^2 \le C_1^2 |x-\bar x|^{2(\gamma_1\wedge \gamma_2)}+ C_2^2 |x-\bar x|^{2\gamma}
\notag\\
& \quad 
+ C_3\! \int_0^t \! ds \sup_{|z_1-z_2| \le |x-\bar x|}\Vert u(s,z_1) - u(s,z_2)\Vert_p^2.
\end{align}
Indeed, using \eqref{sol-1} and \eqref{decom}, the first term on the right-hand side comes from  \eqref{s3-3.02} and \eqref{s3-3.07}. 
The second one comes from the last term on the right-hand side of \eqref{diff-I2-x-d}.
Finally, the very last term is obtained by the sum of the upper bound \eqref{s3-3.15} 
and the first term on the right-hand side of  \eqref{diff-I2-x-d}. 

Apply Gronwall's lemma to the function 
 $ t\mapsto \sup_{|z_1-z_2|\le |x-\bar x|}\Vert u(t,z_1) - u(t, z_2)\Vert_p^2$ to obtain \eqref{s3-3.37}, and then 
\eqref{s3-3.38}.

 The claim (ii) follows form  the definition of $\mathcal{N}_{2\sqrt{L(b)},p}(u)$
 and Proposition \ref{p-s3.1}.
\end{proof}
\noindent{\em Time increments of $I_2(t,x)$}
\smallskip

In order to deduce $L^p$-estimates of increments in time of the stochastic integral term  $I_2(t,x)$, additional assumptions on the covariance of the noise are needed.
\smallskip

\noindent{\bf(h3)}  The spectral measure $\mu$ is such that there exists $\nu>0$ and $C>0$ for which
\beq
\label{h3}
\int_{\rde} \left\vert\cF G(t)(\zeta)\right\vert^2\ \mu(d\zeta)\le C t^\nu,\  {\text{for any}}\ t\in[0,T].
\eeq

\noindent{\bf(h4)} The covariance density function $f$ satisfies the following conditions:
\begin{enumerate}
\item  There exists $b>0$ and $C>0$  such that for any $h\in[0,T]$, 
\beq
\label{h41} 
\int_0^T\!\! ds\  s \int_{\rde} \!\! \int_{\rde}G(1,dy) G(1,dz)
\big\vert f(s(y+z)+h(y+z)) - f(s(y+z)+hz)\big\vert
\le C h^b.
\eeq
\item There exists $\bar b>0$ and $C>0$ such that for any $h\in[0,T]$, 
\begin{align}
\label{h42} 
&\int_0^T ds\ s^2 \int_{\rde} \int_{\rde} G(1,dy) G(1,dz)\notag \\
& \; \times  \big\vert f\big(s(y+z)+h(y+z)\big) 
 - f\big(s(y+z)+hy\big ) - f\big( s(y+z)+hz\big) 
+ f\big(s(y+z)\big)\big\vert \notag\\
&
\le C h^{\bar b}.
\end{align}
\end{enumerate}
\medskip

According to \eqref{s3.11Bis}--\eqref{cmu}, the left-hand side of \eqref{h3} is a function of $t$ uniformly bounded over bounded intervals. Assumption {\bf(h3)} provides a growth rate for this function.
\smallskip

Up to scalings, the assumption {\bf(h4)} is on estimates of one and two-dimensional increments of the covariance density in a $L^2$-type norm. We shall give in Section \ref{s4} examples where these conditions are satisfied.

\begin{prop}
\label{s3-3-p5}
Assume that the hypotheses (1)--(3) of Proposition \ref{s3-3-p4} hold.
Suppose also that
the hypotheses  {\bf(h3)} and  {\bf(h4)} on the covariance of the noise are satisfied. Then there exists a constant $C(T,\nu)$ 
such that for any $p\in[2,\infty)$, $t, \bar t\in[0,T]$ and  $x\in\rde$,
\begin{align}
\label{s3-3.680}
\Vert  & I_2(t,x) - 
 I_2(\bar t,x)\Vert_p^2 \leq C(T,\nu)\,  p\,  \Big( C_\mu L(\sigma)^2 \tilde C^2  |t-\bar t|^{2 \nu_1} \notag\\
& 
+ \Big[c(\sigma) + L(\sigma) \sup_{(t,x)\in[0,T]\times \rde}\Vert u(t,x)\Vert_p\Big]^2 \big\{ |t-\bar t|^{1+ \nu} 
+  \, |t-\bar t|^{\min( b+1,\bar b,\tilde \alpha)}  \big\}\notag\\
&
+ L(\sigma) \tilde C\Big[ c(\sigma) + L(\sigma)  \sup_{(t,x)\in[0,T]\times \rde}\Vert u(t,x)\Vert_p\Big]  |t-\bar t|^{\nu_1+\min(b,1)}\Big) , 
\end{align}
where  $\nu_1=\min(\gamma,\gamma_1,\gamma_2)$,   $\tilde C$ is defined in Proposition \ref{s3-3-p4},  
$\tilde \alpha =(1+\nu)\wedge 2$  if $\nu\ne 1$ and $\tilde \alpha <2$ if $\nu=1$. 
\medskip

If, as in Proposition \ref{p-s3.1}, the Lipschitz constants $L(b)$, $L(\sigma)$ are such that
$L(b)\geq \left( 2^{12}\, 3^2\,  C_\mu^2\,  L(\sigma)^4\right)\vee \frac{1}{4}$,  where $C_\mu$ is given in \eqref{cmu}, 
then there exists a constant $\overline{C(\nu,T)}$ such that  for 
$p\in \Big[ 2, \frac{  \sqrt{L(b)}}{ 2^5\, 3 \, C_\mu \, L(\sigma)^2}\Big]$, $t,\bar{t}\in [0,T]$ and $x\in \RR^d$,
\begin{align}
\label{s3-3.6800}
\Vert   I_2&(t,x) - 
 I_2(\bar t,x)\Vert_p^2 \leq \overline{C(\nu,T)}\, p\,  \Big(  C(T) C_\mu  L(\sigma)^2 \tilde C^2  |t-\bar t|^{2 \nu_1} \notag\\
& 
+ \Big[c(\sigma) + L(\sigma) e^{2T\sqrt{L(b)} }{\mathcal N}_{2\sqrt{L(b)},p}(u)\Big]^2 
\big\{ |t-\bar t|^{1+ \nu} + |t-\bar t|^{\min(b+1,\bar b,\tilde \alpha)}  \big\}\notag\\
&
+  L(\sigma) \tilde C\Big[ c(\sigma) + 
L(\sigma)  e^{2T\sqrt{L(b)} }{\mathcal N}_{2\sqrt{L(b)},p}(u)\Big] |t-\bar t|^{\nu_1+\min(b,1)}\Big) , 
\end{align}
with $\mathcal{N}_{2\sqrt{L(b)},p}(u)$ satisfying \eqref{p-s3.1-e1}. 

\end{prop}
\begin{proof}
For $0\le t\le \bar t\le T$ and $x\in \RR^d$, set
\beqn
I_{2,1}(t,\bar t;x)=\int_t^{\bar t} \int_{\rde} G(\bar t-s,x-dy) \sigma(u(s,y))\ W(ds,dy).
\eeqn
By applying  Burkholder-Davis-Gundy's inequality, and then Minkowski's inequality, 
\begin{align*}
\Vert I_{2,1}(t,\bar t;x)\Vert_p^2 
&\le 4p \int_t^{\bar t} ds J(\bar t-s) \sup_{y\in\rde}\Vert \sigma(u(s,y))\Vert_p^2\\
&\le p\ C |t-\bar t|^{1+\nu}\Big[c(\sigma) + L(\sigma)\sup_{(s,y)\in[0,T]\times \rde}\Vert u(s,y)\Vert_p\Big]^2,
\end{align*}
where $J$ is defined in \eqref{s3.11Bis}, and the last upper estimate is deduced from {\bf(h3)} (see \eqref{h3}).

Let 
\begin{equation} \label{I22}
I_{2,2}(t,\bar t; x) = \int_0^t \int_{\rde}
 [G(\bar t -s, x-dy) - G(t-s, x-dy)]\ \sigma(u(s,y))\ W(ds,dy).
 \end{equation}
 
 We study the $L_{p}$-norm of this term following the proof of \cite[Theorem 4.1]{HHN}. This uses the transfer of increments trick 
 introduced in \cite[Section 3.2]{DSS-Memoirs}. Applying the Burkholder-Davies-Gundy inequality, we obtain 
$ \left\Vert I_{2,2}(t,\bar t; x)\right\Vert_p^2 \le 4p\sum_{i=1}^4\left\Vert R_i(t,\bar t; x)\right\Vert_{\frac{p}{2}}$,
 where, letting $h:=\bar t - t$ and $\Theta_{t,x}(s,y) = \sigma(u(t-s, x-y))$, we set 
 \begin{align*}
 R_1(t,\bar t; x) &= \int_0^t ds\int_{\rde}\int_{\rde} G(1,dy) G(1, dz)(s+h)^2 f\big( (s+h)y-(s+h)z\big)\\
 & \quad\times \big[\Theta_{t,x}(s,(s+h)y) - \Theta_{t,x}(s,sy)\big] \big[\Theta_{t,x}(s,(s+h)z) - \Theta_{t,x}(s,sz)\big],\\
 R_2(t,\bar t; x) &= \int_0^t ds\int_{\rde}\int_{\rde} G(1,dy) G(1, dz)\\
 & \quad\times \big[ (s+h)^2 f\big( (s+h)y-(s+h)z\big) - s(s+h)f\big( sy-(s+h)z\big)\big]\\
 & \quad\times \big[\Theta_{t,x}(s,(s+h)z) - \Theta_{t,x}(s,sz)\big] \Theta_{t,x}(s,sy),\\
 R_3(t,\bar t; x) &= \int_0^t ds\int_{\rde}\int_{\rde} G(1,dy) G(1, dz)\\
 & \quad\times \big[(s+h)^2 f\big( (s+h)y-(s+h)z\big) - s(s+h)f\big( (s+h)y-sz\big)\big]\\
 & \quad\times \big[\Theta_{t,x}(s,(s+h)y) - \Theta_{t,x}(s,sy)\big] \Theta_{t,x}(s,sz),\\
 R_4(t,\bar t; x) &= \int_0^t ds\int_{\rde}\int_{\rde} G(1,dy) G(1, dz)\\
 & \quad\times \big[(s+h)^2 f\big((s+h)y - (s+h)z\big) - s(s+h)f\big(sy-(s+h)z\big) \\
 & \quad\qquad - s(s+h) f\big( (s+h)y - sz\big) + s^2 f\big(sy-sz\big)\big]\\
 & \quad\times \Theta_{t,x}(s,sy)\Theta_{t,x}(s,sz).
 \end{align*}
 Notice that the linear growth and Lipschitz continuity assumptions on $\sigma$ imply that for any $p\in[2,\infty)$, every $s,t\in[0,T]$ and $x,y,z\in\rde$,
  \beq
  \label{s3-3.590}
   \sup_{0\le s\le t\le T; (x,y)\in\rde} \Vert\Theta_{t,x}(s,y)\Vert_p \le c(\sigma) + L(\sigma)\sup_{(t,x)\in[0,T]\times\rde} \Vert u(t,x)\Vert_p,
   \eeq
   and
   \begin{align}
   \label{s3-3.591}
  \sup_{0\le s\le t\le T; (x,y,z)\in\rde,|y-z|\leq r} \!\!\Vert\Theta_{t,x}(s,y) - \Theta_{t,x}(s,z)\Vert_p 
  &\le  L(\sigma) \!
  \sup_{t\in[0,T], |y-z|\leq r} \!\!\Vert u(t,y)-u(t,z)\Vert_p\notag\\
  &
  \le L(\sigma) \tilde C r^{\nu_1},
  \end{align}
  where the last inequality follows from \eqref{s3-3.38}.
  \medskip
  
 \noindent{\em Upper bound of  $\Vert R_1(t,\bar t; x)\Vert_{\frac{p}{2}}$.}
 Apply the Minkowski and Cauchy-Schwarz inequalities. Then, using  \eqref{s3-3.591} we obtain 
 \begin{align}
  \label{s3-3.60}
 \Vert R_1(t,\bar t; x) &\Vert_{\frac{p}{2}} \le L(\sigma)^2 \tilde C^2 |t-\bar t|^{2 \nu_1}\notag\\
 & \quad \times\int_0^t ds \int_{\rde}\int_{\rde} G(1,dy) G(1, dz) (s+h)^2f\big( (s+h)y-(s+h)z \big).
 \end{align}
 Consider the change of variables $((s+h)y, (s+h)z)\mapsto (y,z)$; using the scaling property, 
  \eqref{s3.11Bis} and \eqref{cmu}, we deduce 
 \begin{align}
 \label{s3-3.61}
  &\int_{\rde}\int_{\rde} G(1,dy) G(1, dz) (s+h)^2f\big( (s+h)y-(s+h)z \big)\notag\\ 
  &\qquad \qquad= (2\pi)^{-d} \int_{\rde} \vert\cF G(s+h)(\zeta)\vert^2\ \mu(d\zeta) \le 2(1+(2T)^2)) C_\mu. 
  \end{align}
  Hence, \eqref{s3-3.60} and \eqref{s3-3.61} imply
  \beq
  \label{s3-3.62}
  \Vert R_1(t,\bar t; x)\Vert_{\frac{p}{2}} \le 2\, \big[ 1+(2T)^2 \big] \,  C_\mu L(\sigma)^2 \tilde C^2 |t-\bar t|^{2 \nu_1},
 \eeq
 where $\tilde C$ is defined in  Proposition \ref{s3-3-p4}.
  \bigskip
  
  \noindent{\em Upper bound of  $\Vert R_2(t,\bar t; x)\Vert_{\frac{p}{2}}$ and $\Vert R_3(t,\bar t; x)\Vert_{\frac{p}{2}}$.}
 We will only consider $\Vert R_2(t,\bar t; x)\Vert_{\frac{p}{2}}$, since $\Vert R_3(t,\bar t; x)\Vert_{\frac{p}{2}}$ is similar.
 Set 
 \begin{align*}
R_{2,1}(t,\bar t; x) &= \int_0^t ds\int_{\rde}\int_{\rde} G(1,dy) G(1, dz)\\
&\qquad \times s(s+h)\big[f\big( (s+h)y-(s+h)z\big) - f\big( sy-(s+h)z\big)\big]\\
&\qquad \times \big[\Theta_{t,x}(s,(s+h)z) - \Theta_{t,x}(s,sz)\big] \Theta_{t,x}(s,sy),
\end{align*}
Apply the change of variable $z\mapsto -z$ along with the Minkowski and Cauchy-Schwarz inequalities to obtain
\begin{align}
\label{s3-3.63}
 &\Vert R_{2,1}(t,\bar t; x)\Vert_{\frac{p}{2}} 
  \le \sup_{0\le s\le t\le T; (x,z_1,z_2)\in\rde, |z_1-z_2| \le h} \left(\Vert \Theta_{t,x}(s,z_1)-\Theta_{t,x}(s,z_2)\Vert_p\right) \notag \\
 &\quad \times  \sup_{0\le s\le t\le T; (x,y)\in\rde}\left(\Vert \Theta_{t,x}(s,y)\Vert_p\right)\notag\\
&\quad \times 
 \int_0^t ds\! \int_{\rde}\! \int_{\rde}\!  G(1,dy) G(1, dz)
 s(s+h) 
\big| f\big( (s(y+z)+h(y+z)\big) - f\big( s(y+z)+hz\big)\big| \notag\\
 &\; \le C T L(\sigma) \tilde C |t-\bar t|^{\nu_1+b}
 \Big[c(\sigma) + L(\sigma)  \sup_{(t,x)\in[0,T]\times \rde}\Vert u(t,x)\Vert_p\Big],
 \end{align}
 where we have used \eqref{s3-3.590},  \eqref{s3-3.591} and assumption {\bf (h4)} (see \eqref{h41}). 
 \smallskip
 
 Define
  \begin{align*}
R_{2,2}(t,\bar t; x) = \int_0^t ds\int_{\rde}\int_{\rde}& G(1,dy) G(1, dz)
 h(s+h)  f\big( (s+h)y-(s+h)z\big)\\
& \times\big[\Theta_{t,x}(s,(s+h)z) - \Theta_{t,x}(s,sz)\big] \Theta_{t,x}(s,sy),
\end{align*}
A computation similar to that  used to upper estimate $\| R_{2,1}(t,\bar t; x)\|_p$ implies 
\begin{align*}
 &\Vert R_{2,2}(t,\bar t; x)\Vert_{\frac{p}{2}} 
 \le C L(\sigma)\tilde C |t-\bar t|^{\nu_1} \Big[c(\sigma) + L(\sigma)  \sup_{(t,x)\in[0,T]\times \rde}\Vert u(t,x)\Vert_p\Big]\\
 &\qquad \qquad  \times  \int_0^t ds\int_{\rde}\int_{\rde} G(1,dy) G(1, dz) h(s+h)\big| f\big( (s+h)y-(s+h)z\big)\big| .
\end{align*}
Using the change of variables $((s+h)y,(s+h)z)\mapsto (y,z)$, the scaling property, \eqref{s3.11Bis} 
and {\bf (h3)}, we obtain 
\begin{align*}
 &\int_0^t ds\int_{\rde}\int_{\rde} G(1,dy) G(1, dz) h(s+h)\left[f((s+h)y-(s+h)z)\right]\\
&\quad  = (2\pi)^{-d} h \int_0^t \frac{ds}{s+h}\int_{\rde} \mu(d\zeta) \vert \cF G(s+h)(\zeta)\vert^2 
\le C \ h\int_0^t ds\ (s+h)^{\nu-1} \le C\, T^\nu \ h.
 \end{align*}
 Thus,
 \beq
 \label{s3-3.64}
  \Vert R_{2,2}(t,\bar t; x)\Vert_{\frac{p}{2}}
   \le C T^\nu  L(\sigma)\tilde C |t-\bar t|^{\nu_1+1} \Big[c(\sigma) + L(\sigma)  \sup_{(t,x)\in[0,T]\times \rde}\Vert u(t,x)\Vert_p\Big].
\eeq
Since $R_{2}(t,\bar t; x)=R_{2,1}(t,\bar t; x)+R_{2,2}(t,\bar t; x)$, from \eqref{s3-3.63} and \eqref{s3-3.64}, we deduce
 \begin{align}
 \label{s3-3.65}
  \Vert R_{2}(t,\bar t; x)\Vert_{\frac{p}{2}}
   \le & \; C (T+T^\nu) \, L(\sigma)\tilde C |t-\bar t|^{\nu_1+\min(b,1)} 
     \Big[c(\sigma) + L(\sigma)  \sup_{(t,x)\in[0,T]\times \rde}\Vert u(t,x)\Vert_p\Big].
\end{align}
\medskip

\noindent{\em Upper bound of  $\Vert R_4(t,\bar t; x)\Vert_{\frac{p}{2}}$.}
 Using Minkowski's inequality and  \eqref{s3-3.590}, we obtain
 \begin{align*}
&  \Vert R_{4}(t,\bar t; x)\Vert_{\frac{p}{2}}
  \le  \Big[c(\sigma) + L(\sigma)  \sup_{(t,x)\in[0,T]\times \rde}\Vert u(t,x)\Vert_p\Big]^2  \, I(t,h), 
 \end{align*}
where 
 \begin{align*}
I(t,h)=& \int_0^t ds\int_{\rde}\int_{\rde} G(1,dy) G(1, dz)
\Big| (s+h)^2 f\big( (s+h)(y-z)  \big) \\
&\quad - s(s+h)f\big(sy-(s+h)z\big)- s(s+h) f\big((s+h)y - sz\big) + s^2 f\big(s(y-z) \big)\Big| .
 \end{align*}
Use the  change of variable $z\mapsto -z$ to see that  $I(t,h)=\sum_{j=1}^4 \tilde{I}_j(t,h)$, with
 \begin{align*}
 \tilde{I}_1(t,h)= & \int_0^t ds\int_{\rde}\int_{\rde} G(1,dy) G(1, dz)s^2 \\
 &\times\big\vert f\big((s+h)(y+z) \big) - f\big(sy+ (s+h)z\big) 
 -f\big((s+h)y+sz\big)+f\big(s(y+z)\big)\big\vert,\\
 \tilde{I}_2(t,h)= & \int_0^t ds\int_{\rde}\int_{\rde} G(1,dy) G(1, dz)\ sh\\
 &\times\big\vert f\big( (s+h)y+(s+h) z)  \big)-f\big (sy-(s+h)z\big)\big\vert,\\
\tilde{ I}_3(t,h)=&  \int_0^t ds\int_{\rde}\int_{\rde} G(1,dy) G(1, dz)\ sh\\
 &\times\big\vert f\big((s+h)y + (s+h)z\big)-f\big((s+h)y+ sz\big)\big\vert,\\
 \tilde{I}_4(t,h)=&  \int_0^t ds\int_{\rde}\int_{\rde} G(1,dy) G(1, dz)\ h^2
 \big\vert f\big((s+h)y + (s+h)z \big)\big\vert.
\end{align*}
The hypothesis {\bf(h4)} implies $\tilde{I}_1(t,h) \le C\ h^{\bar b}$ and $\tilde{I}_2(t,h)+\tilde{I}_3(t,h) \le C\ h^{b+1}$,
(see \eqref{h42} and \eqref{h41}, respectively).  As for $\tilde{I}_4(t,h)$, we apply the change of variables
$((s+h)y,(s+h)z)\mapsto (y,z)$, the scaling property, \eqref{s3.11Bis} and  {\bf (h3)}; this yields 
\begin{align*}
\tilde{I}_4(t,h)& =  \int_0^t\!  \frac{ds}{(s+h)^2} h^2 \int_{\rde}\!\int_{\rde}\! G(s+h, dy) G(s+h, dz) f(y-z) 
\le C h^2 \int_0^t (s+h)^{\nu-2} ds.
\end{align*}
For $h\in (0,T]$ and $\varepsilon>0$  arbitrarily small, this yields  that,  up to some multiplicative constant,  $\tilde{I}_4(t,h)$
is upper estimated by $h^2 T^{\nu -1}$ if $\nu>1$ (respectively  by $h^{\nu+1}$ if $\nu<1$, and by $T^\epsilon h^{2-\varepsilon}$ if $\nu=1$).
Summarising the estimates above, we obtain 
\begin{align}
\label{s3-3.68}
\Vert& I_2(t,x)  - I_2(\bar t,x)\Vert_p^2 \le 2\left(\Vert I_{2,1}(t,\bar t;x)\Vert_p^2 +  \Vert I_{2,2}(t,\bar t,x)\Vert_p^2\right)\notag\\
&\le Cp  \Bigg( \Big[c(\sigma) + L(\sigma)  \sup_{(t,x)\in[0,T]\times \rde}\Vert u(t,x)\Vert_p\Big]^2 |t-\bar t|^{1+\nu} 
+ C_\mu L(\sigma)^2 \tilde C^2 |t-\bar t|^{2 \nu_1} \notag\\
&\quad + (T+T^\nu) L(\sigma) \tilde C\Big[c(\sigma) + L(\sigma)
 \sup_{(t,x)\in[0,T]\times \rde}\Vert u(t,x)\Vert_p \Big]|t-\bar t|^{\nu_1+\min(b,1)}\notag\\
&\quad + \tilde{C}(\nu,T) \Big[c(\sigma) + L(\sigma)  \sup_{(t,x)\in[0,T]\times \rde}\Vert u(t,x)\Vert_p\Big]^2|t-\bar t|^{\min(b+1,\bar b,\tilde \alpha)}
\Bigg),
\end{align}
where $\tilde \alpha =(1+\nu)\wedge 2$ 
if $\nu\ne 1$, while  $\tilde \alpha <2$   if $\nu=1$, and $\tilde{C}(\nu,T)$ is a positive constant. 
This completes the proof of \eqref{s3-3.680}. 

From \eqref{s3-3.680},  using Proposition \ref{p-s3.1}, we deduce \eqref{s3-3.6800}. This concludes the proof. 
\end{proof}

From Propositions \ref{s3-3-p1}--\ref{s3-3-p5} we  deduce Theorem \ref{s3-3-p6} below, 
which is the main ingredient towards obtaining uniform bounds on moments. 

The following constants $\nu_1$ and $\nu_2$ will be used in the next Theorem: 
\beq
\label{numbers}
\nu_1 = \min(\gamma, \gamma_1, \gamma_2), \quad 
\nu_2 = \min\left(\nu_1, \frac{1}{2}[\nu_1+ \min(b,1)] , \frac{1+\nu}{2}, \frac{b+1}{2}, \frac{\bar b}{2}, \frac{\tilde\alpha}{2}\right).
\eeq
We recall that  $\gamma_1$, $\gamma_2$, are the H\"older exponents of the initial values (see Proposition \ref{s3-3-p1}), $\gamma$ is the parameter in the assumption {\bf(h2)}, $\nu$ is defined in {\bf(h3)}, $b$ and $\bar b$ in  {\bf(h4)}, and $\tilde \alpha$ in the last part of the proof of Proposition \ref{s3-3-p5}. 
\smallskip

Let
\beq
\label{izeroplus}
\bar K_0(u_0, v_0) = \begin{cases}
\Vert v_0\Vert_{\gamma_2} +  \Vert \nabla u_0\Vert_{\infty,R+T} + \Vert \nabla u_0\Vert_{\gamma_1} + \Vert v_0\Vert_{\infty,R+T} , & d=2,\\
\Vert v_0\Vert_{\gamma_2} + \Vert \nabla u_0\Vert_{\infty,R+T} + \Vert\Delta u_0\Vert_{\gamma_1}, & d=3.
\end{cases}
\eeq
Comparing this definition with \eqref{izero}, we see that $K_0(u_0,v_0)\le \bar K_0(u_0, v_0)$.
\begin{theorem}
\label{s3-3-p6}
Suppose that the hypotheses (1)--(3) of Proposition \ref{s3-3-p4} hold, 
and that
the  conditions  {\bf(h3)} and {\bf(h4)} on the covariance of the noise are satisfied. Fix $T,R>0$. Then the following holds.

1. For any $p\in[2,\infty)$, there exists a constant $C(p,T,R)$ such that, for any 
$t, \bar t\in[0,T]$, $x, \bar x \in B(0; R)$ and $\alpha>0$, 
\begin{align}
\label{s3-3.72bis}
\frac{\Vert u(t,x) - u(\bar t, \bar x)\Vert_p}{|x-\bar x|^{\nu_1} + |t-\bar t|^{\nu_2}}
& \le C(p,T,R) \Big[ {\mathcal M}_1 + {\mathcal M}_2 + {\mathcal M}_3\ e^{T\alpha  } {\mathcal N}_{\alpha,p}(u)\Big],
\end{align}
 where
\begin{align}
\label{defmss}
\mathcal{M}_1 &= \bar K_0(u_0,v_0)\Big\{1+\Big[ L(b) + \sqrt p  \big(1+\sqrt{C_\mu}\, \big) 
 L(\sigma)\Big] \exp\Big(\frac{TC_3}{2}\Big)\Big\},\notag\\
\mathcal{M}_2 &= c(b) +\sqrt p c(\sigma) 
\Big\{1+(C_\mu^{(\gamma)})^{1/2}\left[ L(b) + \sqrt p \big(1+\sqrt{C_\mu}\, \big)  L(\sigma)\right] 
\exp\Big(\frac{TC_3}{2}\Big)\Big\},\notag\\
\mathcal{M}_3 & = \left[L(b) + \sqrt p  \big(1+\sqrt{C_\mu}\, \big)  L(\sigma)\right] 
\Big\{1+(p C_\mu^{(\gamma)})^{1/2}L(\sigma) 
\exp\Big(\frac{TC_3}{2}\Big)\Big\},
\end{align}
with $\bar K_0(u_0,v_0)$ and $C_3$ given in \eqref{izeroplus} and \eqref{ccs}, respectively. 
\medskip

2. Suppose further that $L(b)$ and $L(\sigma)$ do not vanish, and $L(b)\geq \big( 2^{12} 3^2 C_\mu^2 L(\sigma)^4\big)\vee \frac{1}{4}$.
 Then, for any $p\in \Big[2,\frac{\sqrt{L(b)}}{2^5 3 C_\mu L(\sigma)^2}\Big]$ and 
 $\mathcal{T}_0$ defined in \eqref{p-s3.1-e2}, we have 
 \begin{align}
\label{s3-3.72tris}
&\frac{\Vert u(t,x) - u(\bar t, \bar x)\Vert_p}{|x-\bar x|^{\nu_1} + |t-\bar t|^{\nu_2}}
 \le C(p,T,R)\notag \\
&\qquad \qquad \times \left[ {\mathcal M}_1 + {\mathcal M}_2 + {\mathcal M}_3\ e^{2 T\sqrt{L(b)}} 
\left(\mathcal{T}_0 + \frac{c(b)}{L(b)} + \frac{c(\sigma)}{L(\sigma)}\right)
\right].
\end{align} 
\end{theorem} 
\begin{proof}
Fix $x\in B(0;R)$ and consider the time increment $\Vert u(t,x)-u(\bar t,x)\Vert_p$, with $t,\bar{t}\in [0,T]$. 
Using the  estimates \eqref{s3-3.01}, \eqref{s3-3.05}  for the increments of $I_0$ in dimension $d= 2, 3$, respectively, 
 then \eqref{s3-3.17}, \eqref{s3-3.18} and the definition of \eqref{norm} for the  increments of $I_1$,
  and finally \eqref{s3-3.6800}  
 for the increments of $I_2$, we obtain
\begin{align}
\label{incrtime}
\Vert u(t,x)&-u(\bar t,x)\Vert_p \le C(T,R)\Big\{ \bar K_0(u_0,v_0) |t-\bar t|^{\min(\gamma_1, \gamma_2)} \notag\\
&
+\left[c(b)+L(b)  e^{T\alpha}  {\mathcal N}_{\alpha,p}(u) \right] |t-\bar t|\notag\\
& + \tilde{C} \left( \left[L(b) + \sqrt{p} \sqrt{C_\mu}  L(\sigma)\right] |t-\bar t|^{\nu_1} 
+  \sqrt{p} L(\sigma)|t-\bar t|^{\frac{1}{2} [\nu_1 + \min(b,1)]}\right)\notag\\
& +  \sqrt{p}\left[c(\sigma) + L(\sigma)  e^{T\alpha } {\mathcal N}_{\alpha,p}(u)  \right]\notag\\
&\qquad\times  
\left[|t-\bar{t}|^{\frac{1}{2} [\nu_1 + \min(b,1)]} + |t-\bar{t}|^{\frac{1+\nu}{2}} + |t-\bar{t}|^{\frac{1}{2}\min(b+1,\bar{b},\tilde{\alpha})}\right]
\Big\},
\end{align}
where the constant $\tilde{C}$ is the same as in \eqref{s3-3.38}.  Here we
have applied  the inequality $\sqrt{AB}\le \frac{1}{2}(A+B)$ to the product of constants
$A:= \tilde C L(\sigma)$ and $B:=\Big[ c(\sigma) + L(\sigma)  e^{T\alpha} {\mathcal N}_{\alpha,p}(u)\Big]$
in the last line of \eqref{s3-3.6800} (with $\alpha$ instead of $2\sqrt{L(b)}$ ). 

Since by \eqref{s3-3.38} we have $\sup_{t\in [0,T]} \|u(t,x)-u(t,\bar{x})\|_p \leq \tilde{C} |x-\bar{x}|^{\nu_1}$ for $x,\bar{x}\in B(0;R)$, we deduce that the $L^p$ norm or space-time increment $\Vert u(t,x) - u(\bar t, \bar x)\Vert_p$ is bounded from above by the sum of the left-hand side of \eqref{incrtime} and   $\tilde C |x-\bar x|^{\nu_1}$. Using the definition of $\tilde C$  and grouping terms, we obtain the inequality \eqref{s3-3.72bis}.

Part 2. follows from Proposition \ref{p-s3.1} (see  \eqref{p-s3.1-e1}). This concludes the proof.
\end{proof}
\smallskip

With an  approach similar to that used in Section \ref{s2}, from part 2. of Theorem \ref{s3-3-p6}  and the Kolmogorov continuity
 lemma (\cite[Theorem A.3.1]{Da-SS-Book}), we deduce the
uniform $L^p$ moment estimates stated in the next Proposition. 
They are essential in the proof  of existence and uniqueness 
of a global random field solution to \eqref{n3}. 
Set 
\beq
\label{KinKolmogorov}
\mathcal{K}(c(b), c(\sigma), L(b), L(\sigma)) = {\mathcal M}_1 + {\mathcal M}_2 + {\mathcal M}_3 e^{2T\sqrt{L(b)}} 
\left(\mathcal{T}_0 + \frac{c(b)}{L(b)} + \frac{c(\sigma)}{L(\sigma)}\right),
\eeq
where ${\mathcal M}_j$, $j=1,2,3$ and $\mathcal{T}_0$ are defined in \eqref{defmss} and  \eqref{p-s3.1-e2}, respectively. 
Observe that, up to a constant factor, $\mathcal{K}(c(b), c(\sigma), L(b), L(\sigma))$ equals the right-hand side
 of \eqref{s3-3.72tris}. 
\begin{prop}
\label{Kol-d}
Suppose that the hypotheses (1)--(3) of Proposition \ref{s3-3-p4} hold,
and the hypotheses {\bf(h3)} and  {\bf(h4)} on the covariance of the noise are satisfied. Let 
$\nu_1$ and $\nu_2$ be the parameters defined in \eqref{numbers}. 
Suppose that the Lipschitz coefficients $L(b)$ and $L(\sigma)$ are non null and satisfy 
$L(b)\geq \big( 2^{12} 3^2 C_\mu^2 L(\sigma)^4\big)\vee \frac{1}{4}$ and
\begin{equation} 
\label{Lipschitz-d}
\frac{\sqrt{L(b)}}{2^5 3 C_\mu L(\sigma)^2} >  \frac{1}{\nu_1}+\frac{d}{\nu_2}, \quad d= 2,3.
\end{equation}
Fix $T, R >0$. Then, for any $p\in \left(\frac{1}{\nu_1}+\frac{d}{\nu_2}, 
\sqrt{L(b)} / \big( 2^5 3 C_\mu L(\sigma)^2\big) \right]$, there exists  positive constants  $C_1$ and $C_2(p,T,R)$  such that  
\begin{equation} 
\label{s3-3-maxp}
E\Big( \sup_{(t,x)\in [0,T]\times B(0;R)} |u(t,x)|^p\Big) \leq 2^{p-1} C_1 + C_2(p,T,R) \mathcal{K}(c(b), c(\sigma), L(b), L(\sigma)),
\end{equation}
with $\mathcal{K}(c(b), c(\sigma), L(b), L(\sigma))$ defined in \eqref{KinKolmogorov}.
\end{prop} 

The proof is analogous to that of Proposition \ref{Kol-1d}; it  is omitted.

\subsection{Existence and uniqueness of a global solution} \label{Wp-d}
In this section, we consider the equation \eqref{n3} in spatial dimensions $d =2, 3$. We assume that the coefficients $b$ and $\sigma$ 
satisfy the hypothesis {\bf (Cs)} of Section \ref{s2-3}, thereby having superlinear growth. We also assume that $b$ dominates $\sigma$, in the terms expressed by the
condition {\bf(Cd)} below.
\medskip

\noindent {\bf (Cd)} The parameters $\delta$ and $a$ in \eqref{super-coeff} satisfy one of the properties:

(1) $\delta > 4a$;

(2) $\delta = 4a$ and  $\theta_2$ and $\sigma_2$ are such that 
$\theta_2 > 2^{12} 3^2 C_\mu^2 \sigma_2^4 \big( \frac{1}{\nu_1} + \frac{d}{\nu_2}\big)^2$, $d=2,3$,

\noindent where $C_\mu$ is defined in \eqref{cmu} and $\nu_1$, $\nu_2$ are given in \eqref{numbers}. 
\medskip

The next theorem 
is the main result of this section.
\begin{theorem}
\label{wp-d}
The hypothesis are as follows.
\begin{itemize}
\item{(i)}The initial values $u_0$ and $v_0$ are functions satisfying the hypotheses of Proposition \ref{s3-3-p1} with some H\"older exponents $\gamma_1, \gamma_2\in(0,1)$.
\item {(ii)} The coefficients $b$ and $\sigma$ satisfy {\bf (Cs)} 
and {\bf (Cd)}  with $\delta < \frac{1}{2}$. 
\item {(iii)} The covariance of the noise satisfies conditions  {\bf(h1)}, {\bf(h2)}, {\bf(h3)} and {\bf(h4)}.
\end{itemize}
\begin{enumerate}
\item For any $M>0$, there exists a random field solution to \eqref{n3} in $B(0; M)$, $\big(u(t,x),\break (t,x)\in[0,T]\times B(0; M))$. This solution is unique and satisfies
\beq
\label{bd-as}
\sup_{(t,x)\in[0,T]\times B(0; M)}\vert u(t,x)\vert < \infty, \ a.s.
\eeq
\item Suppose that the initial conditions $u_0$, $v_0$ are functions with compact support included in $B(0; \rho)$,  
for some $\rho>0$, and $b(0)=\sigma(0)=0$. Then there exists a random field solution $\big(u(t,x),\ (t,x)\in[0,T]\times \RR)$ to \eqref{n3}.
This solution is unique and satisfies
\beq
\label{bd-as-bis}
\sup_{(t,x)\in[0,T]\times B(0,\rho+T)}\vert u(t,x)\vert < \infty, \ a.s.
\eeq
Equivalently,
\beq
\label{bd-as-bis-bis}
\sup_{(t,x)\in[0,T]\times \RR^d}\vert u(t,x)\vert < \infty, \ a.s.
\eeq
\end{enumerate}

\end{theorem}
\begin{proof}
We refer to Section \ref{s-app} for details on the settings of the claims. 
The proof uses the same approach as in the proof of Theorem \ref{wp-d=1}. First, for $g=b, \sigma$, we consider the truncated globally Lipschitz functions $b_N$, $\sigma_N$, defined in \eqref{b_N}. The assumption {\bf(Cs)} imply that \eqref{s2-3.1} holds. Moreover, by {\bf(Cd)}, we see that the Lipschitz coefficients $L(b_N)$, $L(\sigma_N)$ satisfy the hypotheses of Proposition \ref{Kol-d}.

Let $u_N=\left(u_N(t,x), (t,x)\in[0,T]\times\RR^d\right)$ be the unique global random field solution to \eqref{n3} with coefficients $b_N$, $\sigma_N$. 
Fix $R>0$. Under the standing hypotheses, we can apply Proposition \ref{Kol-d} to the stochastic process $u_N$ to deduce that, for any 
$p\in \Big(  \frac{1}{\nu_1}+\frac{d}{\nu_2}, \frac{\sqrt{L(b_N)}}{2^5 3 C_\mu L(\sigma_N)^2}\Big]$ (and $N$ large enough if necessary), there exist  positive constants  $C_1$ and $C_2(p,T,R)$, not depending on $N$,  such that  
\begin{align} 
\label{s3-3-maxp_N}
E\Big( \sup_{(t,x)\in [0,T]\times B(0;R)} &|u_N(t,x)|^p\Big) \leq   \; 2^{p-1} C_1
+ C_2(p,T,R) \mathcal{K}(c(b_N), c(\sigma_N), L(b_N), L(\sigma_N)).
\end{align}
Here, $\mathcal{K}(c(b_N), c(\sigma_N), L(b_N), L(\sigma_N))$ is given by \eqref{KinKolmogorov}, with $c(b)$, $c(\sigma)$, $L(b)$, $L(\sigma)$ replaced by $c(b_N)$, $c(\sigma_N)$, $L(b_N)$, $L(\sigma_N)$. Recall that 
\beqn
c(b_N)=\theta_1, \quad c(\sigma_N)=\sigma_1, \quad L(b_N)=\theta_2 (\ln (2N))^\delta, \quad L(\sigma_N)=\sigma_2 (\ln (2N))^a, 
\eeqn
(see \eqref{s2-3.1}).
Because of {\bf(Cd)}, and since $\max(a,\delta) = \delta <\frac{1}{2}$, we have
\beq
\label{order}
 \mathcal{K}(c(b_N), c(\sigma_N), L(b_N), L(\sigma_N)) = o(N^p).
 \eeq

Consider the sequence of increasing stopping times defined in \eqref{tauN}. Using  \eqref{order}, we see that $\sup_N \tau_N = T$, a.s.
By the standard localization argument (see the details of the proof of Theorem \ref{wp-d=1}), we finish the proof.
\end{proof}

\section{Examples of covariance densities}
\label{s4}

In this section, we give three examples of covariances which satisfy the assumptions of section {\bf(h0)}-{\bf(h4)} of section \ref{s3} and for each of them, we identify the values of the parameter $\nu_1$, $\nu_2$ in \eqref{numbers}. For $d=3$, the same examples are studied in \cite{HHN}.
 
\subsection{Riesz kernels.}  For  $\beta\in(0,d)$, let $f_{\beta}:\RR^d \to [0,+\infty]$ be defined by 
$f_{\beta}(x) = |x|^{-\beta}$  for $x\in\rde\setminus\{0\}$, and  $f_\beta(0)=+\infty$. The inverse Fourier transform, 
is 
\beq
\label{inv-fourier}
\left(\mathcal{F}^{-1}f\right)(\zeta) = c_{d,\beta} f_{d-\beta}(\zeta)\ d\zeta, \ c_{d,\beta}= 2^{-\beta+d/2}\frac{\Gamma\left(\frac{d-\beta}{2}\right)}
{\Gamma\left(\frac{\beta}{2}\right)},
\eeq
where $\Gamma$ denotes the Euler Gamma function (see \cite[Chapter V]{Stein}). 

Let $\Lambda$ be the non-negative definite tempered distribution given by $\Lambda(dx) = f_{\beta}(x)\ dx$. According to \eqref{inv-fourier}, its spectral measure is  $\mu_\beta(d\zeta)= c_{d,\beta} f_{d-\beta}(\zeta)\ d\zeta$.
Observe that the integral $\int_{\RR^d}\frac{\mu_\beta(d\zeta)}{1+|\zeta|^2}$ converges if and only if  $\beta\in(0, 2\wedge d)$. 
 \smallskip
 
In the remaining of this section, we consider the dimensions $d=2,3$, and assume that $\beta\in(0, 2)$. From the previous discussion, we obtain that $\mu_{\beta}$ satisfies condition {\bf(h0)}.
 Since $f_\beta$ is a lower semicontinuous function, from  Remark \ref{rk4.1} we see that it satisfies {\bf(h1)}.
 
 Let $\gamma\in(0,1)$. Using polar coordinates if $d=2$ and spherical coordinates if $d=3$, we have
\beqn
\int_{\RR^d} \frac{\mu_\beta(d\zeta)}{1+ |\zeta|^{2-2\gamma}} = C_{\beta,d} \int_0^\infty \frac{\rho^{\beta-1}}{1+\rho^{2-2\gamma}} d\rho.
\eeqn
The integral on the right-hand side is finite if and only if  $\gamma < (2-\beta)/2$. Since
$|\zeta|^{2\gamma} \mu_\beta(d\zeta) = c_{d,\beta} |\zeta|^{-(d-\beta-2\gamma)} d\zeta$,
and the Fourier transform of this measure is
$g_\gamma(x)=\tilde c(\beta,d) |x|^{-(\beta +2\gamma)}$ (for some positive constant $\tilde c(\beta,d)$), if $\beta +2\gamma <d$, the function $|\zeta|^{2\gamma} \mu_\beta(d\zeta)$ is locally integrable. 
Therefore, $\mu_\beta$ satisfies the condition {\bf(h2)} for any $\gamma \in(0, (2-\beta \big)/2)$. 
\medskip

Apply the change of variable $\eta=t\zeta$ to deduce 
\begin{align*}
\int_{\rde} \left\vert\cF G(t)(\zeta)\right\vert^2\ \mu_\beta(d\zeta) = c_{d,\beta}\int_{\rde}\frac{\sin^2(t|\zeta|)}{|\zeta|^2} |\zeta|^{-d+\beta} d\zeta
= c_{d,\beta}\ t^{2-\beta} \int_{\rde} \frac{\sin^2(|\eta|)}{|\eta|^{2+d-\beta}} d\eta.
\end{align*}
Since the integral $I_{d,\beta}:= \int_{\rde} \frac{\sin^2(|\eta|)}{|\eta|^{2+d-\beta}} d\eta$ is finite, $\mu_\beta$ satisfies the condition {\bf(h3)} with $\nu=2-\beta$ and $C:=c_{d,\beta}\ I_{d,\beta}$.
\medskip


The function $f_{\beta}$ satisfies the condition  {\bf(h4)}(1)
for any $b\in (0, \min(2 -\beta,1))$. Indeed, the proof relies on \cite[Lemma 2.6, p. 10]{DSS-Memoirs} (which holds in any dimension $d\ge 1$) as follows.  Choose $b>0$ satisfying $0<\beta+b<d$. Letting $a:=d-(\beta+b)$, we have $a+b\in(0,d)$. Then, by applying Lemma 2.6 (a) in \cite{DSS-Memoirs}, we have
\begin{align*}
f_\beta(s(y+z)+h(y+z)) - f_\beta(s(y+z)+hz) &= |h|^b \int_{\rde} dw |s(y+z)+hz-hw|^{-(\beta+b)}\\
&\qquad\times\left[|w+y|^{-(d-b)}-|w|^{-(d-b)}\right].
\end{align*}
Consequenly, {\bf(h4)}(1) will be established if we prove
\begin{align}
\label{h41-r}
&\int_0^T ds \ s \int_{\rde} \int_{\rde} G(1,dy) G(1,dz)\notag\\
& \qquad\qquad \times \int_{\rde} dw \big|s(y+z)+hz-hw\big|^{-(\beta+b)}
\left\vert|w+y|^{-(d-b)}-|w|^{-(d-b)}\right\vert <\infty.
\end{align}
A small modification of the proof of Lemma 6.4 in \cite{DSS-Memoirs} shows that \eqref{h41-r} holds for  $d=2,3$, and 
for any $b$ such that $b\in (0, \min(2 -\beta,1))$.
 We refer also to \cite[Proposition 5.3, p. 383-385]{HHN} for a proof of \eqref{h41-r} 
in dimension $d=3$.
Going through the details of the proof of this proposition, 
we see that it can be extended to $d=2$, thanks to Lemma \ref{l-a-1}.
This completes the proof of the validity of {\bf(h4)}(1) for $f_\beta$, with  $b\in (0, \min(2 -\beta,1))$.
\medskip


Finally, we prove that $f_{\beta}$, satisfies the condition {\bf(h4)}(2)
 for any $\bar b\in (0,2 -\beta)$.  The proof relies on \cite[Lemma 2.6, p. 10]{DSS-Memoirs}. 
 As in the proof of {\bf(h4)}(1), we choose $\bar b>0$ satisfying $0<\beta+\bar b<d$. Letting $a:=d-(\beta+\bar b)$, we have $a+\bar b\in(0,d)$.
  Then,   Lemma 2.6 (e) in \cite{DSS-Memoirs} implies 
\begin{align*}
&f_{\beta}\big(s(y+z)+h(y+z)\big) - f_{\beta}\big(s(y+z)+h y\big) - f_{\beta}\big(s(y+z)+hz \big)
+ f_\beta\big(s(y+z)\big)\\
&\quad \le |h|^{\bar b} \int_{\rde} dw\ |y-hz|^{-(\beta+\bar b)}\\
&\qquad \quad \times\left[\vert w+hy+hz \vert^{-(d-\bar b)} - \vert w+hy\vert^{-(d-\bar b)} - \vert w+hz\vert^{-(d-\bar b)} 
+|w|^{-(d-\bar b)}\right].
\end{align*}
Consequently, {\bf(h4)}(2) will follow from
\begin{align}
\label{h42-r}
&\int_0^T ds \ s^2\int_{\rde} \int_{\rde} G(1,dy) G(1,dz)
\int_{\rde} dw  |y-hz|^{-(\beta+\bar b)} \nonumber \\
&\qquad  \times\left\vert\vert w+hy+hz \vert^{-(d-\bar b)} - \vert w+hy\vert^{-(d-\bar b)} - \vert w+hz\vert^{-(d-\bar b)} 
+|w|^{-(d-\bar b)}\right\vert<\infty.
\end{align}

With a slight modification (and simplification) of the proof of Lemma 6.5 in \cite{DSS-Memoirs}, 
we can check that \eqref{h42-r} holds for $d=2,3$ and for any $\bar b\in(0,2-\beta)$. 
Using ideas introduced in the proof of this lemma, \cite[Proposition 5.3, p. 385-386]{HHN} provides also
a proof of \eqref{h42-r} in dimension $d=3$ with $\bar b\in(0,2-\beta)$. 
Going through the details of the proof of this proposition, we see that it can be extended to $d=2$, thanks once more to Lemma \ref{l-a-1}.
Therefore, $f_\beta$ satisfies {\bf(h4)}(2)  with 
$\bar b\in (0,2 -\beta)$. 
\medskip

\noindent{\em Conclusion.}  Let $d=2,3$ and $\beta\in(0,2)$. For spatially homogeneous Gaussian noises with covariance function given by \eqref{s3.1} with $\Lambda (dx) = f_\beta(x)\ dx$, the parameters in \eqref{numbers} are $\nu_1=\nu_2=\min(\gamma, \gamma_1,\gamma_2)$, with $\gamma < \frac{2-\beta}{2}$.
Hence, from \eqref{s3-3.72bis} we deduce that, almost all sample paths of the solution to \eqref{n3} are  
locally H\"older continuous, jointly in $(t,x)$, with exponent $\theta\in]0, \min((2-\beta)/2, \gamma_1,\gamma_2[$. 
For $d=3$, this is \cite[Theorem 4.11, p. 48]{DSS-Memoirs}. 
Moreover, the critical exponent $\min((2-\beta)/2, \gamma_1,\gamma_2)$ is sharp in both dimensions, $d=2,3$ (see \cite{DSS-Memoirs}, 
\cite{DaSS-Hitting}).

\bigskip

\subsection{Bessel kernels}
\label{bessel}
For any $\kappa>0$, the Bessel kernel is the function defined by 
$\tilde{f}_\kappa(x)=\int_0^\infty w^{\frac{\kappa -d-2}{2}} e^{-w} e^{-\frac{|x|^2}{4w}} dw$ for $x\in \RR^d\setminus \{0\}$,  
$\tilde{f}_\kappa(0)= \infty$  if $0<\kappa\le d$, and $\tilde{f}_\kappa(0)= c(d,\kappa)$  if $\kappa>d$, where $0<c(d,\kappa)<\infty$ (see 
\cite[Chapter V]{Stein}) and  \cite{A}). The inverse Fourier transform is
\beq
\label{Fourier-B}
\left(\mathcal{F}^{-1}\tilde{f}_\kappa\right)(\zeta) = C_{d,\kappa} \big(1+|\zeta|^2\big)^{-\frac{\kappa}{2}}.
\eeq
Let $\Lambda$ be the measure defined by $\Lambda(dx) = \tilde{f}_\kappa(x)\ dx$. 
From \eqref{Fourier-B}, we have that the corresponding spectral measure is 
 $\tilde{\mu}_\kappa (d\zeta)=C_{d,\kappa} \big(1+|\zeta|^2\big)^{-\frac{\kappa}{2}}\ d\zeta$.
 
Throughout this section, we consider the case $d=2, 3$, and we assume $\kappa > d-2$. Our aim is to prove that 
hypotheses  {\bf (h0)}--{\bf (h4)} are satisfied.

Since $\tilde{f}_\kappa$ is lower semicontinuous, the condition {\bf (h1)} holds (see Remark \ref{rk4.1}).
 
 
Fix $\gamma\ge 0$. Using polar coordinates for $d=2$ and spherical coordinates for $d=3$, we obtain
  \beqn
 \int_{\RR^d} \frac{\tilde{\mu}_\kappa(d\zeta)}{1+|\zeta|^{2-2\gamma}} = 
 C(d,\kappa) \int_0^\infty r^{d-1} \Big( \frac{1}{1+r^2}\Big)^{\frac{\kappa}{2}} \frac{1}{1+r^{2-2\gamma}}\ dr.
 \eeqn
  The integral on the right-hand side is finite if and only if $2\gamma < \kappa -d + 2$. Take $\gamma=0$ to deduce that {\bf (h0)} holds.
Furthermore, for $\gamma \in \big(0, \min\big( \frac{\kappa -d+2}{2},  1\big)\big)$, the constant $C^{(\gamma)}_{\tilde{\mu}_\kappa}$ 
defined in \eqref{cmugamma} is finite and therefore, {\bf (h2)} holds. 

\medskip

We next check {\bf (h3)}. For  any $t>0$ we have 
\beqn
\int_{\RR^d} |{\mathcal F} G(t) (\zeta)|^2 \tilde{\mu}_\kappa(d\zeta)=
C_{d,\kappa} \int_{\RR^d} \frac{\sin^2(t|\zeta|)}{|\zeta|^2\left(1+|\zeta|^2\right)^{\frac{\kappa}{2}}} \ d\zeta  = \tilde C_{d,\kappa} \int_0^\infty r^{d-1} \frac{\sin^2(t r)}{r^2}  \big(1+r^2 \big)^{-\frac{\kappa}{2}}  dr.
\eeqn
 Splitting the last integral, for any $t\in (0,T]$ we obtain 
\begin{align*} \int_0^\infty r^{d-1} \frac{\sin^2(t r)}{r^2} \big(1+r^2 \big)^{-\frac{\kappa}{2}}  dr & \leq
\int_0^1 t^2 r^{d-1} dr + \int_1^{\frac{T}{t}} t^2 r^{d-1-\kappa} dr + \int_{\frac{T}{t}}^\infty r^{d-1} r^{-2} r^{-\kappa} dr \\
&\leq \frac{t^2}{d} + \int_1^{\frac{T}{t}} t^2 r^{d-1-\kappa} dr  + \frac{1}{\kappa -d+2} \Big(\frac{t}{T}\Big)^{\kappa -d+2}.
\end{align*} 
Furthermore, by setting $I(t): =\int_1^{\frac{T}{t}} t^2 r^{d-1-\kappa} dr$, we have
 \beqn
 I(t)\le \begin{cases}
 \frac{t^2}{d-\kappa} \Big( \frac{T}{t}\Big)^{d-\kappa},& \qquad \mbox{\rm if}\quad  d-2< \kappa <d, \\
 t^2 \ln \Big( \frac{T}{t}\Big),&  \qquad \mbox{\rm if} \quad \kappa =d, \\
  \frac{t^2}{\kappa-d} ,& \qquad \mbox{\rm if} \quad d<\kappa.
  \end{cases}
  \eeqn 
   This implies 
 \[ \int_{\RR^d} |{\mathcal F} G(t) (\zeta)|^2 \tilde{\mu}_\kappa(d\zeta) \leq C(d,\kappa,T)\;  t^\nu,\ t\in [0,T],\]
 where $C(d,\kappa,T)$ is some positive constant and $\nu < \min(2, \kappa-d+2)$.
  Therefore,  {\bf (h3)} holds with  $\nu < \min(2, \kappa-d+2)$.
 \medskip
 
 
For $d=3$, the validity of {\bf (h4)} is proved in \cite[Section 5.3]{HHN}. Going through the arguments of this reference, we see that they also hold for $d=2$. For the sake of completeness we give some details.
 
Let us first focus on {\bf (h4)}(1). 
Set
\[ I_1 = \int_0^T ds \;  s \int_{\RR^d}\int_{\RR^d} G(1,dy) G(1, dz) \big| \tilde{f}_\kappa \big( (s+h)(y+z)\big) - \tilde{f}_\kappa\big( (s+h)z+sy\big) \big|.\]
Fix $y,z\in\RR^d$, $s, h\in [0,T]$ and $b\in (0,1)$. By writing the inequality in \cite[p. 390, line 3]{HHN} with $x:= (s+h)z+sy$ and $\xi:=y$, we see that 
\begin{align*}
&\left| e^{-\frac{|(s+h) (y+z)|^2}{4w}} - e^{-\frac{|(s+h)z+sy|^2}{4w}} \right| \\
&\qquad \leq   C \, \left(\frac{h}{\sqrt{w}} \right)^b
 \int_0^1 d\lambda\ \left(e^{-\frac{|(s+h)z-(s+\lambda h)y|^2}{8w}}
+ e^{-\frac{|(s+h)(y+z)|^2}{4w}} 
+ e^{-\frac{|(s+h)z +sy|^2}{4w}}\right) .
\end{align*}
Hence, by the definition of $\tilde{f}_\kappa$,
\begin{align*}
\big| \tilde{f}_\kappa \big( (s+h)(y+z)\big) - &\tilde{f}_\kappa \big((s+h)z+sy\big)\big| \leq C h^b \int_0^1 d\lambda \int_0^\infty dw 
\; w^{\frac{\kappa -b -d-2}{2}} \, e^{-w} \\ 
& \times \left(e^{-\frac{|(s+h)z-(s+\lambda h)y|^2}{8w}}
+ e^{-\frac{|(s+h)(y+z)|^2}{4w}} + e^{-\frac{|(s+h)z +sy|^2}{4w}}\right), 
\end{align*}
and therefore,
\[ I_1 \leq C\, h^b\, \int_0^\infty dw \; w^{\frac{\kappa -b-d-2}{2}} \; e^{-w} \sum_{i=1}^3 T_i(w),\]
where
\begin{align*}
T_1(w)=&\int_0^1  d\lambda \int_0^T s \, ds \int_{\RR^d} \int_{\RR^d} G(1,dy) G(1, dz)  e^{-\frac{|(s+h)z-(s+\lambda h)y|^2}{8w}}, \\
T_2(w)=& \int_0^T s \, ds \int_{\RR^d} \int_{\RR^d} G(1,dy) G(1, dz) e^{-\frac{|(s+h)(y+z)|^2}{4w}}, \\
T_3(w)=& \int_0^T s \, ds \int_{\RR^d} \int_{\RR^d} G(1,dy) G(1, dz) e^{-\frac{|(s+h)z +sy|^2}{4w}}.
\end{align*}
Apply the change of variables $x=(s+h)z$, $\bar x = (s+\lambda h)y$, and then Parseval's identity, to deduce
\begin{align}
\label{P-arg}
&\int_{\RR^d} \int_{\RR^d} G(1,dy) G(1, dz)  e^{-\frac{|(s+h)z-(s+\lambda h)y|^2}{8w}}\notag\\
&\qquad\quad= \frac{1}{(s+h)(s+\lambda h)} \int_{\RR^d} \int_{\RR^d} G(s+h,dx) G(s+\lambda h,d\bar x) e^{-\frac{|x-\bar x|^2}{8w}}\notag\\
&\qquad\quad= \left(\frac{2}{\pi}\right)^{\frac{d}{2}}\frac{1}{(s+h)(s+\lambda h)}\int_{\RR^d} w^{\frac{d}{2}} e^{-2w |\zeta|^2} \left(\frac{\sin\big( (s+h)|\zeta|\big)}{|\zeta|}\right)
\left(\frac{\sin\big( (s+\lambda h)|\zeta|\big)}{|\zeta|}\right) d\zeta.
\end{align}
This yields
\begin{align}
 \label{T_1-bessel}
T_1(w)
&\;\le \left(\frac{2}{\pi}\right)^{\frac{d}{2}}  w^{\frac{d}{2}} \, \int_0^1 d\lambda \int_0^T ds\ \frac{s}{(s+h)^{1-\epsilon}(s+\lambda h)^{1-\epsilon}}  
\int_{\RR^d}  e^{-2w |\zeta|^2} \, |\zeta |^{2\epsilon -2} \, d\zeta , 
\end{align}
where in the last inequality, $\epsilon \in (0,1)$, and we have used the estimate $\frac{|\sin(a|x|)|}{a|x|} \leq 1$, $a>0$, which implies 
  $\frac{|\sin(a|x|)|}{a|x|}\leq\big( \frac{\sin(a|x|)}{a|x|}\big)^{1-\epsilon} 
 \leq a^{\epsilon-1} |x|^{\epsilon -1}$.

Since $\epsilon \in (0,1)$, for every $\lambda \in (0,1)$ we have $\int_0^T \frac{s}{(s+h)^{1-\epsilon}(s+\lambda h)^{1-\epsilon}} ds \leq
\int_0^T s^{-1+2\epsilon} ds <\infty$.

Applying the change of variables $\eta = \sqrt{w}\zeta$, we obtain 
\begin{align}
\label{upper-T1-bessel}
T_1(w)\leq &\;  C\; w^{\frac{d}{2}} \int_{\RR^d} e^{-2|\eta|^2} \; |\eta|^{2\epsilon -2}\; w^{1-\epsilon} \; w^{-\frac{d}{2}} d\eta \nonumber  \\
\leq & \; C w^{1-\epsilon} \int_0^\infty e^{-2 \rho^2} \; \rho^{2(\epsilon -1)}\; \rho^{d-1}\; d\rho \leq C w^{1-\epsilon}.
\end{align}
Indeed, the last integral  is finite if and only if $2\epsilon >2-d$; since $d=2,3$ and $\epsilon >0$ this constraint is
satisfied. 
Similarly, for any $\epsilon \in (0,1)$,
\[ T_2(w)+T_3(w) \leq C w^{1-\epsilon}.\]

Thus, for any $h\in [0,T]$ and $b\in(0,1)$, we have proved 
\beq
\label{I1}
I_1 \leq C h^b \int_0^\infty w^{\frac{\kappa -b-d-2}{2}}  e^{-w} w^{1-\epsilon}\ dw.
\eeq
By taking $\varepsilon$ arbitrarily close to zero, we see that this integral is finite if $b<\kappa-d+2$. Consequently, 
 {\bf (h4)}(1) is satisfied for $b<\min(\kappa -d+2, 1)$. 
\medskip


Finally, we address the validity of  {\bf (h4)}(2), by using a similar approach as for {\bf (h4)}(1). Set
\begin{align*}
I_2 = & \int_0^T ds\ s^2 \int_{\RR^d} \int_{\RR^d} G(1, dy) G(1,dz)\\
&\quad \times 
\big|  \tilde{f}_\kappa \big( (s+h)(y+z)\big) - \tilde{f}_\kappa\big( (s+h) y +sz\big)
- \tilde{f}_\kappa\big( sy+(s+h)z\big) + \tilde{f}_\kappa\big( s(y+z)\big)|.
\end{align*}
Apply the first inequality in \cite[p. 392]{HHN} with $x:=s(y+z)$, $\xi:=y$ and $\eta:=z$, to see that, for any $y,z\in\RR^d$, $s,h\in[0,T]$ and $\bar{b}\in (0,2)$,
\beqn
\left\vert e^{-\frac{|(s+h)(y+z)|^2}{4w}} - e^{-\frac{|s(y+z)+hy|^2}{4w}} - e^{-\frac{|sy+(s+h)z|^2}{4w}} + e^{-\frac{|s(y+z)|^2}{4w}}\right\vert
\le Ch^{\bar{b}} \, w^{-\frac{\bar{b}}{2}} q_{h,y,z}(s,w),
\eeqn
where 
\begin{align}
\label{def-q}
q_{h,y,z}(s,w) = \int_0^1 d\lambda \int_0^1 d\mu &\left(e^{-\frac{|(s+\lambda h) y+ (s+\mu h) z|^2}{8w}} + 
 e^{-\frac{|(s+h)(y+z) |^2}{8w}} +  e^{-\frac{|s(y+z) +  h y |^2}{4w}}\right. \notag\\
 &\left.\quad  +  e^{-\frac{|s(y+z) +  hz|^2}{4w}} +  e^{-\frac{|s(y+z)|^2}{4w}} \right) .
\end{align}
Therefore, 
\begin{align}
\label{I2}
I_2 
&\leq  C\, h^{\bar{b}} 
 \int_0^\infty dw\ w^{\frac{\kappa -\bar{b} -d-2}{2}} \, e^{-w} \int_0^T ds\  s^2
\int_{\RR^d} \int_{\RR^d} G(1, dy) G(1,dz) q_{h,y,z}(s,w)\notag\\
& = C\, h^{\bar{b}}  \int_0^\infty dw\ w^{\frac{\kappa -\bar{b} -d-2}{2}} \, e^{-w} \sum_{i=1}^5 S_i.
\end{align}
In the last expression,
\beqn
S_1= \int_0^1 d\lambda \int_0^1 d\mu \int_0^T ds\  s^2 \int_{\RR^d} \int_{\RR^d} G(1, dy) G(1,dz)e^{-\frac{|(s+\lambda h) y+ (s+\mu h) z|^2}{8w}},
\eeqn
and $S_i$, $i=2,\ldots,5$, are defined in a similar way, by taking each of the remaining exponential terms in \eqref{def-q}.

As in \eqref{P-arg}, 
\begin{align*}
& \int_{\RR^d} \int_{\RR^d} G(1, dy) G(1,dz) e^{-\frac{|(s+\lambda h) y+ (s+\mu h) z|^2}{8w}}\\
&\qquad \qquad \quad \le \left(\frac{2}{\pi}\right)^{\frac{d}{2}} w^{\frac{d}{2}} \frac{1}{(s+\lambda h)^{1-\epsilon} (s+\mu h)^{1-\epsilon}} 
\int_{\RR^d}  e^{-2w|\zeta|^2} |\zeta|^{2\epsilon -2}\ d\zeta,
\end{align*}
implying
\beqn
S_1 \le \left(\frac{2}{\pi}\right)^{\frac{d}{2}} w^{\frac{d}{2}}  \int_0^1 d\lambda \int_0^1 d\mu \int_0^T ds\
\frac{s^2}{(s+\lambda h)^{1-\epsilon} (s+\mu h)^{1-\epsilon}}\int_{\RR^d}  e^{-2w|\zeta|^2} |\zeta|^{2\epsilon -2}\ d\zeta,
\eeqn
for any $\varepsilon\in(0,1)$. 

Observe the analogy between this inequality and \eqref{T_1-bessel}. Arguing as in the analysis of the right-hand side of \eqref{T_1-bessel}, we obtain $S_1\le C(T) w^{1-\epsilon}$, and, with similar arguments, also $S_i\le C(T) w^{1-\epsilon}$ for any  $i=2,\ldots, 5$. These estimates along with \eqref{I2} yield
\beqn
 I_2 \leq  C  h^{\bar{b}} \int_0^\infty w^{\frac{\kappa -\bar{b}-d-2}{2}} e^{-w} w^{1-\epsilon}\  dw.
 \eeqn
As in \eqref{I1}, by taking $\varepsilon$ arbitrarily close to zero, we see that the last integral converges if 
$\bar b< \kappa-d+2$, thereby proving that {\bf (h4)}(2) holds with $\bar{b} < \min (\kappa -d+2 , 2)$.
\bigskip


\noindent{\em Conclusion.}  Let $d=2,3$, $\kappa>d-2$. For spatially homogeneous Gaussian noises with covariance function given by \eqref{s3.1} with $\Lambda (dx) = \tilde f_\kappa(x)\ dx$, the parameters $\nu_1$,
$\nu_2$ defined in \eqref{numbers} are
 $\nu_1 = \nu_2=\min(\gamma, \gamma_1,\gamma_2)$, with $\gamma< \min\big( \frac{\kappa -d+2}{2}, 1\big)$.
  
Thus, we deduce that almost all sample paths of the solution to \eqref{n3} are  locally H\"older continuous, jointly in $(t,x)$, 
with exponent $\theta\in\left]0, \min(\frac{\kappa-d+2}{2}, 1, \gamma_1,\gamma_2)\right[$. When $d=3$, we recover the results in \cite[p. 393]{HHN}. Whether this H\"older exponent is sharp seems to be an open question.

\bigskip 


\subsection{Fractional kernels} Let $d=2,3$ and $H=(H_i)_{1\le i\le d}$, with 
$\frac{1}{2} < H_i <1$.
Let $\bar{f}_H(x)=C_H \prod_{i=1}^d |x_i|^{2H_i-2}$, when $\prod_{i=1}^d x_i \neq 0$, where $C(H)=\prod_{i=1}^d H_i (2H_i-1)$, 
and $\bar{f}_H(x)=+\infty$ otherwise. The inverse Fourier transform of $\bar{f}_H$ is 
$\left(\mathcal{F}^{-1}\bar{f}_H\right)(\zeta) = C_H \prod_{i=1}^d |\zeta_i|^{1-2H_i}$, 
where $C_H$ is some positive constant depending only on $H$.

Consider the non-negative definite tempered distribution $\Lambda(dx) =\bar{f}_H(x) dx$, whose spectral measure is 
$\bar{\mu}_H(\zeta) = C_H \prod_{i=1}^d |\zeta_i|^{1-2H_i} \, d\zeta$. 
In this section, we prove that $\bar{\mu}_H$ satisfies the conditions {\bf (h0)}--{\bf (h4)}.

Since the function $\bar{f}_H$ is lower semi-continuous, condition {\bf(h1)} holds, by 
Remark \ref{rk4.1}. 

We next check Condition {\bf (h0)}. For any $H$ as above, we have 
\[ I=\int_{\RR^d} \frac{\bar{\mu}_H(d\zeta)}{1+|\zeta|^2} = C_H \int_{\RR^d} \frac{\prod_{i=1}^d |\zeta_i|^{1-2H_i}}{1+|\zeta|^2} d\zeta = I(d)\tilde{I}(d),\]
where using polar (resp. spherical) coordinates when $d=2$ (resp. $d=3$),  we have
\[ I(d)=\int_0^\infty r^{d-1} r^{d-2\sum_{i=1}^d H_i} (1+r^2)^{-1} dr, \]
and 
\begin{align} 			\label{J(2)}
\tilde{I}(2)= & 4 \int_0^{\frac{\pi}{2}} |\cos(\theta)|^{1-2H_1}  |\sin(\theta)|^{1-2H_2} d\theta, \\
\tilde{I}(3)=& \; 2\, \tilde{I}(2)\; \int_{0}^{\frac{\pi}{2}} |\sin (\phi)|
|\sin(\phi)|^{2-2(H_1+H_2)} |\cos(\phi)|^{1-2H_3} d\phi.		\label{J(3)}
\end{align}
The change of variable $u=\sin(\theta)$ on the interval $(0,\frac{\pi}{2})$ and the constraints on  $H_i$  imply 
\[ \tilde{I}(2)=C \int_0^1 u^{1-2H_2} (1-u^2)^{-H_1} du \leq  C \int_0^1 u^{1-2H_2} (1-u)^{-H_1} du<\infty.\]
Similarly, the change of variable $v=\sin(\phi)$ on the interval $(0,\frac{\pi}{2})$ yields
\[ \int_{0}^{\frac{\pi}{2}} 
|\sin(\phi)|^{3-2(H_1+H_2)} |\cos(\phi)|^{1-2H_3} d\phi = 2\int_0^1 v^{3-2(H_1+H_2)} (1-v^2)^{-H_3} dv <\infty;\]
hence $\tilde{I}(3)<\infty$. 

Finally, it  is easy to see that $I(d)<\infty$ if and only if $2d-1-2\sum_{i=1}^d H_i >-1$ and $2d-3-2\sum_{i=1}^d H_i <-1$. 
Since $H_i<1$, the first constraint is satisfied. The second one is equivalent to $\sum_{i=1}^d H_i > d-1$. 
Hence, if  $\sum_{i=1}^d H_i > d-1$, condition {\bf (h0)} is satisfied.

Let us now check that condition {\bf (h2)} is satisfied. Given $\gamma >0$, for 
$\tilde{I}(d)$ defined by \eqref{J(2)} and \eqref{J(3)} for $d=2, 3$, we have 
\[ \int_{\RR^d} \frac{\bar{\mu}(d\zeta)}{1+|\zeta|^{2-2\gamma}} = 
\tilde{I}(d) \int_0^\infty r^{d-1} r^{d-2\sum_{i=1}^d H_i} \big( 1+r^{2-2\gamma}\big)^{-1} dr.\]
Set
\begin{equation}			\label{kappa_bar}
\bar{\kappa}=\sum_{i=1}^d H_i -(d-1) > 0.
\end{equation}
A  computation similar to that used to check {\bf (h0)} shows that this last integral is finite if and only if 
$2d-3 -2\sum_{i=1}^d H_i+ 2\gamma<-1$,
that is $\gamma <\bar{\kappa}$, where $\bar{\kappa}$ is defined in \eqref{kappa_bar}.

To check condition {\bf (h3)}, we use similar arguments to deduce that for $t\in [0,T]$, $\tilde{I}(d)$ defined as above 
\begin{align*}
\int_{\RR^d} \big| {\mathcal F} G(t) (\zeta) \big|^2 \bar{\mu}_H(d\zeta) = & C_H \int_{\RR^d} \Big| \frac{\sin(t|\zeta|)}{|\zeta|}\Big|^2 
\prod_{i=1}^d |\zeta_i|^{1-2H_i} d\zeta \\
&= C_ H \tilde{I}(d) \int_0^\infty r^{d-1} \Big| \frac{\sin(tr)}{r}\Big|^2 r^{d-\sum_{i=1}^d H_i}  dr \\
& = C \Big[ \int_0^{t^{-1}} t^2 r^{2d-1-2\sum_{i=1}^d H_i} dr + \int_{t^{-1}}^\infty r^{d-1-2+d-2\sum_{i=1}^d H_i} dr \\
&= C\Big[ t^{2-2 (d-\sum_{i=1}^d H_i)} + t^{2\sum_{i=1}^d H_i -2d+2}\Big].
\end{align*} 
We deduce $\int_{\RR^d} \big| {\mathcal F} G(t) (\zeta) \big|^2 \bar{\mu}_H(d\zeta) \leq C t^{2\bar{\kappa}}$, that
is \eqref{h3} holds with $\nu = 2\bar{\kappa}$,  where $\bar{\kappa}$ is defined in \eqref{kappa_bar}.

The proof of {\bf (h4)} is similar to that of a similar condition in \cite{HHN}. We sketch it below for the sake of completeness.
We start with  {\bf (h4)}(1). Apply the inequality in \cite[p. 395, bottom]{HHN} to see that if $d=3$, 
\begin{align*}
|\bar{f}_H(x+hy)-\bar{f}_H(x)| \leq &  C_H \big| |x_1+hy_1|^{2H_1-2} - |x_1|^{2H_1-2}\big| |x_2+hy_2|^{2H_2-2} |x_3+h y_3|^{2H_3-2} \\
&+ |x_1|^{2H_1-2} \big| |x_2+hy_2|^{2H_2-2} - |x_1|^{2H_2-2}\big|  |x_3+h y_3|^{2H_3-2} \\
&+ |x_1|^{2H_1-2} |x_2|^{2H_2-2} \big| |x_3+hy_3|^{2H_3-2} - |x_1|^{2H_2-2}\big|,
\end{align*}
while for $d=2$
\begin{align*}
|\bar{f}_H(x+hy)-\bar{f}_H(x)| \leq & C_H \big| |x_1+hy_1|^{2H_1-2} - |x_1|^{2H_1-2}\big| |x_2+hy_2|^{2H_2-2} \\
&+  |x_1|^{2H_1-2} \big| |x_2+hy_2|^{2H_2-2} - |x_1|^{2H_2-2}\big|.
\end{align*} 

In both cases, we have $|\bar{f}_H(x+hy)-\bar{f}_H(x)|= \sum_{i=1}^d F_{d,i}(x,y,h)$. In the sequel, we only describe the computations in the
case $d=3$; the case $d=2$ is easier and dealt with in a similar way.

Using the identity (2.11) in \cite{DSS-Memoirs} for $d:=1$, $a=2H_1-1-\rho$, $b=\rho$ for some positive 
$\rho <  (2\bar{\kappa}) \wedge \min(2H_i-1;  i=1, \cdots, d)$, we deduce since $a+b=2H_1-1\in (0,1)$, 
\begin{align*}
\big| |x_1-hy_1|^{2H_1-2} - |x_1|^{2H_1-2}\big| =&\Big| \int_{\RR} dw |x_1-w|^{2H_1-2-\rho} |w+hy_1|^{\rho-1} \\
&\qquad\qquad 
- \int_{\RR} dw |x_1-w|^{2H_1-2-\rho} |w|^{\rho -1} \Big| \\
=&  \Big| \int_{\RR} h\ dw |x_1-hw|^{2H_1-2-\rho} h^{\rho-1} |w+y_1|^{\rho-1} \\
&\qquad\qquad 
- \int_{\RR} h\ dw |x_1-hw|^{2H_1-2-\rho} h^{\rho-1} |w|^{\rho -1} \Big| .
\end{align*}
Hence, changing $y_1$ into $-y_1$, we have 
$ \int_{\RR^d} \int_{\RR^d} G(1,dy) G(1,dz) F_{3,1}((s+h)z+sy,y,h) \leq h^\rho \sum_{j=1}^3 T_{1,j}(s,h)$, where
\begin{align*}
T_{1,1}(s,h)=&  \int_{\RR^d} \!\! \int_{\RR^d} G(1,dy) G(1,dz) \int_{\{|w|\leq 3\}} |(s+h)z_1+sy_1-hw|^{2H_1-2-\rho} |w-y_1|^{\rho -1}\\
&\qquad \times  |(s+h)(z_2+y_2)|^{2H_2-2} 
 |(s+h)(z_3+y_3)|^{2H_3-2} dw , \\
 T_{1,2}(s,h)=& \int_{\RR^d}\!\!  \int_{\RR^d} G(1,dy) G(1,dz) \int_{\{|w|\leq 3\}} |(s+h)z_1+sy_1-hw|^{2H_1-2-\rho} |w|^{\rho -1}\\
&\qquad \times  |(s+h)(z_2+y_2)|^{2H_2-2} 
 |(s+h)(z_3+y_3)|^{2H_3-2} dw , \\
 T_{1,3}(s,h)=&  \int_{\RR^d}\!\!  \int_{\RR^d} G(1,dy) G(1,dz) \int_{\{|w| > 3\}} |(s+h)z_1+sy_1-hw|^{2H_1-2-\rho} \\
&\qquad \times \big| |w-y_1|^{\rho -1} - |w|^{\rho -1}\big| \;  |(s+h)(z_2+y_2)|^{2H_2-2} 
 |(s+h)(z_3+y_3)|^{2H_3-2} dw.
\end{align*}
To prove an upper estimate of $T_{1,1}(s,h)$, let $\tilde{w}=w-y_1$; then using the scaling property \eqref{scaling} 
we obtain
\begin{align*}
T_{1,1} &(s,h)=   \int_{\RR^d} \!\! \int_{\RR^d} G(1,dy) G(1,dz) \int_{\{|\tilde{w}|\leq 4\}} |(s+h)(z_1+y_1)-h\tilde{w}|^{2H_1-2-\rho} |\tilde{w}|^{\rho -1}\\
&\qquad\qquad \qquad  \times  |(s+h)(z_2+y_2)|^{2H_2-2} 
 |(s+h)(z_3+y_3)|^{2H_3-2} dw , \\
 =& \frac{1}{(s+h)^2} \int_{\RR^d} \!\! \int_{\RR^d}\!\! G(s+h,dy) G(s+h,dz) \!\! \int_{\{|\tilde{w}|\leq 4\}}\!\!
  \bar{f}_{H_1-\frac{\rho}{2}, H_2,H_3}\big(y+z-(h\tilde{w},0,0)\big)  
 |\tilde{w}|^{\rho-1} d\tilde{w}\\
 =&   \frac{C}{(s+h)^2} \int_{\{|w|\leq 4\}} dw\, \int_{\RR^d}  \Big| \frac{\sin\big((s+h)|\zeta|\big)}{|\zeta|}\Big|^2 \big| e^{ihw\zeta_1}\big| |\zeta_1|^{1-2H_1+\rho}
 |\zeta_2|^{1-2H_2}  |\zeta_3|^{1-2H_3} d\zeta.
\end{align*}
Since $\frac{1}{2}<H_1-\frac{\rho}{2}<1$, the computations made to check {\bf (h3)} imply 
 \[ \int_0^T s \, T_{1,1}(s,h) ds \leq C \int_0^T (s+h)^{2\bar{\kappa}-1} ds \leq C(T) <\infty\] 
  for all $h\in [0,T]$. 

We next upper estimate $T_{1,2}(s,h)$ using once more the scaling property; this yields 
\begin{align*}
T_{1,2}&(s,h) =   \int_{\RR^d} \!\! \int_{\RR^d} G(1,dy) G(1,dz) \int_{\{|w|\leq 3\}} \left| (s+h)z_1+ sy_1-h{w}\right|^{2H_1-2-\rho} |{w}|^{\rho -1}\\
&\qquad \qquad \qquad \times  |(s+h)(z_2+y_2)|^{2H_2-2} 
 |(s+h)(z_3+y_3)|^{2H_3-2} dw , \\
 =&  \frac{1}{(s+h)^2} \int_{\RR^d} \!\! \int_{\RR^d} G(s+h,dy) G(s+h,dz) \int_{\{|w|\leq 3\}}
  \left|z_1+ \frac{s}{s+h} y_1-h{w}\right|^{2H_1-2-\rho} |{w}|^{\rho -1} \\
 & \qquad \times  |(s+h)(z_2+y_2)|^{2H_2-2} 
 |(s+h)(z_3+y_3)|^{2H_3-2} dw . 
\end{align*}
Let $\psi(y)=\big( \frac{s}{s+h} y_1,y_2,y_3\big)$ and $G^\psi(s+h)$ denote the image of the measure $G(s+h, dy)$ by $\psi$. Then
Fubini's theorem implies for $s>0$ 
\begin{align*} 
T_{1,2}&(s,h) = \frac{1}{(s+h)^2} \int_{\{|w|\leq 3\}} \!\! |w|^{\rho-1} dw \!\! \int_{\RR^d}\!\! 
\big[ G^\psi(s+h) * G(s+h)\big](dx) \\
&\qquad \qquad \qquad \times |x_1|^{2H_1-\rho-2} |x_2|^{2H_2-2}
|x_3|^{2H_3-2} \\
\leq & \; C   \frac{1}{(s+h)^2} \int_{\{|w|\leq 3\}} |w|^{\rho-1} dw \int_{\RR^d} \big| {\mathcal F}\big[  G^\psi(s+h) * G(s+h)\big](\zeta) |e^{i \zeta_1 w}|  \\
&\qquad\qquad \times| \zeta_1|^{1+\rho-2H_1}
|\zeta_2|^{1-2H_2} |\zeta_3|^{1-2H_3} d\zeta \\
\leq &\;  C  \frac{1}{(s+h)^2}\int_{\{|w|\leq 3\}} |w|^{\rho-1} dw \int_{\RR^d}  \Big| \frac{\sin\big( (s+h) |\zeta|\big)}{|\zeta|}\Big| 
\Big| \frac{\sin\big( (s+h)\big| \big( \frac{s}{s+h} \zeta_1, \zeta_2,\zeta_3\big) \big| \big)}{\big| \big(\frac{s}{s+h} \zeta_1, \zeta_2,\zeta_3\big)\big|}\Big| \\
&\qquad\qquad \times 
|\zeta_1|^{1+\rho-2H_1}
|\zeta_2|^{1-2H_2} |\zeta_3|^{1-2H_3} d\zeta \\
\leq & \; C \int_{\{|w|\leq 3\}} |w|^{\rho-1} dw \Big[ \int_0^1  r^{d-1} r^{d+\rho-2\sum_{i=1}^d H_i} dr  \\
&\qquad \qquad \qquad +
 \frac{1}{(s+h)^2}\, \frac{s+h}{s} \int_1^\infty r^{d-1}  r^{d+\rho-2-2\sum_{i=1}^d H_i} dr \Big].
\end{align*}
These integrals are convergent since $\sum_{i=1}^d H_i<d$ implies $2d+\rho-1-2\sum_{i=1}^d H_i>-1$, and since  $\rho < 2\bar{\kappa}$ implies
$2d-3+\rho-2\sum_{i=1}^d H_i<-1$. Therefore, $ \int_0^T s T_{1,2}(s,h)  ds \leq C(T)$ for all $h\in [0,T]$. 

Finally, 
\[\big| |w+y_1|^{\rho-1} - |w|^{\rho-1}\big| \leq \int_0^1 \frac{\partial}{\partial \lambda}\big( |w+\lambda y_1|^{\rho-1}\big) 
d\lambda \leq C |w|^{\rho-2}.\]
Notice that since $\rho<2H_1-1<1$, we have $\rho-2<-1$. Hence   arguments similar to that used to upper estimate $T_{1,2}(s,h)$ imply 
$\int_0^T s \, T_{1,3}(s,h) ds \leq C(T) <\infty$ for all $h\in [0,T]$. 

A similar computation implies 
\[ \int_0^T ds\, s  \int_{\RR^d} \int_{\RR^d} G(1,dy) G(1, dz) \sum_{i=1}^d F_{3,i}((s+h)z+sy,y,h) 
\leq C h^\rho.\] 
This shows that \eqref{h41} holds with $b < (2\bar{\kappa}) \wedge \min( 2 H_i-1 ; i=1, \cdots, d)$.

To prove {\bf (h4)}(2), we use the inequality
\begin{align*}
 &\big|\bar{f}_H\big( (s+h)(y+z) \big) - \bar{f}_H\big( (s+h)y+sz\big) - \bar{f}_H\big( (s+h)z+sy\big) + \bar{f}_H \big( s(y+z)\big)  \big| \\
 &\qquad  \leq \big| \bar{f}_H\big( (s+h)(y+z) \big) - \bar{f}_H\big( (s+h)z+sy\big)
\big| + \big|  \bar{f}_H\big( (s+h)y+sz\big) + \bar{f}_H\big( s(y+z)\big)  \big|.
\end{align*}
The first difference is estimated by {\bf (h4)}(1) and the second one is dealt with in a similar way. Therefore, \eqref{h42} is satisfied with
$\bar{b} < (2\bar{\kappa}) \wedge \min( 2 H_i-1 ; i=1, \cdots, d)$.
This completes the proof. 

 \noindent{\em Conclusion.} 
 Let $d=2,3$ and  $H=(H_i)_{1\le i\le d}$, where 
$\frac{1}{2} < H_i <1$. 
For spatially homogeneous Gaussian noises with covariance function given by \eqref{s3.1} with $\Lambda (dx) = f_H(x)\ dx$, 
the parameters $\nu_1$, $\nu_2$ in \eqref{numbers} are 
\beqn
\nu_1=\nu_2=\min\left(\gamma_1, \gamma_2, \bar{\kappa}, \min\left(H_i-1/2; i=1, \cdots, d\right) \right), \ 
 {\text{with}}\ \bar{\kappa} =\sum_{i=1}^d H_i -(d-1)>0.
\eeqn

As a consequence, from \eqref{s3-3.72bis} we deduce that  almost all sample paths of the solution to \eqref{n3} 
are  locally H\"older continuous, jointly in $(t,x)$, with exponent 
\beqn
\theta\in\left]0, \min\left(\gamma_1, \gamma_2, \bar\kappa, \min\left(H_i-1/2; i=1, \cdots, d\right) \right)\right[. 
\eeqn
For $d=3$, this is \cite[Theorem 6.1]{HHN}. However, following \cite[Theorem 6.2]{HHN}, the critical exponent must be $\min(\gamma_1,\gamma_2, \bar\kappa)$, and therefore the result is not optimal.

\bigskip

 
 \section{Appendix: some elements on stochastic wave equations}
 \label{s-app}
 In this section, we give some basic elements relative to the stochastic wave equations considered in this article.
 
 Throughout the section, $(u(t,x),\ (t,x)\in[0,T]\times \RR^d)$ denotes the random field solution to the stochastic wave equation in each one of the settings:
 \begin{itemize}
 \item {{\bf (s1)}}\ $d=1$, $W$ is a space-time white noise, and the conditions of Proposition \ref{Proposition II.3} hold;
 \item {{\bf (s2)}}\  $d=2,3$, $W$ is a noise white in time and coloured in space, and the hypotheses of Proposition \ref{randomfield-multi} are satisfied.
 \end{itemize}
 
 For any bounded Borel set $\mathcal{O}\subset \RR^d$ and $\varepsilon_0$,  $\mathcal{O}^{(\varepsilon_0)}$ denotes the closed $\varepsilon_0$-neighborhood of $\mathcal{O}$, that is, 
 \beqn
 \mathcal{O}^{(\varepsilon_0)} = \{z\in\RR^d, \ d(x,\mathcal{O})\le \varepsilon_0\}.
 \eeqn

  \subsection{Propagation of support}
  \label{s-app-s1}
  For the homogeneous (deterministic) wave equation there is a well-known {\em compact support property} saying that if the initial conditions have compact support then the solution in the classical sense has also compact support. The proposition below tells us that, under suitable conditions, this property extends to the stochastic wave equation in dimensions $d\in\{1,2,3\}$. As for the deterministic case, this relies on the fact that the support of the fundamental solution of the wave equation $G(t, \cdot)$, $t\in [0,T]$, is included in the closed ball $\overline{B(0; t)}$.

  \begin{prop}
 \label{app-p1} 
 Consider the two cases {\bf (s1)} and {\bf (s2)} described above. Assume that 
 \begin{itemize}
 \item{(i)} the initial conditions $u_0$, $v_0$ are functions with compact support $K\subset \RR^d$;
\item{(ii)} the coefficients $b$ and $\sigma$ satisfy $b(0)=\sigma(0)=0$.
\end{itemize}
Then, for any $t\in[0,T]$, 
\beq
\label{app.1}
u(t,x) = 0, \quad {\text{for any}}\quad x\notin K^{(t)}.
\eeq
Hence, the support of the sample paths of the solution $(u(t,x),\ (t,x)\in[0,T]\times \RR^d)$ is included in $[0,T]\times K^{(T)}$.
\end{prop}
 \begin{proof}
 We will extend the arguments in \cite[p. 925]{MiMo} in dimension $d=2$ to any $d\in\{1,2,3\}$.
  
 First, notice that since the mapping $[0,T]\ni t\mapsto K^{(t)}$ is increasing, the last statement is an immediate consequence of \eqref{app.1}. 
 
 Next, we prove that \eqref{app.1} holds with $u(t,x)$, replaced by $u^n(t,x)$--the $n$-th Picard iteration of $u$ defined by
 \beq
 \label{app.2}
 \begin{cases}
 u^0(t,x)&=\quad I_0(t,x),\\
 u^n(t,x)&=\quad I_0(t,x) + \int_0^t ds \int_{\RR^d} dy\ G(t-s, x-y)b(u^{n-1}(s,y))\\ 
 &\quad \quad\quad\quad  + \int_0^t\int_{\RR^d} G(t-s,x-y) \sigma(u^{n-1}(s,y))\ W(ds,dy),\ n\ge 0.
 \end{cases}
 \eeq
Indeed, fix $t\in[0,T]$ and let $n=0$. If $x\notin K^{(t)}$, we have  $|x-y|>t$ for any $y\in K$. Therefore, the integrals defining $I_0(t,x)$ (see \eqref{decom}) vanish, because of the above mentioned property on the support of $G(t, \cdot)$.

In the next induction step, we will make use of the following fact:
\smallskip

{\em {\bf (PS)} Let  $x\notin K^{(t)}$. Then for all $s\in[0,t]$ and all $y\in\RR^d$ such that $|x-y|\le t-s$, we have $y\notin K^{(s)}$}.
\smallskip

Indeed, if $x\notin K^{(t)}$ then for any $z\in\RR^d$, $|z-x|>t$. By the triangle inequality, this implies
$|z-y|\ge |z-x| - |x-y|>t-(t-s) = s$. Thus, $y\notin K^{(s)}$.

Assume that \eqref{app.1} holds with $u(t,x)$ replaced by $u^l(t,x)$, $l=0,\ldots,n-1$. We observe that the integrands in \eqref{app.2} (with $n:=n-1$) vanish if $|x-y|>t-s$. If on the contrary, $x-y\le t-s$, from {\bf (PS)} and the induction assumption, we have that $u^{(n-1)}(s,y)=0$. Hence, $b(u^{n-1}(s,y))= \sigma(u^{n-1}(s,y))=0$, by assumption (ii), which implies \eqref{app.1} with $u(t,x)$ replaced by $u^n(t,x)$.

In Propositions \ref{Proposition II.3} and \ref{randomfield-multi} the random field solutions $(u(t,x),\ (t,x)\in [0,T]\times\RR^d)$ are obtained as limits of the Picard iterations \eqref{app.2}. Thus, the {\em propagation of support property} \eqref{app.1} holds.
 \end{proof}
 
 \subsection{The solution of the stochastic wave equation restricted to a bounded domain}
 \label{s-app-s2}

 Let $d\in\{1,2,3\}$ and $T>0$.  Following the setting of \cite[Chapter 4]{DSS-Memoirs}, we consider a bounded domain $D\subset \RR^d$ and the associated past light cone
 \beqn
 D^{(T-t)} = \{z\in \RR^d: d(z,D) \le T-t\},\quad t\in[0,T].
 \eeqn
 For any $(t,x)\in[0,T]\times \RR^d$, set
\begin{align}
\label{inD}
u(t,x)1_{D^{(T-t)}}(x) & =1_{D^{(T-t)}}(x) \left([G(t) \ast v_0](x)  + \frac{\partial}{\partial t} \big[ G(t)\ast u_0\big] (x)\right)\notag\\
 &+ 1_{D^{(T-t)}}(x) \int_0^t ds\int_{\RR^d} dy\  G(t-s,x-y) b(u(s,y)1_{D^{(T-s)}}(y)) \notag\\
& +1_{D^{(T-t)}}(x) \int_0^t \int_{\RR^d} G(t-s,x-y) \sigma(u(s,y)1_{D^{(T-s)}}(y))\ W(ds,dy).
\end{align}
By the triangle inequality, if 
$x\in D^{(T-t)}$ and $x-y\in \overline{B(0;t-s)}$ then $y\in D^{(T-s)}$ for all  $0\le s\le t\le T$
(compare with property {\bf (PS)} in Section \ref{s-app-s1}). Hence, \eqref{inD} is coherent in the following sense: when $x\notin D^{(T-t)}$,  it is a trivial equation, while when $x\in D^{(T-t)}$, the stochastic process $\left(u(t,x)1_{D^{(T-t)}}(x),\ (t,x)\in [0,T]\times \RR^d\right)$ satisfies \eqref{n3} (with $(t,x)\in[0,T]\times D^{(T-t)}$).
In \cite{DSS-Memoirs}, this process is called a solution to the wave equation ``in D''.

Observe that $[0,T]\ni t\mapsto D^{(T-t)}$ is decreasing and  $\cap_{t\in[0,T]}D^{(T-t)}=D$. 

\begin{prop}
 \label{app-p2} 
 Consider the two cases {\bf (s1)}, {\bf (s2)} described above.
 
 Let $u=(u(t,x),\ (t,x)\in [0,T]\times\RR^d)$ be the respective random field solutions given in Propositions \ref{Proposition II.3} and \ref{randomfield-multi}.

Let $D\subset\RR^d$ be a bounded domain. Then, almost surely, $u(t,x) = u(t,x)1_{D^{(T-t)}}(x)$ for all $(t,x)\in[0,T]\times D$, where $\left(u(t,x)1_{D^{(T-t)}}(x),\ (t,x)\in [0,T]\times \RR^d\right)$ is the random field solution to \eqref{inD}.

Therefore, the support of the sample paths of the random field $u=(u(t,x),\ (t,x)\in [0,T]\times D)$ is included in $[0,T]\times D^{(T)}$. 
\end{prop}
  \begin{proof}
 For $d=3$ and covariance densities of the noise belonging to a class that include Riesz kernels, the existence and uniqueness of a random field solution to \eqref{inD} is established in \cite[Chapter 4, ]{DSS-Memoirs} (see, in particular Proposition 4.3 and Theorem 4.6 there). With covariance densities satisfying the conditions ${\bf (h0)}$ and ${\bf (h1)}$ of Section \ref{s3-0}, the results are still valid (\cite{nq-2007}, \cite{dalang-quer-2011}, \cite{HHN}). 
With similar but simpler arguments, those results can be extended to $d=1$ (with space time white noise) and to $d=2$ with covariance densities satisfying the conditions ${\bf (h0)}$ and ${\bf (h1)}$ of Section \ref{s3-0}.

Because of the uniqueness of solution assertions, for all $(t,x)\in[0,T]\times D$, $u(t,x) = u(t,x)1_{D^{(T-t)}}(x)$ a.s. Actually, since the sample paths of both processes are continuous (and even locally H\"older continuous), the processes $(u(t,x),\ (t,x)\in[0,T]\times D)$ and $(u(t,x)1_{D^{(T-t)}}(x),\ (t,x)\in[0,T]\times D)$ are indistinguishable.
  \end{proof}
  


\section*{ Acknowledgements}
\noindent Annie Millet's research  has been conducted within the FP2M federation (CNRS FR 2036).\\
Marta Sanz-Sol\'e was supported by the grant MTM2015-65092-P and the research project MdM-2014-0445
from the Ministerio de Ciencia, Innovaci\'on y Universidades (Spain).\\  
Annie Millet wants to thank the University of Barcelona for the hospitality during her stay. \\
The authors gratefully acknowledge the support and excellent conditions provided by the {\em Research in Pairs} programme of the Mathematisches Forschunginstitut Oberwolfach in June 2018.\\

\end{document}